\documentclass[a4paper]{amsart}
\usepackage[mathscr]{eucal}
\usepackage{amssymb,amsfonts,amsmath,amsthm}
\theoremstyle{plain}
\newtheorem{theorem}{Theorem}[section]
\newtheorem{lemma}[theorem]{Lemma}
\newtheorem{sublemma}[theorem]{Sublemma}
\newtheorem{proposition}[theorem]{Proposition}
\newtheorem{corollary}[theorem]{Corollary}
\newtheorem{definition}[theorem]{Definition}
\newtheorem{remark}[theorem]{Remark}
\newtheorem{remarks}[theorem]{Remarks}
\newtheorem{notation}[theorem]{Notation}
\newtheorem{Fact}[theorem]{Fact}
\newtheorem*{claim}{Claim}
\newtheorem*{thrm}{Theorem}
\newtheorem*{lemma*}{Lemma}
\newtheorem*{prop}{Proposition}
\newcommand{\N}{\mathbb{N}}

\newcommand{\Q}{\mathbb{Q}}
\newcommand{\mc}{\mathcal}
\newcommand{\e}{\epsilon}
\newcommand{\al}{\alpha}
\newcommand{\mca}[1][A]{\mathcal{#1}}

\newcommand{\fa}{f_{\alpha}}

\newcommand{\fb}{f_{\beta}}
\newcommand{\fg}{f_{\gamma}}
\newcommand{\fd}{f_{\delta}}

\newcommand{\ga}{g_{\alpha}}
\newcommand{\gb}{g_{\beta}}

\newcommand{\norm}[1][\cdot]{\lVert #1\rVert}
\newcommand{\hxn}[1][x]{(\widehat{#1}_{n})_{n\in\N}}
\newcommand{\hyn}{(\hat{y}_{n})_{n\in\N}}
\newcommand{\xn}[1][x]{(#1_{n})_{n\in\N}}

\DeclareMathOperator{\minsupp}{minsupp}
\DeclareMathOperator{\maxsupp}{maxsupp}
\DeclareMathOperator{\supp}{supp}
\DeclareMathOperator{\conv}{conv}
\DeclareMathOperator{\ran}{range}
\DeclareMathOperator{\dist}{dist}
\DeclareMathOperator{\ind}{ind}

\DeclareMathOperator{\inde}{ind}

\begin{document}
\title{On the hereditary proximity to $\ell_1$ }
\author{S.A. Argyros}

\address{Department of Mathematics, National Technical University of Athens, Athens 15780, Greece}
\email{sargyros@math.ntua.gr}
\author{A. Manoussakis}
\address{Department of Mathematics, University of Aegean, Karlovasi, Greece}
\email{amanouss@aegean.gr}
\author{A.M. Pelczar}
\address{Institute of Mathematics, Jagiellonian University, {\L}ojasiewicza 6, 30-348 Krak\'ow, Poland}
\email{anna.pelczar@im.uj.edu.pl}
\begin{abstract}
In the first part of the paper we present and discuss
concepts of local and asymptotic hereditary proximity to
$\ell_1$. The second part is devoted to a complete separation of
the hereditary local proximity to $\ell_1$ from the asymptotic one. More
precisely for every countable ordinal $\xi$ we construct a
separable reflexive space $\mathfrak{X}_{\xi}$ such that every
infinite dimensional subspace of it has  Bourgain $\ell_1$-index
greater than $\omega^\xi$ and the space itself has no $\ell_1$-spreading model.  We also present a reflexive HI space admitting
no $\ell_{p}$ as a spreading model.
\end{abstract}
\maketitle

\section{Introduction}
Concepts of proximity to a classical $\ell_p$ (or $c_0$) space play a significant
role to the understanding of the structure of a Banach space. They are categorized as
follows:

The first one is the global proximity to $\ell_p$ which simply means that $\ell_{p}$  is
isomorphic to a subspace  $Y$ of $X$. The local proximity which occurs more frequently,
due to  J.L. Krivine's theorem \cite{K},  is measured through the Bourgain
$\ell_{p}$-index \cite{B}. The last concept is the asymptotic proximity that varies from
A. Brunel-L. Sucheston $\ell_p$-spreading models, \cite{BS}, to the asymptotic $\ell_p$ spaces.
The latter class of Banach spaces appeared after B.S. Tsirelson space \cite{T}  that
answered in negative the famous Banach's problem by showing that global proximity to some
$\ell_p$ is not always possible.

It is easy to see that the global proximity to $\ell_p$ is the strongest one followed by
the asymptotic one. The local proximity  is the weakest among them. It is also known that
the three classes are separated for each $\ell_{p}$. Namely there are spaces with
arbitrarily large local proximity to $\ell_p$ but no asymptotic one and similarly for the
asymptotic and the global ones. The present paper is mainly devoted  to  the separation
of the local and asymptotic proximity to $\ell_{1}$ when the first one is hereditarily
large. In particular our work is motivated by a result of the third named author stated
as follows.
\begin{thrm} (\cite{P}) Let $X$ be a separable Banach space and $\xi$ be a
countable ordinal. If $X$ is boundedly distortable  and has hereditary Bourgain
$\ell_{1}$-index greater than  $\omega^{\xi}$  then it is saturated by asymptotic
$\ell_{1}^{\xi}$ spaces.
\end{thrm}
Let's recall that the  hereditary Bourgain $\ell_{p}$-index of a Banach space  $X$ is
the minimum  of   Bourgain $\ell_{p}$- index of its subspaces.  In the sequel by the
$\ell_{p}$-index we will mean the Bourgain $\ell_{p}$-index.

In view of the above theorem it is natural to ask how critical is
the bounded distortion of $X$ for the final conclusion. It is also worth  adding that heredity assumptions for the local proximity to $\ell_1$  could yield large asymptotic one. In this direction we prove the following.
\begin{prop} Let $(e_n)$ be a Schauder basis of a  Banach space $X$ such that
the Bourgain $\ell_{1}$-tree supported by any subsequence of
$\xn[e]$ has order greater than $\omega^{\xi}$.
Then  there exists a subsequence   generating an
$\ell_{1}^{\xi}$-spreading model.
\end{prop}
Our aim is to show that large hereditary $\ell_{1}$-structure in a Banach space $X$ does
not imply in general any asymptotic one. More precisely  the main goal at the present
paper is to prove the next
\begin{thrm}
For every countable ordinal $\xi$ there exists a separable reflexive space
$\mathfrak{X}_{\xi}$  with the hereditary  $\ell_{1}$-index greater than $\omega^{\xi}$
such that $\mathfrak{X}_{\xi}$ does not admit an $\ell_{1}$-spreading model. Moreover the
dual $\mathfrak{X}_{\xi}^{*}$  has  hereditary $c_0$-index greater than $\omega^{\xi}$
and does not admit $c_0$ as a spreading model.
\end{thrm}
Our approach in constructing the space $\mathfrak{X}_{\xi}$ is based on mixed Tsirelson
extensions of a ground set $G_{\xi}$ using the method of  attractors. The latter appeared
in \cite{AT} and is extensively used in \cite{AAT}. The method of  attractors has two
separated steps that we are about to describe.

In the first step an auxiliary space is constructed  that partially solves the required problem. In our case for a given countable ordinal $\xi$ we construct  a ground set  $G_{\xi}$ such that the resulting space  $X_{G_{\xi}}$  is reflexive, has a Schauder basis $\xn[e]$ and satisfies the following properties.

i) The space $X_{G_{\xi}}$  does not have an $\ell_1$-spreading model.

ii)   For every $L\in [\N]$ the  $\ell_1$-index  of the subspace $\langle (e_n)_{n\in L}\rangle$ is greater than $\omega^{\xi}$.

The above stated proposition together with property i) of the space $X_{G_{\xi}}$ indicate that property ii) requires special attention. Namely for every subsequence $(e_n)_{n\in L}$ of the basis  the  $\ell_{1}$-index of the corresponding subspace should be greater than $\omega^{\xi}$ but the  $\ell_{1}$-tree supported by that  subsequence  should be of small height.  To achieve those requirements we include  in the set $G_{\xi}$ a $c_0$-tree generated by functionals of the form $m_{2j-1}^{-2}\sum_{k\in B}e_{k}^{*}$, $\# B\leq n_{2j-1}$ and the tree order induced by an appropriate coding function.

The proof that  $X_{G_{\xi}}$ does not have an $\ell_1$-spreading model  heavily relies  on combinatorics in particular on  Ramsey theory.

The next step is to define the space $X_{\xi}$. For this purpose
we proceed to a mixed Tsirelson extension  $K_{\xi}$ of the
set $G_{\xi}$  including also the attracting functionals.
The space   $X_{\xi}$ is the completion of $(c_{00}(\N),\norm_{K_{\xi}})$
where $\norm_{K_{\xi}}$ is the norm induced by the set $K_{\xi}$.

For our approach the  attracting  functionals are of significant importance. They are  the tool for transferring   to every  block subspace of  $X_{\xi}^{*}$ a large $c_0$-subtree from the set $G_{\xi}$.
Thus  we are able to show that the hereditary $c_0$ index of $X_{\xi}^{*}$ is greater than $\omega^{\xi}$.
This  does not yield that the corresponding $\ell_{1}$-index of $X_{\xi}$ is also greater than  $\omega^{\xi}$.  Therefore we need one more step, namely the desired space $\mathfrak{X}_{\xi}$ is
a quotient $X_{\xi}/X_{\xi}^{L}$ where
$X_{\xi}^{L}=\overline{\langle (e_n)_{n\in L}\rangle}$ with
$L$ a suitable  subset of $\N$.
Both spaces $X_{\xi}, \mathfrak{X}_{\xi}$ do not have an $\ell_{1}$
spreading model.

We proceed now to  describe how the paper is organized.

Section \ref{sec2} is devoted to preliminary notions and results.

In Section \ref{sec3} we discuss different concepts of
local and asymptotic proximity to $\ell_{1}$.
More  precisely we introduce the hereditary strategic
$\ell_{1}$-index of a Banach space $X$ denoted as $I_{hs}(X)$.
This index is formulated in terms of Gowers  game and Schreier
families  $\mc{S}_{\xi}$, $\xi<\omega_{1}$.
A non-hereditary version of the aforementioned index, related to the asymptotic structures
defined in \cite{MMT}, is also presented.
The  $I_{hs}(X)$ is essentially equivalent  to the hereditary block
 $\ell_{1}$-index, denoted as $I_{hb}(X)$,
in the following manner:  First we show  that for every countable ordinal $\xi$ if $I_{hs}(X)>\xi$ then $I_{hb}(X)>\omega^{\xi}$. In the opposite direction Gowers dichotomy \cite{G} yields that if  $I_{hb}(X)>\omega^{\xi}$ then there exists a closed subspace $Y$ of $X$ with $I_{hs}(Y)>\xi.
$
Furthermore we  examine the relation between the concepts of non-hereditary $\ell_{1}$-proximity.
We finish Section \ref{sec3} by showing strong correlation between the above notions in
a subsequence setting, i.e. with block sequences replaced by subsequences of a fixed basis. The main result in this part is Proposition \ref{p2.7} which has been mentioned before.

With Section \ref{sec4}  we start dealing with our final goal namely the space
$\mathfrak{X}_{\xi}$. Thus in this section for a given countable ordinal $\xi$ we define
a ground set $G_{\xi}$ which serves as a norming set for the  aforementioned auxiliary
space $X_{G_{\xi}}$. The set $G_{\xi}$ includes a rich $c_0-$tree of height greater than
$\omega^{\xi}$. This tree is defined as follows.   First $G_{\xi}$ contains all
functionals of the form $m_{2j-1}^{-2}\sum_{k\in B}e_{k}^{*}$ with $\# B \leq n_{2j-1}$
where $(n_j)_{j\in\N}$, $(m_{j})_{j\in\N}$ are  appropriate increasing sequences of
natural numbers. Next using a coding function $\sigma$, in a similar manner as in the
classical work of B. Maurey and H. Rosenthal  \cite{MR}, we define a well-founded tree of
$\sigma$-special sequences $(f_{1},\dots,f_{d})$ ordered by the initial segment
inclusion. Each $f_{l}$ is of the form  $m_{2j_{l}-1}^{-2}\sum_{k\in B_{l}}e_{k}^{*}$,
$\#B_{l}\leq n_{2j_{l}-1}$. Then we include into $G_{\xi}$  all $\sum_{l=1}^{d}\pm f_{l}$
where $(f_{l})_{l=1}^{d}$ is a $\sigma$-special sequence.

A second ingredient of $G_{\xi}$ is  coming from James tree-like spaces \cite{Ja} (see
also \cite{Kal}, Chapter 13 or \cite{AAT}). Namely we include all rational
$\ell_2$-convex combinations of  functionals $\sum_{l=1}^{d}\pm f_{l}$ with disjoint
weights.  We denote by $X_{G_{\xi}}=\overline{( c_{00}(\N),\norm_{G_{\xi}})}$
the space with the set $G_{\xi}$ as the norming set.

In Section \ref{sec5} we present the basic properties of the space $X_{G_{\xi}}$. Namely it is reflexive, has a Schauder basis $\xn[e]$, is $\ell_2$-saturated and for every $L\in [\N]$ the  $\ell_{1}$-index of the subspace $\langle (e_n)_{n\in L}\rangle $ is greater than $\omega^{\xi}$.
Also we show a dual result. Namely for every $(e_n^{*})_{n\in L}$  in the dual $X_{G_{\xi}}^{*}$ the  $c_{0}$-index of  $\langle (e_{n}^{*})_{n\in L}\rangle $ is similarly large.

In Section \ref{sec6} it is shown that the space $X_{G_\xi}$ does not have an
$\ell_{1}$-spreading model. This is the most involved part of the study of $X_{G_{\xi}}$.
This result is critical as with some small additional effort yields that the ultimate
space $\mathfrak{X}_{\xi}$ shares the same property. The difference of the space
$X_{G_{\xi}}$ from the earlier examples of spaces with no $\ell_{p}$-spreading model,
(i.e. \cite{OSsm})  is that $X_{G_{\xi}}$ has a rich local $\ell_{1}$ structure.
Therefore the proof requires new tools which are of combinatorial nature. The main part
of the proof is given by Proposition \ref{nol1}.

In Section \ref{sec7} we define the set $K_{\xi}$ which is the norming set of the intermediate space $X_{\xi}$.  The ingredients of the set $K_{\xi}$ are the following:

i)  The  set $G_{\xi}$  is included into $K_{\xi}$. In particular $K_{\xi}$ is a  mixed Tsirelson extension of $G_{\xi}$.

ii) For the aforementioned  sequences $(m_j)_{j}, (n_{j})_{j}$  the set  $K_{\xi}$  is closed under the  even operations. Namely  it contains all  $f=m_{2j}^{-1}\sum_{i=1}^{n_{2j}}f_{i}$ with $f_{1}<\dots<f_{n_{2j}}$ in
$K_{\xi}$.

iii) For the odd operations
$(\mc{A}_{n_{2j-1}},m_{2j-1}^{-1})$ $K_{\xi}$
includes the attracting functionals.
Those  are functionals of the form
$f=m_{2j-1}^{-1}\sum_{i=1}^{n_{2j-1}/2}(f_{2i-1}+e_{l_{2i}}^{*})$
with $f_{1}<e^{*}_{l_2}<f_{3}<e^{*}_{l_4}<\dots$
and the whole sequence is selected with the aid of a coding function.

iv)  The set $K_{\xi}$ contains all rational $\ell_{2}$-convex combinations of its
weighted functionals with disjoint weights and also  it is a rationally convex set.

The set $K_{\xi}$ induces a norm on $c_{00}(\N)$  and the space
$X_{\xi}$ is its completion.
The attracting functionals are responsible for carrying
structure from the set $G_{\xi}$ to  block subspaces of
$X_{\xi}^{*}$. To make more transparent the role of the attracting functionals
let us first notice, that each attracting functional $f$ consists of two parts.
Namely $f=g_1+g_{2}$  where
$g_{1}=m_{2j-1}^{-1}(f_{1}+f_{3}+\dots)$ and
$g_{2}=m_{2j-1}^{-1}(e^{*}_{l_{2}}+e^{*}_{l_{4}}+\dots)$.
The key result for applying the method of attractors is the following
\begin{lemma*}
For every block subspace $Y$ of $X_{\xi}^{*}$, for every $j_0\in\N$  and every $\e>0$ there exists an attracting  functional $f=m_{2j_{0}-1}^{-1}\sum_{i=1}^{n_{2j_{0}-1}/2}(f_{2i-1}+e_{l_{2i}}^{*})$ such that writing $f=g_{1}+g_{2}$ we have $\dist(g_{1},Y)<\e$ and
\begin{equation}\label{ei1} \norm[g_{1}]\geq cm_{2j_0-1}
\end{equation}
where $c\in (0,1)$ is a universal constant.
\end{lemma*}
Granting \eqref{ei1} we proceed as follows. We observe that $\norm[f]\leq 1$ as $f\in
K_{\xi}$ and also \eqref{ei1} yields $\norm[m_{2j_0-1}^{-1}g_{1}]\geq c$. Therefore the
functional $m_{2j_0-1}^{-1}g_{1}$ has norm bounded from below by $c$ and also
\begin{equation}
\label{ei2} \norm[m_{2j_0-1}^{-1}g_{1}-(-m_{2j_0-1}^{-1}g_{2})]\leq m_{2j_0-1}^{-1}.
\end{equation}
We recall that $(-m_{2j_0-1}^{-1}g_{2})\in G_{\xi}$ and in fact is a component of the
$c_0$-tree structure included in $G_{\xi}$. Therefore inequality \eqref{ei2} permit us to
transfer into an arbitrary block subspace of $X_{\xi}^{*}$ a $c_0-$tree structure of
height greater than $\omega^{\xi}$.  It is worth pointing out that the method of attractors offers a
considerable reduction to the complexity of the proofs for properties of the ultimate
space. For example in the case of the space $\mathfrak{X}_{\xi}$  the proof that the
space does not have an $\ell_{1}$-spreading model is essentially given for the auxiliary
space $X_{G_{\xi}}$ where the norming set $G_{\xi}$ is simpler than the set  $K_{\xi}$.

Sections \ref{sec8},\ref{sec9},\ref{sec10} are rather technical and include the necessary
estimations for proving the aforementioned  inequality \eqref{ei1}. This part is closely
related to the well known estimations in mixed Tsirelson and Hereditarily Indecomposable
spaces.  The additional complexity of this part, compared to the previous similar results,
arises from property $iv)$ of the norming set $K_{\xi}$ and the local $c_{0}-$structure
of the set $G_{\xi}$. A consequence of the new estimations  is the
following theorem.
\begin{thrm} There exists
a reflexive HI space $X$ admitting no $\ell_p$, $1\leq p <\infty$, or $c_0$ as a spreading model.
\end{thrm}
This result is presented in Section  \ref{sec11} which also includes  a general result
concerning spaces with no $\ell_{p}$ (or $c_0$) as a spreading model (Theorem
\ref{t11.3}). E. Odell and Th. Schlumprecht, \cite{OSsm}, have presented the first
example of a Banach space with no $\ell_p$ (or $c_0$) as a spreading model. The
aforementioned result provides an alternative proof of the latter property of E. Odell
and Th. Schlumprecht example and also yields that the space $X_{\xi}$ does not have an
$\ell_{1}$-spreading model.

In Section \ref{sec12} we show   that the space  $X_{\xi}^{*}$  has  hereditary
$c_0$-index  greater than $\omega^{\xi}$. This result does not allow us to conclude that
the hereditary $\ell_{1}$-index of $X_{\xi}$ is also greater than $\omega^{\xi}$. It is
worth pointing out  that  in general the existence of local or asymptotic $c_0$-structure
in $X^{*}$ yields that the corresponding $\ell_{1}$ will occur on $X$.  For example it is
easy to see that  if $X^{*}$  has a $c_0$-spreading model then  $X$ will have a
corresponding $\ell_{1}$. This fact seems not to remain valid for the hereditary local or
asymptotic structure as $X_{\xi}$ indicates.

In Section \ref{sec13}  we make the final step in defining the space
$\mathfrak{X}_{\xi}=X_{\xi}/ X_{\xi}^{L}$ where $X_{\xi}^{L}=\langle(e_n)_{n\in
L}\rangle$ with a suitable subset $L$ of $\N$. The space  $\mathfrak{X}_{\xi}$ has  a
Schauder basis and $I_{hb}(\mathfrak{X}_{\xi})>\omega^{\xi}$. The space
$\mathfrak{X}_{\xi}$  does not have an $\ell_{1}$-spreading model since it is quotient of
$X_{\xi}$ which satisfies the same property. Actually both spaces $X_{\xi}$ and
$\mathfrak{X}_{\xi}$ does not have any $\ell_{p}$ as a spreading model.

\section{Preliminaries}\label{sec2}
We start by recalling  some basic definitions and standard
notation. Let $X$ be a Banach space with a basis $(e_i)$. Given
any basic sequence $(x_n)$ by $\langle (x_n)_n\rangle$ we denote the closed vector
subspace spanned by $(x_n)$. The \textit{support} of a vector
$x=\sum_i x_i e_i$ is the set $\supp x =\{ i\in\N :\ x_i\neq 0\}$,
we define the $\ran x$ of a vector $x\in X$ as the smallest
interval in $\N$ containing support of $x$. Given any $x=\sum_i
x_ie_i$ and finite $E\subset\N$ put $Ex=\sum_{i\in E}x_ie_i$. We
write $x<y$ for vectors $x,y\in X$, if $\maxsupp (x)<\minsupp(
y)$. A \textit{block sequence} is any sequence $(x_i)\subset X$
satisfying $x_{1}<x_{2}<\dots$, a \textit{block subspace} of $X$ -
any closed subspace spanned by an infinite block sequence. A
\textit{tail subspace} of $X$ is any subspace of the form
$\langle (e_n)_{n\geq n_0}\rangle$ for some $n_0\in\N$.

Let  $(x_n)_{n\in\N}$ be a sequence and $(\e_{n})_{n\in\N}$ be a sequence of positive
numbers. We say that $(y_n)_{n\in\N}$  is $(\e_n)$-close  to $(x_n)_{n\in\N}$ if
$\norm[x_n-y_n]<\e_n$ for every $n\in\N$.

Given infinite $M\subset\N$ by $[M]$ we denote the family of all infinite subsets of $M$,
by $[M]^{<\infty}$ - the family of all finite subsets of $M$. By $[M]^{n}$, $n\in\N$, we denote all finite subsets of $M$ of cardinality $n$. A family $\mc{F}$ of finite
subsets of $\N$ is \textit{regular}, if it is \textit{hereditary}, i.e. for any $G\subset
F$, $F\in \mc{F}$ also $G\in \mc{F}$, \textit{spreading}, i.e. for any integers
$n_1<\dots<n_k$ and $m_1<\dots<m_k$ with $n_i\leq m_i$, $i=1,\dots, k$, if
$(n_1,\dots,n_k)\in \mc{F}$ then also $(m_1,\dots,m_k)\in \mc{F}$, and \textit{compact}
in the product topology of $2^\N$.

The families $\mc{A}_n$, $n\in\N$, are defined by the following formula:
$$
\mc{A}_n=\{F\subset \N:\ \# F\leq n\}, \  \ n\in\N.
$$
Define the \textit{generalized Schreier families} $(\mc{S}_{\xi})_{\xi<\omega_1}$ of
finite subsets of $\N$ by the transfinite induction \cite{AA}:
$$\mc{S}_0=\{\{ n\}:\ n\in\N\}\cup\{\emptyset\}$$
$$\mc{S}_{\xi+1}=\mc{S}_{1}[\mc{S}_{\xi}]=\{ F_1\cup\dots\cup F_m:\ m\in\N,\ F_1,\dots,F_m\in \mc{S}_\xi,\ m\leq F_1<\dots<F_m\}$$
for any $\xi<\omega_1$. If $\xi$ is a limit ordinal, choose
$\xi_n\nearrow \xi$ and set
$$\mc{S}_{\xi}=\{F:\ F\in \mc{S}_{\xi_n}\ \mathrm{and}\ n\leq F\ \mathrm{for\ some}\ n\in\N\}.
$$
It is well known that the families $\mc{A}_n$, $n\in\N$, $\mc{S}_{\xi}$,
$\xi<\omega_{1}$, are regular families of finite subsets of $\N$.

Let us recall that a set $F\in \mc{S}_\xi$ is called a maximal set if  there is no $G\in
\mc{S}_\xi$ such that $F\subseteq G$.  In \cite{GA} it is proved that $F$ is
$\mathcal{S}_{\xi}$-maximal if and only if there is no $k\in\N$ with $F<k$ and
$F\cup\{k\}\in\mc{S}_{\xi}$.
\begin{definition} Let $\mc{F}$ be one of the families  $\mc{A}_n$, $n\in\N$, $\mc{S}_{\xi}$, $\xi<\omega_{1}$, and $\theta\in (0,1)$.

1) A finite sequence $(f_{1},\dots, f_{k})$ in $c_{00}(\N)$ is said to be $\mc{F}$-admissible if
$$
\mbox{$\supp(f_1)<\dots<\supp(f_k)$ and $\{\min (f_{1}),\dots,\min (f_{k})\}\in\mc{F}$}
$$
2) The $(\mc{F},\theta)$-operation on $c_{00}(\N)$ is the operation which assigns to each $\mc{F}$-admissible sequence $f_{1}<\dots<f_{d}$ the vector $\theta(f_{1}+\dots+f_{d}).$
\end{definition}
Throughout this paper by a \textit{tree on a set} $X$ we mean a subset $\mc{T}$ of
$\bigcup_{n=1}^\infty X^n$ such that $(x_1,\dots,x_k)\in \mc{T}$ whenever
$(x_1,\dots,x_k,x_{k+1})\in \mc{T}$, $k\in\N$, ordered by the initial segment inclusion.
A tree $ \mc{T}$ is \textit{well-founded}, if there is no infinite sequence $(x_i)\subset
X$ with $(x_1,\dots,x_k)\in \mc{T}$ for any $k\in\N$. Given a tree $ \mc{T}$ on $X$ put
$$
D( \mc{T})=\{(x_1,\dots,x_k):\ (x_1,\dots,x_k,x)\in \mc{T} \ \ \mathrm{for}\ \ \mathrm{some}\ \ x\in X\}.
$$
Inductively define trees $D^\al( \mc{T})$: $D^0( \mc{T})= \mc{T}$,
$D^{\al+1}(\mc{T})=D(D^\al( \mc{T}))$ for $\al$ ordinal and $D^\al(
\mc{T})=\bigcap_{\xi<\al}D^\xi( \mc{T})$ for $\al$ limit ordinal. The
\textit{order} of a well-founded tree $ \mc{T}$ is given by $o(
\mc{T})=\inf\{\al:\ D^\al( \mc{T})=\emptyset \}$.

Let $\mc{F}$ be a countable family  of  finite subset of $\N$
endowed with the topology of the pointwise topology.  For
$\al<\omega_{1}$, we set $\mc{F}^{\al+1}=\{F\in \mc{F}: F
\mbox{\,- a limit point of } \mc{F}^{\al}\}$ and for $\al$ limit
ordinal $\mc{F}^{\al}=\cap_{\beta<\al}\mc{F}^{\beta}$. The
Cantor-Bendixson index of $\mc{F}$, denoted by $CB(\mc{F})$, is
defined as the least $\al$ for which $\mc{F}^{\al}=\emptyset$.

Let $\mc{\mc{T}}$ be a countable tree on $\N$. Then $\mc{T}$
defines the family
$$
\mc{F}_{\mc{T}}=\{F\in [\N]^{<\infty}: \text{there exists $t\in
\mc{T}$ with  $F$ a subset of the range of $t$}\}.
$$
It follows that the family $\mc{F}_{\mc{T}}$ is hereditary. Conversely given $\mc{F}$ a countable family of finite subsets of $\N$ with  each
$F\in\mc{F}$ we associate  the finite strictly increasing sequence $t_{F}$ of integers
with range equal to the set $F$. We set
$$\mc{T}_{\mc{F}}=\{t\in \cup_{n} [\N]^{n}: \mbox{there exists $F\in\mc{F}$ such that $t$ is initial segment of $t_{F}$}\}.
$$
From the above definitions it follows that $F\in \mc{F}^{1}$ if and only if the
corresponding node in the tree has infinitely many immediate successors. It follows
$CB(\mc{F}) \leq o(\mc{T_{F}})$. In the case of the Schreier families  (c.f \cite{AA}) it
follows that
$$
CB(\mc{S}_{\xi})=o(T_{\mc{S}_{\xi}})=\omega^{\xi}+1.
$$
For  unexplained notions and notations we refer the reader to
\cite{LT}.
\section{Concepts of proximity to $\ell_{1}$}\label{sec3}
In this section we introduce the hereditary strategic $\ell_{1}$-index (Def. \ref{hs}) and
show that it is essentially equivalent to the notion of the hereditary Bourgain
$\ell_{1}$-index (Prop. \ref{her-asympt}). We discuss also non-hereditary concepts of
proximity to $\ell_{1}$ and show strong relation between these notions in the sequence
setting (Prop. \ref{p2.7}).

\begin{definition}
Let $X$ be a Banach space with a basis.

A tree $ \mc{T}$ on $X$ is an $\ell_{1}$-\textit{tree on} $X$ with constant $C\geq 1$, if any
$(x_1,\dots,x_k)\in  \mc{T}$ is a normalized sequence $C$-equivalent to the unit vector basis
of $\ell_{1}^k$. Let
$$
I(X,C)=\sup\{o( \mc{T}):\  \mc{T} \ -\ \ell_{1}\text{-tree on } X \text{ with constant
}C\}, \ \ C\geq 1
$$
The (\textit{Bourgain) $\ell_{1}$-index of $X$} is defined by $I(X)=\sup\{I(X,C):\ C\geq
1\}$.

The \textit{block (Bourgain)} $\ell_{1}$-index $I_b(X)$, is defined
analogously, using $\ell_{1}$-trees consisting only of block
sequences.
\end{definition}
It follows by  \cite{B} that for a separable Banach space $X$ the $\ell_{1}$-index is a
countable ordinal if and only if $X$ does not contain $\ell_{1}$. In this case the
$\ell_{1}$-index is of the form $\omega^{\xi}$ for some $\xi<\omega_{1}$ and also it is
greater than $I_b(X,C)$ for any $C\geq 1$, and the same holds for the block
$\ell_{1}$-index \cite{JO}. Recall that if $I(X)\geq \omega^\omega$, then $I(X)=I_b(X)$,
if $I(X)=\omega^{n+1}$, then $I_b(X)=\omega^n$ or $I_b(X)=\omega^{n+1}$ \cite{JO}. For
more information on the block $\ell_{1}$-index see \cite{JO}.

The hereditary (Bourgain) $\ell_{1}$-index is defined as
$I_{h}(X)=\min\{I(Y): Y$ is  subspace of $X\}$. The
hereditary block (Bourgain) $\ell_{1}$-index, denoted as
$I_{hb}(X)$, is defined similarly taking block Bourgain $\ell_{1}$-indices of block subspaces.

Next we introduce the following notion.
\begin{definition}
Let $X$ be a Banach space with basis, $\xi<\omega_{1}$, $C\geq 1$ and consider an
$\mc{S}_\xi$-game between the players $S$ and $V$ defined as follows:

in the $i-$th move player $S$ chooses a block subspace $X_i$ of $X$ and player $V$ picks a
normalized block vector $x_i\in X_i$.

We say that player $V$ wins, if the resulting sequence $(x_{i})_{i=1}^{k}$ is
$C$-equivalent to the standard basis of $\ell_{1}^{k}$ and maximal
$\mc{S}_{\xi}$-admissible (i.e.  $\{\minsupp x_i:\ i=1,\dots,k\}$ is
$\mc{S}_{\xi}$-maximal). We say that $V$ has a winning strategy, if $V$ wins the game for
any possible choice of the player $S$.
\end{definition}
\begin{definition}\label{hs}
Let $X$ be a Banach space with a basis. We define the \textit{hereditary strategic $\ell_{1}$-index} of $X$ by the formula
\begin{align*}
I_{hs}(X)= \sup\{&\xi<\omega_1:\ \text{for any}\ \zeta<\xi\ \text{there is}\ C\geq 1\ \text{such that}
\\
&\text{V has a winning strategy in }S_\zeta-\text{game with
constant}\ C\}.
\end{align*}
If we consider the game above where $S$ chooses only tail
subspaces instead of arbitrary block subspaces, then we define the
\textit{strategic $\ell_{1}$-index} of $X$ denoted by $I_s(X)$.
\end{definition}
In non-hereditary case for $\xi=1$ we obtain Definition 2.1.
\cite{MMT}.

The next result describes the relations between the above
introduced notions.
\begin{proposition}\label{her-asympt}
Let $X$ be a Banach space with a basis, $\xi<\omega_1$.

If $I_{hs}(X)>\xi$ then $I_{hb}(X)>\omega^\xi$.

If $I_{hb}(X)>\omega^{\xi}$ then $I_{hs}(Y_0)>\xi$ for some block
subspace $Y_0$ of $X$.
\end{proposition}
\begin{proof}
We can assume that $X$ does not contain a copy of $\ell_{1}$. Notice
that for any $\xi<\omega_1$ we have $I_{hs}(X)>\xi$ if and only if
player $V$ has a winning strategy in $S_\xi$-game with some
constant $C\geq 1$.

Assume first that player $V$ has a winning strategy in the
$\mc{S}_\xi$-game (with some constant $C$). Fix a block subspace
$Y$. Then player $V$ in particular has a winning strategy in
producing maximal $\mc{S}_\xi$-admissible sequences
$C$-equivalent to the unit vector basis of $\ell_{1}$ in the
$\mc{S}_\xi^Y$-game, in which player $S$ chooses at each step some
$m_k\in\N$ and $V$ chooses $x_k\in Y$ with $x_k>m_k$ (i.e. player
$S$ chooses only tail subspaces of $Y$).

Take the tree $\mc{T}$ of all block sequences produced by player
$V$ in all $\mc{S}_\xi^Y$-games for all block subspaces $Y$,
with all possible moves of player $S$,
i.e. all block sequences $(x_1,\dots,x_k)$ produced by player
$V$ at some point in all $\mc{S}_\xi^Y$-games according to his
winning strategy. It follows that $\mc{T}$ is an $\ell_{1}$-tree
with a constant $C$. We show now that $o(\mc{T})\geq
o(\mc{S}_\xi)=\omega^\xi+1$ which implies that
$I_b(Y)>\omega^\xi$.

For any $(m_1,\dots,m_k)\in S_\xi$ by $\mc{T}_{m_1,\dots,m_k}$
denote the set of all block sequences of length $k$ produced by
$V$ according to his winning strategy in all $\mc{S}_\xi^Y$-games,
where $S$ has chosen in his first $k$ moves $m_1,\dots,m_k$.

By induction we show that for any $\al<\omega_1$ we have
$$
D^\al(\mc{T})\supset \cup\{\mc{T}_{m_1,\dots,m_k}:\
(m_1,\dots,m_k)\in D^\al(\mc{S}_\xi)\}.
$$
Indeed, if $\al=\beta+1$ (in particular if $\beta=0$), then by
the inductive assumption (or by the definition of $\mc{T}$ in case of
$\beta=0$) we have
\begin{align*}
D(&D^\beta\mc{T})
\supset 
\cup\{D (\mc{T}_{m_1,\dots,m_k}): (m_1,\dots,m_k)\in
D^\beta(\mc{S}_\xi)\}
\\
&
=\cup\{\mc{T}_{m_1,\dots,m_{k-1}}:
(m_1,\dots,m_k)\in D^\beta(\mc{S}_\xi)\}
=
\cup\{\mc{T}_{m_1,\dots,m_l}: (m_1,\dots,m_l)\in D^\al
(\mc{S}_\xi)\}
\end{align*}
For any limit $\al<\omega_1$ by the inductive assumption we have
\begin{align*}
\cap_{\beta<\al}D^\beta\mc{T}&\supset \cap_{\beta<\al}\cup\{
\mc{T}_{m_1,\dots,m_k}:\ (m_1,\dots,m_k)\in
D^\beta(\mc{S}_\xi)\}
\\
&\supset
\cup\{\mc{T}_{m_1,\dots,m_k}:\ (m_1,\dots,m_k)\in D^\al (\mc{S}_\xi)\}.
\end{align*}
It follows that $o(\mc{T})\geq o(\mc{S}_\xi)$, which ends the
proof of the first implication.

Assume now that for any block subspace $Y$ of $X$ we have $I_b(Y)>\omega^\xi$. We will
show now that in any block subspace $Y$ there is a block subspace $Y_0$ and $C\geq 1$
such that in any subspace $W\subset Y_0$ there is a maximal $\mc{S}_\xi$-admissible block
sequence $(x_1,\dots,x_k)$ $C$-equivalent to the unit vector basis of $\ell_{1}^k$. Then
by Gowers dichotomy for games for families of finite block sequences \cite{G} in some
block subspace of $Y$ player $V$ has a winning strategy for producing maximal
$S_\xi$-admissible block sequences $2C$-equivalent to unit vector basis of suitable
finite dimensional spaces $\ell_{1}$, which will prove that $I_{hs}(X)>\xi$.

First notice that there is some block subspace $Y_0$ and universal
constant $C$ such that for any block subspace $Z$ of $Y_0$ there
is an $\ell_{1}$-tree $\mc{T}_Z$ on $Z$ with constant $C$ and
$o(\mc{T_{Z}})>\omega^\xi+1$. Since we deal with $\ell_{1}$-trees, we
can assume that they are hereditary trees, i.e. for any
$\ell_{1}$-tree $\mc{R}$, any $(x_i)_{i\in F}\in \mc{R}$ and any
$G\subset F$ we have also $(x_i)_{i\in G}\in \mc{R}$.

Indeed, otherwise we can produce a decreasing sequence of block
subspaces $(Y_n)$ such that in each $Y_n$ there is no
$\ell_{1}$-tree with constant $n$ and order greater than
$\omega^\xi+1$. Then the diagonal subspace does not contain any
$\ell_{1}$-tree of order greater than $\omega^\xi+1$, a
contradiction with the assumption, since the block  $\ell_{1}$-index
of a Banach space is a limit ordinal.

Take $Y_0$ and the constant $C$ as above, for any block subspace
$Z$ of $Y_0$ pick an $\ell_{1}$-tree $\mc{T}_Z$ with constant $C$
and $o(\mc{T}_Z)>\omega^\xi+1$. Let $\mc{T}=\bigcup \{\mc{T}_Z:
Z\subset Y_0\}$. Then $\mc{T}$ is also an hereditary $\ell_{1}$-tree
(with constant $C$).

Pick any block subspace $W=\langle (w_n)_n\rangle$ of $Y_0$. Let $L=\{\minsupp w_n:\
n\in\N\}$ and
$$
\mc{F}=\{(\minsupp x_1,\dots,\minsupp x_k)\subset L: \ (x_1,\dots,x_k)\in
\mc{T}\cap W^k,\ k\in\N\}
$$
Notice that the tree $\mc{F}$ is hereditary and well-founded.
Indeed, assume that there is $(n_i)_{i\in\N}$ such that
$(n_1,\dots,n_k)\in \mc{F}$ for any $k\in\N$. Thus for any
$k\in\N$ there is $(x_1^k,\dots,x_k^k)\in \mc{T}$ with $n_1\leq
x_1^k<n_2\leq x_2^k<\dots<n_k\leq x_k^k$. By compactness argument
we can assume that for some block sequence $(x_i)$ we have
$x_i^k\to x_i$ in $X$ as $k\to\infty$. Since each sequence
$(x_1^k,\dots,x_k^k)$ is $C$-equivalent to the unit vector basis
of $\ell_{1}^k$, thus $(x_i)_{i\in\N}$ is equivalent to the unit
vector basis of $\ell_{1}$, a contradiction with the assumption from
the beginning of the proof.

By I. Gasparis' dichotomy \cite{GA} there is some infinite
$M\subset L$ such that either $\mc{S}_\xi\cap [M]^{<\infty}\subset
\mc{F}$ or $\mc{F}\cap [M]^{<\infty}\subset \mc{S}_\xi$.

We show that the first case holds by the following
\begin{claim}
Given any tree $\mc{R}$ on $X$ let
$$
\mc{F}_\mc{R}=\{(\minsupp x_1,\dots,\minsupp x_k): (x_1,\dots,x_k)\in
\mc{R}\}
$$
Then if trees $\mc{R}$ and $\mc{F}_\mc{R}$ are well-founded it
follows that $o(\mc{R})\leq o(\mc{F}_\mc{R})$.
\end{claim}
To prove the claim it is enough to show by induction that
$\mc{F}_{D^\alpha (\mc{R})}\subset D^\alpha (\mc{F}_\mc{R}) $ for
any $\alpha<\omega_1$.

Let now $Z=\langle (w_n)_{\minsupp w_n\in M}\rangle$. By definition of $\mc{T}$ and
$\mc{F}$ we have that $\mc{F}\cap [M]^{<\infty}\supset
\mc{F}_{\mc{T}_Z}$, therefore by the above Claim
$$
o(\mc{F}\cap [M]^{<\infty})\geq
o(\mc{T}_Z)>\omega^\xi+1=o(\mc{S}_\xi)
$$
hence $\mc{S}_\xi\cap [M]^{<\infty}\subset \mc {F}$ must hold.

Take now any $(n_1,\dots,n_k)$ maximal in $\mc{S}_\xi\cap
[M]^{<\infty}\subset \mc{F}$ and the corresponding
$(x_1,\dots,x_k)\in \mc{T}\cap W^k$ with $\minsupp x_i=n_i$,
$i=1,\dots,k$. It is clear that $(x_1,\dots, x_k)$ is
maximal $\mc{S}_\xi$-admissible. Therefore we picked in $W$
a maximal $\mc{S}_\xi$-admissible block sequence
$C$-equivalent to the unit vector basis of some finite dimensional
$\ell_{1}$, which by previous remarks ends the proof.
\end{proof}
\begin{remark} \label{asympt}
The non-hereditary case the above proposition takes the following form for a Banach space
$X$ with a basis:

If $I_s(X)> \xi$, then $I_b(X)>\omega^\xi$.

If $I_b(\langle (e_n)_{n\in M}\rangle)>\omega^\xi$ for any $M\in [\N]$, then $I_s(X)> \xi$.

The proof goes along the same scheme as above, with passing to subspaces spanned by
subsequences of the basis instead of block subspaces.
\end{remark}
A stronger representation of $\ell_{1}$ in a Banach space is
described by the following notions:
\begin{definition}
Let $X$ be a Banach space with a basis and  $\xi<\omega_{1}$ be a countable ordinal.

A normalized basic sequence $(x_i)_{i\in\N}\subset X$ \textit{generates an
$\ell_{1}^{\xi}$-spreading model}, $\xi<\omega_1$, with constant $C\geq 1$, if for any
$F\in  \mc{S}_\xi$ the sequence $(x_i)_{i\in F}$ is $C$-equivalent to the unit vector
basis of $\ell_{1}^{\# F}$.

We say that $(x_i)$ generates an $\ell_{1}$-spreading model if the above property holds for
$\xi=1$.

The space $X$ is $\ell_{1}^{\xi}$-\textit{asymptotic},
 if any $\mc{S}_{\xi}$-admissible sequence
$(x_i)_{i=1}^n$ is $C$-equivalent to the unit vector basis of $\ell_{1}^n$.

The space $X$  is said to be $\ell_{1}$-asymptotic if the above property holds for
$\xi=1$.
\end{definition}
The asymptotic $\ell_{1}$ spaces were introduced in \cite{MT}, $\ell_{1}$-spreading models of
higher order were studied in \cite{AMT},\cite{AG}.

Replacing $\ell_{1}$ by $c_0$ in the above definitions we obtain the (Bourgain) $c_0$-index,
a $c_0^\xi$-spreading model and an asymptotic $c_0$ space.

The structures described in the above definitions give us an hierarchy of the
representation of $\ell_{1}$ in a Banach space $X$ with a basis, not containing
$\ell_{1}$. For a countable ordinal $\xi<\omega_{1}$ we consider the following four
``local`` structures:
\begin{enumerate}
\item[A)] The space $X$ is  $\ell_{1}^{\xi}$ asymptotic. \item[B)]
The space $X$ contains a sequence generating an
$\ell_{1}^{\xi}$-spreading model.
 \item[C)] The strategic
$\ell_{1}$-index of $X$ is greater than $\xi$.
\item[D)] The block
$\ell_{1}$-index of $X$ is greater than $\omega^{\xi}$.
\end{enumerate}
Let us observe that the following implications hold for the above
structures
$$
A)\Rightarrow B)\Rightarrow C) \Rightarrow D)
$$
The reverse implications are not true.

First observe that for every $\xi<\omega_{1}$ the Tsirelson spaces
$T[\mc{S}_{\xi},\theta]$, $\theta\in (0,1)$, are examples of reflexive Banach space which
are $\ell_{1}^{\xi}$-asymptotic however they do not contain $\ell_{1}$.

The Schreier space  $X_{\xi}$, $\xi<\omega_{1}$, is an example of  Banach space  with
basis for which every subsequence of the basis   generates an $\ell_{1}^{\xi}$-spreading
model and it does not contain an asymptotic  $\ell_{1}^{\zeta}$ subspace for any
$\zeta\leq\xi$ since, as it is well known, $X_{\xi}$ is $c_0$-saturated.

In the next section of the paper for every $\xi<\omega_{1}$ we
provide an example of Banach space $X_{G_\xi}$ with  strategic
$\ell_{1}$-index greater than $\xi$ yet $X_{G_{\xi}}$ does not
contain sequence generating an $\ell_{1}$-spreading model. More
precisely the space $X_{G_{\xi}}$ has basis such that in the game,
where $S$ chooses subspaces spanned by subsequences of the basis,
player $V$ has a winning strategy. Therefore if a space $X$ has
 strategic $\ell_{1}$-index greater than $\xi$, then it does not
follow that $X$ contains a sequences generating an
$\ell_{1}$-spreading model.

A richer asymptotic $\ell_{1}^{\xi}$ structure than the one described in property $C)$
would be provided by a winning strategy of player $V$ in the following modification of
the $\mc{S}_\xi$-game: player $S$ chooses subsequences of the basis instead of block
subspaces and player $V$ picks vectors from subsequences chosen by player $S$. In such
a case every subsequence of the basis would admit $\ell_{1}$-tree (i.e. formed by elements
of this subsequence) of order $\omega^{\xi}$. Let us observe the following
\begin{proposition}\label{p2.7} Let $(e_n)$ be a basis of $X$ such that any subsequence $(e_n)_{n\in
L}$ admit an $\ell_{1}$-tree $\mc{T}_L\subset\{(e_n)_{n\in F}:\ F\in [L]^{<\infty}\}$ of order
$\omega^\xi$, $\xi<\omega_1$. Then some subsequence of $(e_n)_{n\in\N}$ generates an
$\ell_{1}^{\xi}$-spreading model.
\end{proposition}
\begin{proof}
Assume first that any subsequence $(e_n)_{n\in L}$ of the basis admits an $\ell_{1}$-tree
of order greater than $\omega^\xi+1$.

Repeating the reasoning and notation from the proof of Prop. \ref{her-asympt} we can
assume that there is some subsequence $(e_n)_{n\in L_0}$ and an hereditary $\ell_{1}$-tree
$\mc{T}$ formed by elements of this subsequence such that for any $M\subset L_0$ the
order of $\mc{F}_\mc{T}\cap [M]^{<\infty}$ is greater than $\omega^\xi+1$.

As before by I.Gasparis' dichotomy \cite{GA} pick an infinite $M\subset\N$ such that
either $\mc{S}_\xi\cap [M]^{<\infty}\subset \mc{F}_\mc{T}$ or $\mc{F}_\mc{T}\cap
[M]^{<\infty}\subset \mc{S}_\xi$. By the above remark the first case holds. Notice that
if $M=(m_i)$, then for any $F\in \mc{S}_\xi$ we have $(m_i)_{i\in F}\in \mc{S}_\xi\cap
[M]^{<\infty}\subset \mc{F}_\mc{T}$, which implies that $(e_n)_{n\in M}$ generates an
$\ell_{1}^\xi$-spreading model.

Now notice that if $(e_n)$ admits an $\ell_{1}$-tree of order $\omega^\xi$, then it also
admit an $\ell_{1}$-tree of order greater than $\omega^\xi+1$ (maybe with worse constant).

It follows by repeating the reasoning in Lemma 5.7 \cite{JO} in case of the block
$\ell_{1}$-index defined not by using all block $\ell_{1}$-trees but only
$\ell_{1}$-trees consisting of finite subsequences of $(e_n)$. Therefore we can reduce the
case where all subsequences admit $\ell_{1}$-trees of order $\omega^\xi$ to the case of
order greater than $\omega^\xi+1$, which was treated above and hence we finish the proof
of the observation.
\end{proof}
We end this section with the following observation regarding the non-hereditary indices.
Assuming only $I_b(X)>\omega^\xi$ does not imply $I_s(X)> 1$, as it shown by the following example.

Let $X$ be an $\ell_{2}$-direct sum of $\ell_{1}^{n}$'s, i.e.
$X=\left(\bigoplus_{n=1}^{\infty}\ell_{1}^n\right)_{\ell_{2}}$.

It is clear that $I_b(X)>\omega$, since $\sup\{o(\mc{T}): \mc{T} \ \ell_{1}\text{-tree
with constant } 1\}=\omega$.

On the other hand $I_s(X)=1$. Indeed, consider $\mc{S}_1$-game with player $S$ choosing
tail subspaces and let in the first move player $S$ pick a tail subspace
$(\oplus_{n=n_0}^\infty\ell_{1}^n)_{\ell_2}$. In the $(i+1)$-th move player $S$ chooses
tail subspaces after the support of vector $x_i$ picked by $V$ in the $i$-th move, i.e.
if in the $i$-th move player $V$ picks a vector
$x_i\in(\oplus_{n=1}^{n_i}\ell_{1}^n)_{\ell_2}$, then in the $(i+1)$-th move player $S$
chooses tail subspace $(\oplus_{n=n_i+1}^\infty\ell_{1}^n)_{\ell_2}$.

Therefore player $S$ forces the sequence $(x_1,\dots,x_k)$ produced by player $V$ to be
1-equivalent to the unit vector basis of $\ell_2^k$, hence there is no universal constant
$C$ for which player $V$ has a winning strategy for producing $\mc{S}_1$-admissible block sequences
$C$-equivalent to the unit vector basis of suitable finite dimensional $\ell_{1}$.
\section{Definition of the ground set $G_{\xi}$}\label{sec4}
In this section we define for any $\xi<\omega_1$ ground sets $G'_\xi$ and $G_\xi$ for the
auxiliary spaces $Y_{G_\xi}$ and $X_{G_\xi}$ respectively, in particular we introduce the
tree of $G_\xi$-special functionals, defined with the use of special coding function
$\sigma_1$, and related notions. We show that the space $Y_{G_\xi}$ is $c_0$-saturated
(Prop. \ref{c0}).

We recall that a  subset $G$ of $c_{00}(\N)$  is said to be a ground set if:
\begin{enumerate}
 \item  $G$ is symmetric and  $\{e_{n}^{*}:n\in\N\}$ is contained in $G$.
\item $\norm[\phi]_{\infty}\leq 1$ and $\phi(n)\in\Q$ for $\phi\in G$.
\item $G$ is closed under the restriction of its elements to intervals of $\N$.
\end{enumerate}
 A ground set $G$ induces a norm $\norm_{G}$ in $c_{00}(\N)$ defined by
\begin{center} $\norm[x]_{G}=\sup\{\phi(x):\phi\in G\}$ and we set
$X_{G}=\overline{( c_{00}(\N), \norm_{G})}$.
\end{center}
It is easy to see that the natural basis of $c_{00}(\N)$ is a Schauder
basis of the space $X_{G}$.  In the opposite for every Banach space with a Schauder basis $\xn$ there exists a ground set $G$ such that  the natural correspondence $e_{n}\to x_{n}$ is extended to an isomorphism between
$X_{G}$ and $X$.

We pass now  to define the ground set $G_{\xi}$.
Fix $\xi<\omega_{1}$. We choose two strictly increasing sequences  $(n_{j})_{j}$,
$(m_{j})_{j}$ of positive integers, such that
\begin{enumerate}
\item[(i)]
$m_{1}=2^{5}$  and  $m_{j+1}= m_{j}^{5}$
\item [(ii)]
$n_{1}=2^{6}$  and $n_{j+1}=(2n_{j})^{s_{j}}$ where
$2^{s_{j}}=m_{j+1}^{3}$\,.
\end{enumerate}
Let us observe, for later use,
that $260m_{2j}^{4}\leq n_{2j-1}$ for  $j\geq 2$.
Set $G_{0}=\{\pm e^{*}_{n}: n\in\N\}$
\begin{align*}
G_{1}^{j}=\left\{\frac{1}{m^2_{2j-1}}\sum_{i\in E}f_{i}:,
f_{i}\in G_{0}, (f_{i})_{i\in E},\,\,\# E\leq n_{2j-1}\right\},\,\,
j\in \N.
\end{align*}
Finally we set $G_{1}=\cup_{j\in\N}G_{1}^{j}$.
\begin{notation}
For every  $f\in G_{1}^{j}$ we define  the weight of $f$ as $w(f)=m^2_{2j-1}$
 and  the index as $\ind(f)=j$.
The elements of $G_1$ are called functionals of type I.
 \end{notation}
We consider the following set
$$
\mathcal{U}=\{(f_{1},\dots,f_{d}): f_{1}<\dots<f_{d}, f_{i}\in
G_{1}^{j_i}, j_{i}<j_{i+1}\,\,\textrm{for all}\, i<
d\in\N\}.
$$
Let  $\N=M_{1}\cup M_{2}$ where  $M_{1},M_{2}$ are infinite disjoint sets. Let
also $ \sigma_{1}:\mathcal{U}\to M_{2} $
be a $1-1$ function (i.e. a coding function) such that
$\sigma_{1}(f_{1},\dots,f_{i+1})>\sigma_{1} (f_{1},\dots,f_{i})$ for every $i\in\N$.
\begin{definition}\label{dee1}
A special sequence  is an element
$(f_{1},\dots,f_{d})$ of $\mathcal{U}$ satisfying the following:
\begin{enumerate}
\item[1)] $(\minsupp f_{i})_{i=1}^{d}\in\mc{S}_{\xi}$. \item[2)]
$f_{1}\in G_{1}^{j_1}, j_{1}\in M_{1}$,
$f_{i+1}\in G_{1}^{\sigma_{1}(f_{1},\dots,f_{i})}$\,\, for every
$i=1,\dots,d-1$.
\end{enumerate}
\end{definition}
We define
$$
G_{sp}=\left\{E\sum_{i=1}^{d}\e_{i}f_{i}:
(f_{1},\dots,f_{d})\, \textrm{is a special
sequence},\,\,\e_{i}\in\{-1,1\},\,\, E\,\,\textrm{interval
of}\,\N\right\}
$$
For any element
$\phi=E\sum_{i=1}^{d}\e_{i}f_{i}$ of $G_{sp}$ we set
$\ind(\phi)=\{\ind
(f_i):\, Ef_i\neq \emptyset, i\leq d\}$. We define also
$$
G_{\ell_2}=\left\{\sum_{i=1}^{d}a_i\phi_{i}:d\in\N,\,\,\sum_{i=1}^{d}a_{i}^{2}\leq
1,\,\,\,(\phi_{i})_{i=1}^{d}\subset G_{sp}\cup G_1,\,\,\,
(\ind\phi_i)_{i=1}^d\,\,\,\textrm{pairwise disjoint}\right\}.
$$
\begin{definition}\label{dee2}
 The ground set $G_{\xi}$ is defined to be  the set
\begin{align*}
G_{\xi}=G_0\cup G_{1}\cup G_{sp}\cup G_{\ell_2}.
\end{align*}
\end{definition}
\begin{remarks}\label{r3}
1) The set  $G_{\xi}$ is symmetric and closed  under the restriction of its
elements  on intervals of $\N$.

2) The injectivity  of the coding function $\sigma_{1}$ yields that the set of the
special sequences has a tree structure i.e. if  $(f_{1},\dots,f_{d})$,
$(g_{1}\,\dots,g_{n})$ are two special sequences then either $f_{i}\ne g_{j}$ for all
$i,j$ or there exists $i_0\leq\min\{d,n\}$ such that $f_{i}=g_{i}$ for all $i<i_{0}$ and
$f_{i}\ne g_{j}$ for all $i,j\geq i_0$, in particular $w(f_i)\neq w(g_j)$ for all
$i,j>i_0$.
\end{remarks}
We shall call every special sequence also a segment of the tree of the special sequences.
The elements of $G_{sp}$ are called $G_{\xi}-$special functionals.
\begin{notation}\label{not2} For every segment $s=(f_1,\dots,f_{d})$ of the tree of the special sequences  we set
$
F(s)=\{\sum_{i=1}^{d}\e_{i}f_{i}: \e_{i}\in\{-1,1\}\}
$
and for $f\in F(s)$ we set $\inde(f)=\inde(s)$.

If $V=\{s_{1},\dots,s_{d}\}$ is a set of segments we set
$$
\overline{V}=\left\{\sum_{i=1}^{d}\lambda_{i}\phi_{i}: \sum_{i=1}^{d}\lambda_i^2\leq 1,\phi_{i}\in F(s_i)\,\,\text{and $\inde(\phi_{i})\cap\inde(\phi_{j})=\emptyset$ for all $1\leq i\ne j\leq d$}\right\}
$$
It readily follows that $\overline{V}$ is symmetric.
\end{notation}
We set  $G_{\xi}^{\prime}=G_{0}\cup G_{1}\cup G_{sp}\cup \{ 0\}$.
\begin{lemma}\label{l4}
The set $G^{\prime}_{\xi}$ is a closed  subset of $[0,1]^{<\omega}$ in the pointwise
topology.
\end{lemma}
\begin{proof}
Let  $(\phi_{n})_{n\in\N}$ be  a  sequence of elements of $G^{\prime}_{\xi}$ such that
$\lim_{n}\phi_{n}=\phi$  pointwise. It is clear that if  $\phi_{n}\in G_0$ for infinitely
many $n$'s then $\phi\in G_0\cup\{0\}$.

Also if  $\phi_n\in G_{1}$ for all but finitely many $n$'s, then if
$\{\phi_{n}:n\in\N\}\cap G_{1}^{j}\ne\emptyset$ for infinitely many $j$'s then $\phi=0$.
Otherwise  there exists $j_0$ such that $\phi_{n}\in G_{1}^{j_0}$ for all but finitely
many $n$'s. Since $\mc{A}_{n_{2j_0}-1}$ is regular we get $\phi\in G_{1}^{j_0}$.

Finally  assume that   for all but finitely many $n$'s there exists a segment
 $s_{n}=(f_{1}^{n},\dots,f_{d_n}^{n})$ of the tree of the special sequences such that
$\phi_{n}=E_n\sum_{i=1}^{d_{n}}\e_i^{n}f_{i}^{n}\in F(s_n)$ .

If $\sup_{n}\min E_n=+\infty$, then $\phi=0$.
Otherwise we can
assume that for all but finitely many $n$'s we have $E_n=[k,K)$, for some $k\in\N$ and
$K\in\N\cup\{+\infty\}$.

Since the family $\mc{S}_{\xi}$ is regular we get $\{m_{i}\}_{i=1}^{d}\in \mc{S}_{\xi}$
such that
$$
m_{1}\leq f_{1}^{n}<m_{2}\leq f_{2}^{n}\dots <m_{d}\leq f_{d}^{n}\,\,\,\textrm{(and $\minsupp f_{d+1}^{n}\to\infty$)}
$$
Let $i_0=\min\{i\leq d:\textrm{there is no $n\in\N$ such that $f_{i}^{n}= f_{i}^{m}$ for
all $m\geq n$}\}. $ By the reasoning above in case $(\phi_n)\subset G_1$ we can assume
that for any $i=1,\dots,d$ we have $f_i^n\to f_i$ for some $f_i\in G_1\cup \{0\}$. Also we may assume that $\e_{i}^{n}=\e_{i}$ for all $i\leq d$ and  $n\in\N$.

If $\limsup_{n} w(f_{i_0}^{n})=+\infty$ then $f_{i_0}=0$. Also since
$\sigma_{1}$ is $1-1$ we get  $\limsup w(f_{i_0+j}^{n})=+\infty$
and therefore $f_{i_0+j}=0$ for all $j=1,\dots,d-i_0$.

If $w(f_{i_0}^{n})=m_{2j_0-1}^2$ for all but finitely many $n$'s, since the set
$G_{1}^{j_0}$ is compact we get that  $f_{i_0}^{n}\to f_{i_0}\in G_{1}^{j_0}$. Again
since $\sigma_{1}$ is $1-1$ we get $\limsup w(f_{i_0+j}^{n})=+\infty$ and therefore
$f_{i_0+j}=0$ for all $j=1,\dots,d-i_0$.

It follows that  $(f_{1},\dots, f_{d})$ is a special sequence and therefore
 $\sum_{i=1}^{d_n}\e_{i}^{n}f_{i}^{n}\to \sum_{i=1}^{d}\e_{i}f_{i}$ is a $G_{\xi}-$special functional.
\end{proof}
We define $X_{G_{\xi}}=\overline{(c_{00}(\N),\norm_{G_{\xi}})}$ and
$Y_{G_\xi}=\overline{(c_{00}(\N),\norm_{G^{\prime}_{\xi}})}$.

Let us observe that since the sets $G_{\xi}, G_{\xi}^{\prime}$ are symmetric and closed
under projections on intervals it follows that $(e_{n})_{n\in\N}$ is a bimonotone basis
for the spaces  $X_{G_{\xi}}$ and $Y_{G_{\xi}}$.
\begin{proposition}\label{c0}
$Y_{G_{\xi}}$ is $c_0$-saturated.
\end{proposition}
\begin{proof}
The set $G_{\xi}^\prime$ is countable by the construction and compact by Lemma \ref{l4}.
Since  $G_{\xi}^{\prime}$ is norming set of  $Y_{G_{\xi}}$ we conclude that $Y_{G_{\xi}}$
is isometric to a subspace of $C(G_{\xi}^{\prime})$ and hence is $c_0$-saturated, see
\cite{PS} (see also \cite{AK} and \cite{Kal}, Theorem 4.5).
\end{proof}
We end this section with the following observations regarding the norming set $G_{\xi}$.

From the definition of the norming set  $G_{\xi}$ it follows that for any finite sequence
$(\phi_i)\subset G_1\cup G_{sp}$ with pairwise disjoint index sets and any
$(a_i)\in\ell_2$ we have $\norm[\sum_ia_i\phi_i]_{G_\xi}^*\leq (\sum_ia_i^2)^{1/2}$.
Therefore for any infinite sequence $(\phi_i)_{i}\subset G_{1}\cup G_{sp}$ with pairwise
disjoint index sets and any $(a_i)_i\in\ell_2$ the series $\sum_ia_i\phi_i$ is convergent
in norm and
$$
\|\sum_{i=1}^{\infty}a_{i}\phi_{i}\|_{G_\xi}^*\leq
(\sum_{i=1}^{\infty}a_{i}^{2})^{1/2}.
$$
Notice also that since the set $G_{\xi}$ is a norming set for $X_{G_\xi}$ we obtain that
$B_{X_{G_{\xi}}^{*}}=\overline{\conv(G_{\xi})}^{w^{*}}$.
Also it is not hard  to show that
 $\overline{G_{\xi}}^{w^*}=G_{\xi}\cup  \mathcal{F}$ where
$$
\mathcal{F}=\{\sum\limits_{i=1}^{\infty}a_{i}\phi_{i}: (a_{i})_{i\in\N}\in B_{\ell_2},
\phi_{i}\in G_{1}\cup G_{sp},\,\,\,\ind(\phi_{i})\cap \ind(\phi_{j})=\emptyset\,\,
\mbox{for all $i\ne j$}\}.
$$
We omit the proof since we shall not make use of this result.
\section{Basic properties of the space $X_{G_{\xi}}$}\label{sec5}
In this section we shall prove that the space
$X_{G_{\xi}}$ has the following properties
 \begin{enumerate}
 \item[1)] It is reflexive.
\item[2)] For every subsequence $(e_{n})_{n\in M}$ of the basis the subspace $\langle (e_{n})_{n\in M}\rangle$ has   $\ell_{1}$-index greater than $\omega^{\xi}$.
\item[3)] Every subspace of $X_{G_{\xi}}$ contains $\ell_{2}$.
\end{enumerate}
In order  to prove the theorem we shall need the following definition
\begin{definition} \label{aa1} Let $(x_n)_n$ be a bounded block
sequence in $X_{G_{\xi}}$ and $\e >0.$ We say that $(x_n)_n$ is {\em
$\e$-separated} if for every $\phi\in G_{1}=\cup_{j\in
\mathbb{N}}G_{1}^j$
$$
\# \{n: |\phi(x_n)|\ge\e \}\le 1.
$$
In addition, we say that $(x_n)_n$ is {\em separated} if for every
$L\in [\N]$ and $\e >0$ there exists an $M\in [L]$ such that
$(x_n)_{n\in M}$ is $\e$-separated.
\end{definition}
Concerning the  separated sequences the following holds:
\begin{proposition}\label{aa4}
Let $(x_n)_{n\in\N}$ be a  separated sequence in $X_{G_{\xi}}$ with
$\norm[x_n]_{G_\xi}\leq 1$. Then for all $m\in\N$ there is $L\in[\N]$ such that for all
$g\in G_{\ell_2}$
 $$
\#\{n\in L: \vert g(x_n)\vert\geq m^{-1}\}\leq 65m^2.
$$
\end{proposition}
The proof of the above results follows the arguments of Lemma 3.12 and Proposition 3.14 in \cite{AAT} where  we refer  for the proofs.
\begin{lemma}\label{l1a}
Let  $\xn$ be a bounded block sequence. Setting $y_{n}=\frac{1}{n}\sum_{i\in F_{n}}x_i$
where $\# F_{n}=n$ and $F_{n}<F_{n+1}$ we get that $\xn[y]$ is separated.
\end{lemma}
\begin{proof}
Let  $\e>0$, $L\in [\N[ $  and  assume that $\norm[x_n]_{G_\xi}\leq C$ for any $n$.
Pick inductively sequences $(j_i)$, $(l_i)\subset L$ such that
\begin{center}
1) $\e m_{2j_i-1}^2>C\#\ran y_{l_i}$\,\,\,\  and\,\,\,\
2)$\e l_i>Cn_{2j_{i-1}-1}$.
\end{center}
Take now any $\phi\in G_1^j$ and $l_i$. If $j\geq j_i$, then by (1) $\vert\phi
(y_{l_i})\vert <\e$. If $j<j_{i-1}$, then by (2) also $\vert\phi (y_{l_i})\vert <\e$.
Therefore $\#\{i:\ \vert\phi (y_{l_i})\vert \geq \e\}\leq 1$ and hence setting
$M=(l_i)_i$ we end the proof that $(y_l)_{l\in M}$ is $\e$-separated.
\end{proof}
Combining the above lemma  with Proposition \ref{aa4} and  the choice of $(m_j)$, $(n_j)$
we obtain the following.
\begin{proposition}\label{ris-sss} Let $\xn$ be a bounded block sequence. Assume that $y_{n}=\frac{1}{n}\sum_{i\in
F_{n}}x_i$, where $\# F_{n}=n$  and $F_{n}<F_{n+1}$, satisfy $\norm[y_n]\leq 1$. Then for
any $j\geq 2$ there is an infinite $L\subset\N$ such that for any $g\in G_{\xi}$
$$
\#\{n\in L: \ \vert g(y_n)\vert \geq 2m_{2j}^{-2}\}\leq n_{2j-1}
$$
\end{proposition}
As corollary of the above proposition we obtain that the basis is shrinking.
\begin{corollary}\label{shrinking}
Every bounded block sequence in $X_{G_{\xi}}$ is weakly null.
\end{corollary}
\begin{proof}
Assume that there exist a normalized block sequence $\xn$, $x^{*}\in X_{G_{\xi}}^{*}$ of norm one and $\e>0$ such that
$x^{*}(x_n)>\e$ for all $n$.

Let $j\in\N$ , $j\geq 2$ such that $1/2m_{2j}^2<\e/4$. By Proposition  \ref{ris-sss} setting  $y_{n}=\frac{1}{\# F_{n}}\sum_{i\in F_{n}}x_{n}$,  where
$\#F_{n}=n$ and  $F_{n}<F_{n+1}$ for all $n\in\N$,
we may assume that  for all  $g\in G_{\xi}$ it holds that
\begin{equation}
\#\{n: \vert
g(y_n)\vert\geq 2m_{2j}^{-2}\}\leq n_{2j-1}\Rightarrow
\vert g(\frac{1}{n_{2j}}\sum_{i=1}^{n_{2j}}y_n)\vert\leq \frac{n_{2j-1}}{n_{2j}}+\frac{1 }{2m_{2j}^{2}}\leq \frac{1}{m_{2j}^{2}}<\frac{\e}{2}. 
\end{equation}
This yieldd a contradiction
since   $G_{\xi}$ is a norming set for $X_{G_{\xi}}$.
\end{proof}
We prove now the reflexivity of  $X_{G_{\xi}}$.
\begin{theorem}\label{refl} The space $X_{G_\xi}$ is reflexive.
\end{theorem}
\begin{proof}
We show that the basis is shrinking and boundedly complete.
Corollary \ref{shrinking} yields that the basis $(e_{n})_{n\in\N}$
is shrinking. We  prove that $(e_{n})_{n\in\N}$  is also
boundedly complete.

On the contrary assume that
$\sup_{n\in\N}\norm[\sum_{i=1}^{n}a_{i}e_{i}]_{G_\xi}\leq 1$ and there exist $\e_0>0$,
successive intervals $F_{1}<F_{2}<\dots$ of   $\N$ such that  $\norm[\sum_{i\in
F_{n}}a_{i}e_{i}]_{G_\xi}> \e_0$.

For every $n$ choose $g_{n}\in G_{\xi}$ such that  $g_{n}(\sum_{i\in
F_{n}}a_{i}e_{i})\geq \e_0$ and  $\ran g_{n}\subset F_{n}$.

We distinguish the following cases

\textit{Case} 1. $g_{n}\in G_0$ for infinitely many $n$'s.

Then  for $j\in\N$ and $M>\max F_{n_{2j-1}}$ and  we get
$$
\|\sum_{i=1}^{M}a_{i}e_{i}\|_{G_\xi}\geq\frac{1}{m_{2j-1}^2}\sum_{n=1}^{n_{2j-1}}g_{n}(\sum_{i=1}^{M}a_{i}e_{i})\geq
\frac{\e_0 n_{2j-1}}{m^2_{2j-1}} ,
$$
a contradiction for large $j$.

We state the next three cases.

\textit{Case 2}. $g_{n}\in G_{1}$ for all but finitely many $n$'s.

\textit{Case 3}. $g_{n}\in G_{sp}$ for all but  finitely many $n$'s.

\textit{Case 4}. $g_{n}\in G_{\ell_2}$ for all but  finitely many $n$'s.

The proofs of these three  cases  follow the same argument hence we shall give only the  proof of \textit{Case 4}.

If we have that   infinitely many $g_{n}$'s have pairwise disjoint index sets then for
suitable $n\in\N$ and $A\in \mc{A}_{n}$ the functional $g=\frac{1}{\sqrt{n}}\sum_{i\in
A}g_{i}\in G_{\ell_2}$ will give us  a contradiction.

Assume that only finitely many $g_{n}$'s  have pairwise disjoint index sets.

Let $g_{n}=\sum_{i\in D_{n}}c_{i}\phi_{i}$ where $(\phi_{i})_{i\in D_{n}}$ have pairwise
disjoint index sets. For every $j\in\N$ we set $\phi_{i}=\phi_{i,j}^{1}+\phi_{i,j}^{2}$
where $\ind(\phi_{i,j}^{1})\subset\{1,\dots,j\}$ and
$\ind(\phi_{i,j}^{2})\subset\{j+1,j+2,\dots\}$.

Then $g_{n}=g_{n,j}^{1}+g_{n,j}^{2}=\sum_{i\in D_{n,1}}c_{i}\phi_{i,j}^{1}+\sum_{i\in
D_{n,2}}c_{i}\phi_{i,j}^{2}$. Let us observe that $\# D_{n,1}\leq j$.

We distinguish the following two subcases.

\textit{Subcase 4a.} There exists $j_0\in\N$ such that for all but finitely many $n$'s,
$$
\vert g_{n,j_0}^{1}(\sum_{i\in F_n}a_{i}e_{i})\vert\geq \e_{0}/2.
$$
Since  $\ind(\phi_{i,j_0}^{1})\subset\{1,\dots,j_0\}$ it follows that  $\#\supp \phi_{i,j_0}^{1}\leq\sum_{i\leq j_0}n_{2i-1}=n_{0}$ and
$$
\vert \phi_{i,j_0}^{1}(\sum_{i\in F_{n}}a_{i}e_{i})\vert\leq n_0\norm[\phi_{i,j_0}^{1}]_{\infty}\max_{i\in F_{n}}\vert a_i\vert
\leq \frac{n_0}{m_{1}^{2}}\max_{i\in F_{n}}\vert a_{n}\vert.
$$
It follows
$$
\e_0/2\leq \vert \sum_{i\in D_{n,1}}c_{i}\phi_{i,j_0}^{1}(\sum_{i\in
F_{n}}a_{i}e_{i})\vert\leq \sum_{i\in D_{n,1}}\vert
c_i\vert\frac{n_0}{m_{1}^{2}}\max_{i\in F_{n}}\vert a_{n}\vert\leq
j_{0}\frac{n_0}{m_{1}^{2}}\max_{i\in F_{n}}\vert a_{i}\vert.
$$
From the above relation as in \textit{Case 1} we derive a contradiction.

\textit{Subcase 4b}. For every $j,m\in\N$ there exists $n\in\N$,
$n>m$ with
$$
\vert g_{n,j}^{1}(\sum_{i\in F_{n}}a_{i}e_{i})\vert\leq \e_0/2.
$$
Then  we choose inductively an increasing sequence $(n_{i})_{i\in\N}$ such that
$$
\vert g_{n_i, j_{i-1}}^{1}(\sum_{l\in F_{n_i}}a_{l}e_{l})\vert\leq\e_{0}/2
\,\,\,\mbox{where $j_{i-1}=\max\{\ind(g_{n_k}):k\leq i-1\}$.}
$$
It follows that
$
\vert g_{n_i, j_{i-1}}^{2}(\sum_{l\in F_{n_i}}a_{l}e_{l})\vert\geq\e_{0}/2
$
and the functionals $g^{2}_{n_i,j_{i-1}}$ have pairwise disjoint index sets.

Setting $g=\frac{1}{\sqrt{n}}
\sum_{i\in A}g_{n_i,j_{i-1}}^{2}$ for suitable $n$ and $A\in\mathcal{A}_{n}$ we derive a contradiction.
\end{proof}
From Theorem \ref{refl} and  Proposition \ref{c0} we obtain the following
\begin{corollary}\label{iss}
The identity map $Id:\ X_{G_{\xi}}\to Y_{G_{\xi}}$ is strictly singular.
\end{corollary}
We prove now that every subsequence of the basis generates a subspace with 
$\ell_{1}$-index greater that $\omega^{\xi}$.
\begin{proposition}
For every $M \in [\N]$ the subspace  $\langle (e_{n})_{n\in M}\rangle$ has   $\ell_{1}$-index greater than $\omega^{\xi}$.
\end{proposition}
\begin{proof}
It is not hard to see that for every $j\in\N$ the following holds
\begin{equation}\label{las}\frac{1}{m_{2j-1}^{2}}\leq\|\frac{1}{n_{2j-1}}\sum_{i=1}^{n_{2j-1}}e_{k_i}\|_{G_\xi}\leq \frac{2}{m_{2j-1}^2}.
\end{equation}
Let  $(f_{1},\dots,f_{d})$ be a special sequence where
$f_{i}=\frac{1}{m_{2j_i-1}^{2}}\sum_{k\in F_{i}}e^{*}_{k}$ with $\# F=n_{2j_i-1}$ and
$F\subset M$.

For every $i\leq d$  take the vector $x_{i}=\frac{m_{2j_i-1}^{2}}{n_{2j_i-1}}\sum_{k\in
F_{i}}e_{k}$. It follows that $1\leq \norm[x_i]_{G_\xi}\leq 2$ and
$$
\|\sum_{i=1}^{d}a_{i}x_{i}\|_{G_\xi}\geq\sum_{i=1}^{d}\e_{i}f_{i}(\sum_{i=1}^{d}a_{i}x_{i})=\sum_{i=1}^{d}\vert a_i\vert.
$$
Since for every $F\in\mc{S}_{\xi}$ we have a special sequence $(f_{i})_{i\in F}$ it follows
that the   subspace  $\langle (e_{n})_{n\in M}\rangle$
contains  an $\ell_{1}$-tree with constant $1$ of order $\omega^{\xi}$. It follows that the  $\ell_{1}$ -index of the  subspace  $\langle (e_{n})_{n\in M}\rangle$ is greater than $\omega^{\xi}$.
\end{proof}
From \eqref{las} it follows that for every $M\in[\N]$ the $c_0$-index of the subspace
generated by the subsequence $(e_{n}^{*})_{n\in M}$ of the basis   of $X_{G_{\xi}}^{*}$
is  greater than $\omega^{\xi}$.   Indeed, take an $G_{\xi}$-special  sequence
$(f_{i})_{i=1}^{d}$ supported in the set $M$ such that each
$f_{i}=m_{2j_{i}-1}^{-2}\sum_{k\in B}e_{k}^{*}$ with  $\# B=n_{2j_{i}-1}$.  From
\eqref{las} we get  $\norm[f_{i}]\geq 1/2$ for all $i\leq d$ and hence
$$1/2\leq \|\sum_{i=1}^{d}\pm f_{i}\|\leq 1.
$$
From the above inequality easily follows that the $c_0$-index  of $\langle
(e_{n}^{*})_{n\in M}\rangle$  is greater than $\omega^{\xi}$.

We show now the the space $X_{G_{\xi}}$ is $\ell_{2}$-saturated.
\begin{theorem}
For any subspace $Y$ of $X_{G_{\xi}}$ and any $\e>0$ there is a subspace of $Y$ which is
$(1+\e)$ isomorphic to $\ell_2$.
\end{theorem}
The  proof of the theorem  follows the arguments of Lemma B.13 and Theorem B.14  from \cite{AAT} and we omit it.
\section{The space $X_{G_{\xi}}$ does  not have $\ell_{1}$ as a spreading model.}\label{sec6}
In this section we prove that  $X_{G_{\xi}}$ does not contain a sequence generating  an $\ell_{1}$-spreading model. The basic tool for the proof is  Proposition \ref{nol1} which is a combinatorial result.
We begin with the following lemmas.
\begin{lemma}\label{ll1a}
For any $x\in c_{00}(\N)$ and $\e>0$ there is $j_{0}=j_0(x,\e)\in\N$
such that for any $g\in
G_{\ell_2}$ with $\ind (g)\cap \{1,\dots,j_0\}=\emptyset$ we have $\vert g(x)\vert<\e$
\end{lemma}
\begin{proof}
Let $D=\norm[x]_{\ell_{1}}$ and take $j_0$ so that
$\sum_{j=j_0+1}^{\infty}m_{2j-1}^{-4}<(\e/D)^2$. Take now any $g\in G_{\ell_2}$,
$g=\sum_ia_i\phi_i$ with $\ind (g)\cap \{1,\dots, j_0\}=\emptyset$ and compute
\begin{align*}
\vert g(x)\vert&\leq (\sum_ia_i^2)^{1/2}(\sum_i\vert \phi_i(x)\vert^2)^{1/2}\leq
(\sum_i\norm[\phi_i]_{\infty}^2\norm[x]_{\ell_{1}}^2)^{1/2}
\\
&\leq D(\sum_{j=j_0+1}^{\infty}m_{2j-1}^{-4})^{1/2}<\e.
\end{align*}
\end{proof}
\begin{lemma}\label{ll2}
Let $\xn$ be a  normalized block sequence and $\e>0$. There exist
$j_{1}\in\N$ and $L\in [\N]$ such that for every $f$ of type I
with $\ind(f)\geq j_1$ it holds that $\vert f(x_n)\vert>\e$ for at
most one $n\in L$.
\end{lemma}

\begin{proof}
 We set
$
\mathcal{D}_{1}=\{(n,m):\text{exists $f$ of type I with}\,\, \vert f(x_n)\vert\geq\e\,\text{and}\,\,\vert f(x_m)\vert\geq\e\}.
$

If there exists  $L\in [\N]$ with
$[L]^{2}\cap\mathcal{D}=\emptyset$ the proof is complete.
Otherwise by Ramsey theorem there exists $L_{1}$ such that
$[L_1]^{2}\subset\mathcal{D}$. Let $l_1=\min L_1$ and let $j_{1}$
be such that $\vert f(x_{l_1})\vert<\e$ for all $f$ of type I with
$\ind(f)\geq j_{1}$.

We set
$$
\mathcal{D}_{2}=\{(n,m)\in [L_1]^{2}:\text{exists $f$ of type I with  $\ind(f)\geq j_1$},\, \vert f(x_n)\vert\geq\e\,\text{and}\,\vert f(x_m)\vert\geq\e\}
$$
If there exists $L_{2}\in [L_1]$ with
$[L_{2}]^{2}\cap\mathcal{D}_{2}=\emptyset$ the proof is complete.
Otherwise by Ramsey theorem there exists $L_{2}$ such that
$[L_2]^{2}\subset\mathcal{D}_{2}$. Let $l_2=\min L_{2}>l_{1}$ and
let $j_{2}$ be such that $\vert f(x_{l_2})\vert<\e$ for all $f$ of
type I with $\ind(f)\geq j_{2}$.

If the conclusion does not hold, choosing $k_0$ such the
$\e\sqrt{k_0}>1$,  following the above arguments we get
$L_{1}\supset L_{2}\supset\dots \supset L_{k_0+1}$ such that
$[L_{i}]^{2}\subset \mathcal{D}_{i}$, where
$$
\mathcal{D}_{i}=\{(n,m)\in [L_{i-1}]^{2}:\text{exists $f$ of type I with  $\ind(f)\geq j_{i-1}$},\,\, \vert f(x_n)\vert\geq\e\,\text{and}\,\vert f(x_m)\vert\geq\e\}
$$
and    $\vert f(x_{\min  L_{i}})\vert<\e$ for all $f$ of type I with $\ind(f)\geq j_{i}$.

Setting $l_{i}=\min L_{i}$ we get
$(l_{i},l_{k_0+1})\in\mathcal{D}_{i}$ for every $i=1,\dots,k_0$
and therefore there exists  $f_{i}$ of type I with $\ind(f_{i})\in
[j_{i-1},j_{i})$ such that $\vert
f_{i}(x_{l_{k_0+1}})\vert\geq \e$.

It follows $\frac{1}{\sqrt{k_{0}}}\sum_{i=1}^{k_0}\e_{i}f_{i}(x_{k_0+1})\geq \e\sqrt{k_0}>1$, a contradiction.
\end{proof}

\begin{proposition}\label{nol1}
Let $0<\e<10^{-10}$ and  $\xn$ be a normalized block basis such that for all $n_{1}<n_{2}<n_{3}$  there exists $\phi\in G_{\ell_2}$  such that
\begin{equation}\label{lo}
 \vert \phi(x_{n_i})\vert\geq 1-\e\,\,\,\,\text{for all i=1,2,3}.
\end{equation}
Then there exists  $j_{1}\in\N$ and  $L\in [\N]$ such that for all $n\in L$ there exists
$\phi_{n}=\sum_{i\in B_{n}}\lambda_{i,n}\phi_{i,n}$ with $\inde(\phi_{i,n})\subset\{1\dots,j_1\}$
for all $i\in B_{n}$ and $\sum_{i\in B_{n}}\lambda_{i,n}^{2}\leq 1$ such that
$$
\phi_{n}(x_{n})\geq 0.75.
$$
\end{proposition}
\begin{proof}
From Lemma \ref{ll1a} we get     $j_0\in\N$ such that
\begin{equation}\label{lo1}
\mbox{$\forall\,\phi=\sum_{i=1}^{d}\lambda_{i}\phi_{i}$ with $\inde(\phi_{i})\cap\{1,\dots,j_0\}=\emptyset$ and $\sum_{i\leq d}\lambda_{i}^{2}\leq 1$}
\Rightarrow
\vert\phi(x_1)\vert<\e
\end{equation}
Let $\delta>0$ such that $\delta j_0<\e$. From Lemma \ref{ll2} we get $M\in [\N]$  and $j_{1}\in \N$ such that
\begin{equation}\label{lo2}\mbox{
for every $f$ of type I with $\ind(f)>j_1$, $\vert f(x_n)\vert>\delta$ for at most one
$n\in M$.}
\end{equation}
Let $M_{1}=\{1\}\cup M$. After reordering we may assume that $M_{1}=\N$.

Let $n\in \N$ and for every $1<k<n$ consider the triple $(1,k,n)\in [\N]^{3}$. Let
$\phi_{k,n}=\sum_{s\in S_{k,n}}c_{s}\phi_{s}\in G_{\ell_2}$, where $\phi_{s}\in F(s)$, be
the functional we obtain  from \eqref{lo} for $x_{1}, x_{k}$ and $x_{n}$. We set
$$
\mbox{$J_{k,n}=\{s\in S_{k,n}: \inde(\phi_{s})>j_{0}\}$ and $G_{k,n}$ the complement of $J_{k}$.}
$$
From \eqref{lo1} we get $\vert\sum_{s\in J_{k,n}}c_{s}\phi_{s}(x_{1})\vert<10^{-10}$ and
therefore
\begin{equation}\label{lo3}
(\sum_{s\in G_{k,n}}c_{s}^{2})^{1/2}\geq\vert \sum_{s\in G_{k,n}}c_{s}\phi_{s}(x_{1})\vert>1-2/10^{10}.
\end{equation}
The disjointness of the index sets of the segments $s\in S_{k,n}$ yields that  $\# G_{k,n}\leq j_0$ for every $k$.

By \eqref{lo3} we get $\sum_{s\in G_{k,n}}c^{2}_{s}\geq (1-2/10^{10})^2$  and therefore
\begin{equation}\label{lo4}
\sum_{s\in J_{k,n}}c^{2}_{s}\leq 1-(1-2/10^{10})^2<4/10^{10}\Rightarrow
\|\sum_{s\in J_{k,n}}c_{s}\phi_{s}\|_{G_\xi}^*\leq 2/10^{5}.
 \end{equation}
From  \eqref{lo} and \eqref{lo4} it follows that for $r=k,n$
\begin{equation}\label{lo5}
\vert \sum_{s\in G_{k,n}}c_{s}\phi_{s}(x_{r})\vert\geq
1-10^{-10}-2/10^{5}>1-3/10^{5}.
\end{equation}
For every $s\in G_{k,n}$, if $\phi_{s}=\sum_{i\in V_{s}}\e_{i}f_{s,i}$
let $f_{s,k,n}$  be its first node  with the property $\ran(f_{s,i})\cap\ran (x_{n})\ne\emptyset$. Set
\begin{align*}
G_{k,n}^{0}&=\{s\in G_{k,n}: \ind(f_{s,k,n})\leq j_1\},
\\
G^{1}_{k,n}&=\{s\in G_{k,n}:\ind(f_{s,k,n})> j_1\,\text{and}\,\, \vert f_{s,k,n}(x_{n})\vert \leq\delta\},
\\
G_{k,n}^{2}&=\{s\in G_{k,n}\setminus G_{k,n}^{1}:\ind(f_{s,k,n})>j_1\,\text{and}\,\,\vert f_{s,k,n}(x_{k})\vert \leq\delta\}.
\end{align*}
From \eqref{lo2} it follows that for every $s\not\in(G_{k,n}^{0}\cup G_{k,n}^{1})$ it
holds that $\vert f_{s,k,n}(x_{k})\vert\leq\delta$ and hence  the sets  $G_{k,n}^i$,
$i=0,1,2$ give a partition of $G_{k,n}$.

For every $1<k<n$ we set
\begin{equation}\mbox{
$H_{k,n}=\begin{cases}
G_{k,n}^{0}\,\,&\mbox{if\,\, $\vert \sum_{s\in G_{k,n}^0}c_{s}\phi_{s}(x_k)\vert\geq 0.75\,\,\,\,(**)$}. \\
G^{1}_{k,n}\,\, &\mbox{if $(**)$ fails  and $\sum_{s\in G_{k,n}^{1}}c_{s}^2\geq\sum_{s\in G_{k,n}^{2}}c_{s}^{2}$}. \\
G_{k,n}^{2}\,\,&\mbox{\qquad otherwise.}
             \end{cases}
$}
\end{equation}
By Ramsey theorem passing to an infinite subset $N$ of $\N$ we may assume that there
exists  $i\in\{0,1,2\}$ such that for all $1<k<n\in N$, $H_{k,n}=G_{k,n}^{i}$.  Without
loss of generality we may assume that $N=\N$.

If  $H_{k,n}=G_{k,n}^{0}$ for all $k<n$ then  the proof is complete.
Next we  show that that the cases $H_{k,n}=G_{k,n}^{1}$ or $H_{k,n}=G_{k,n}^{2}$ are not possible.
\begin{claim}Assume that either $H_{k,n}=G_{k,n}^{1}$  for all $k<n$
or $H_{k,n}=G_{k,n}^{2}$ for all $k<n$.

Then there exists $N\in\N$ such that for all $n\in\N$ there exists a family $U_{n}$ of segments such that:

1) $\# U_{n}\leq N$.

2) For every $1<k<n$, $\sup\{\phi(x_k):\phi\in \overline{U}_{n}\}\geq\delta$.
\end{claim}
Let's see first how the conclusion of the claim makes impossible  $H_{k,n}=G_{k,n}^{1}$
or $H_{k,n}=G_{k,n}^{2}$.

Let $U_{n}=\{s_{1,n},\dots,s_{d_n,n}\}$ with $d_n\leq N$ for every $n>1$. From 2) we
obtain that for every $1<k<n$ there exist $\phi_{k}\in \overline{U}_{n}$   such that
$$
\phi_{k}(x_k)\geq \delta.
$$
By the compactness of the set $G_{\xi}^{\prime}$ passing to an
infinite subset of $\N$ we get that there exist finite segments
$s_{1},\dots,s_{N}$ such that
$w^{*}-\lim_{n\to\infty}s_{i,n}=s_{i}$ for every $i\leq N$.

Let $k\in\N$ such that $\supp (x_k)>\max\{\maxsupp s_{i}: i\leq N\}$.  Then since
$w^{*}-\lim_{n\to\infty}s_{i,n}=s_{i}$ we obtain that there exists  $n_0>k$ such that
$\supp s_{i,n_0}\cap\supp (x_k)=\emptyset$ for all $i\leq N$, a contradiction.

\medskip

We proceed now to the proof of the Claim.

Assume that $H_{k,n}\ne G_{k,n}^{0}$. It follows that  $\vert \sum_{s\in G_{k,n}^0}c_{s}\phi_{s}(x_k)\vert< 0.75$ and from  \eqref{lo5} we get
\begin{equation}
\vert \sum_{s\in G_{k,n}\setminus G_{k,n}^0}c_{s}\phi_{s}(x_k)\vert\geq 1-3/10^{5}-0.75=0.25-3/10^{5}.
\end{equation}
As in \eqref{lo3}-\eqref{lo4} we get that
\begin{align}
2\sum_{s\in H_{k,n}}c_{s}^{2}\geq
\sum_{s\in G_{k,n}\setminus G_{k,n}^{0}}c_{s}^{2}\geq
(0.25-3/10^{5})^2
\Rightarrow \sum_{s\in H_{k,n}}c_{s}^{2}>0.031.
\label{lo5aa}
\end{align}
From \eqref{lo5aa}  we obtain $\sum_{s\in G_{k,n}\setminus H_{k,n}}c_{s}^{2}\leq
0.97\Rightarrow\|\sum_{s\in G_{k,n}\setminus H_{k,n}}c_{s}\phi_{s}\|_{G_\xi}^*< 0.985
$
and therefore for $l=k,n$ the following holds:
\begin{align}\label{lo5c}
\vert \sum_{s\in H_{k,n}}c_{s}\phi_{s}(x_l)\vert&\geq
\vert \sum_{s\in G_{k,n}}c_{s}\phi_{s}(x_l)\vert
-\vert \sum_{s\in G_{k,n}\setminus H_{k,n}}c_{s}\phi_{s}(x_l)\vert
\\
&\geq
 1-3/10^{5}-0.985=0.0149.\notag
\end{align}
Assume now  that  $H_{k,n}=G^{1}_{k,n}$ for all $k<n$.

For every $s\in G_{k,n}^{1}$ we set $ \tilde{s}=s_{|[1,\maxsupp (f_{s,k,n})]}$.

Using finite induction we define sets  $U_{2,n}\subset U_{3,n}\subset\dots\subset U_{n-1,n}$ as follows:

Set $U_{2,n}=\{\tilde{s}: s\in G_{2,n}^{1}\}$. Let $k=3,\dots,n-1$
and assume that the set  $U_{k-1,n}$ has been defined.

If there exists
$\phi=\sum_{i}c_{i}\phi_{i}\in \overline{U}_{k-1,n}$, see Notation
\ref{not2},  with $\vert \phi(x_{k})\vert\geq\delta$ we set
$U_{k,n}=U_{k-1,n}$.

Assume that
\begin{equation}\label{laq}\mbox{for all $\phi=\sum_{i}c_{i}\phi_{i}\in \overline{U}_{k-1,n}$ it holds that $\vert \phi(x_{k})\vert<\delta$.}
\end{equation}
For the functional $\phi_{k,n}=\sum_{s\in G^{1}_{k,n}}c_{s}\phi_{s}$ we set
$$
E_{k,n}=\{s\in G^{1}_{k,n}: \tilde{s}\in U_{k-1,n}\}.
$$
Notice that the segments $\tilde{s}, s\in G^{1}_{k,n}\setminus E_{k,n},$ are different
from the segments in $U_{k-1,n}$ and therefore the functionals
$\phi_{s}^{1}=\phi_{s|(\maxsupp(f_{s,k,n}),+\infty)}, s\in G^{1}_{k,n}\setminus E_{k,n}$,
have disjoint index sets from the functionals in $\overline{U}_{k-1,n}$.

From \eqref{laq} we get
\begin{align}
\vert \sum_{s\in E_{k,n}}c_{s}\phi_{\tilde{s}}(x_{k})\vert=\vert \sum_{s\in E_{k,n}}c_{s}\phi_{s}(x_{k})\vert<\delta.
\end{align}
Combining the above inequality with \eqref{lo5} we get
\begin{align}\label{lo7}
(\sum_{s\in G_{k,n}\setminus E_{k,n}}c_{s}^{2})^{1/2}\geq \vert \sum_{s\in  G_{k,n}\setminus E_{k,n}}c_{s}\phi_{s}(x_{k})\vert>1-3/10^{5}-\delta>1-4/10^{5}.
\end{align}
while from \eqref{lo5c}
\begin{align}\label{lo7a}
\vert \sum_{s\in  G^{1}_{k,n}\setminus E_{k,n}}c_{s}\phi_{\tilde{s}}(x_{k})\vert>0.0149-\delta>0.0148.
\end{align}
From \eqref{lo7} following the  arguments we used to get \eqref{lo4} it follows
\begin{equation}\label{lo8}
(\sum_{s\in E_{k,n}}c_{s}^{2})^{1/2}\leq (1-(1-4/10^{5})^2)^{1/2}\leq 0.009.
\end{equation}
From the definition of the set $H_{k,n}$, using that $\vert
f_{s,k,n}(x_{n})\vert<\delta$, and \eqref{lo8} we obtain
\begin{align}\label{lo11} \vert
\sum_{s\in G^{1}_{k,n}\setminus E_{k,n}}c_{s}\phi_{s}^{1}(x_{n})\vert &\geq \vert
\sum_{s\in G^{1}_{k,n}}c_{s}\phi_{s}(x_{n})\vert
-
\sum_{s\in G^{1}_{k,n}}\vert c_{s}f_{s,k,n}(x_{n})\vert
-
\vert \sum_{s\in  E_{k,n}}c_{s}\phi_{s}(x_{n})\vert
\notag\\
&\geq 0.0149-j_0\delta-0.009>0.005.
\notag
\end{align}
Set $U_{k,n}=U_{k-1,n}\cup \{\tilde{s}: s\in G^{1}_{k,n}\setminus E_{k,n}\}$.

Note that in this case  $\#(U_{k,n}\setminus U_{k-1,n})\leq j_0$ and from \eqref{lo7a} we
obtain that for the functional $\phi_{k}=\sum_{s\in G_{k,n}^{1}\setminus
E_{k,n}}c_{s}\phi_{\tilde{s}}\in \overline{U}_{k,n}$ it holds that $\phi_{k}(x_{k})\geq
0.0148$.

Let $U_{n-1,n}$ be the set of the segments we get after we complete the above procedure  for $k=2,\dots,n-1$.

We show now that there is no   $n$ with $\# U_{n-1,n}>j_{0}(2+1/\delta_0^2)$,
$\delta_0=0.005$.

Indeed assume that there exists $n$ such that $\#
U_{n-1,n}>j_{0}(2+1/\delta_0^2)$. Since $\# U_{2,n}\leq j_{0}$ and
$\# (U_{k,n}\setminus U_{k-1,n})\leq j_0$ it follows that $U_{k-1,n}\subsetneqq U_{k,n}$ holds for at
least $t=1+1/\delta_0^2$ different  $k$'s. For every such $k$ we get  a
functional $\phi^1_{k}=\sum_{s\in G_{k,n}^{1}\setminus
E_{k,n}}c_{s}\phi^{1}_{s}$ such that $\phi^1_{k}(x_{n})>\delta_0$
and the  functionals $\phi^1_{k}$ have pairwise disjoint index sets.
It follows that
$$
\frac{1}{\sqrt{t}}\sum_{k}\phi^1_{k}(x_n)\geq \sqrt{t}\delta_0>1,
$$
a contradiction.

\medskip
\textit{Case 2}.  $H_{k,n}=G_{k,n}^{2}$ for every $1<k<n$.

For every $s\in G_{k,n}^{2}$ we set $\tilde{s}=s_{|[1,\minsupp f_{s,k,n})}$.

Using finite induction we define sets  $U_{2,n}\subset U_{3,n}\subset\dots\subset U_{n-1,n}$ as follows:

Set $U_{2,n}=\{\tilde{s}: s\in G_{2,n}^{2}\}$. Let $k=3,\dots,n-1$
and assume that  $U_{k-1,n}$ has been defined.

If there exists $\phi\in \overline{U}_{k-1,n}$ such that  $\vert \phi(x_{k})\vert>\delta$ we set $U_{k,n}=U_{k-1,n}$.

Assume that
\begin{equation}\label{laq3}\mbox{for all $\phi=\sum_{i}c_{i}\phi_{i}\inde\overline{U}_{k-1,n}$ it holds that $\vert \phi(x_{k})\vert<\delta$.}
\end{equation}
Set
$$
E_{k,n}=\{s\in G_{k,m}^{2}: \tilde{s}\in U_{k-1,n}\}.
$$
We get that the segments $\tilde{s}, s\in G_{k,n}^{2}\setminus E_{k,n}$ are different
from the segments in $U_{k-1,n}$ and therefore the functionals
$\phi_{s}^{2}=\phi_{s|[\minsupp f_{s,k,n},+\infty)}$ have disjoint index set from the
functionals in $\overline{U}_{k-1,n}$.

As in the previous case, see \eqref{lo7},\eqref{lo8}, we obtain that $(\sum_{s\in
E_{k,n}}c_{s}^{2})^{1/2}<0.009$.

From \eqref{lo5c} and  \eqref{laq3} we obtain \begin{align}\label{www1} \vert\sum_{s\in
G_{k,n}^{2}\setminus E_{k,n}}c_{s}\phi_{\tilde{s}}(x_{k})\vert& \geq\vert\sum_{s\in
G_{k,n}^{2}}c_{s}\phi_{s}(x_k)\vert- \vert \sum_{s\in
G_{k,n}^{2}}c_{s}f_{s,k,n}(x_k)\vert- \vert\sum_{s\in
E_{k,n}}c_{s}\phi_{\tilde{s}}(x_{k})\vert
\\
& =0.0149-j_0\delta-\delta>0.0148
\notag
\end{align}
Also we get
\begin{align*}
\vert\sum_{s\in G_{k,n}^{2}\setminus E_{k,n}}c_{s}\phi_{s}^{2}(x_{n})\vert&\geq
\vert\sum_{s\in G_{k,n}^{2}}c_{s}\phi_{s}(x_{n})\vert
 -\vert\sum_{s\in E_{k,n}} c_{s}\phi_{s}(x_{n})\vert
 \\
 &\geq 0.0149-0.009>0.005
\end{align*}
We set $U_{k,n}=U_{k-1,n}\cup\{\tilde{s}: s\in G^{2}_{k,n}\setminus E_{k,n}\}$.
Note that  in this case  $\#(U_{k,n}\setminus U_{k-1,n})\leq j_0$ and from \eqref{www1} we get that for the functional $\phi_{k}=\sum_{s\in G_{k,n}^{2}\setminus E_{k,n}}c_{s}\phi_{\tilde{s}}\in \overline{U}_{k,n})$ it holds that $\phi_{k}(x_{k})\geq 0.0148$.

Let $U_{n-1,n}$ be the set of the segments we obtain after we complete the above
procedure. Assuming that there exists  $n$ such that $\# U_{n-1,n}>j_0(2+1/\delta_0^2)$
as in the case where $H_{k,n}=G_{k,n}^{1}$ we get a contradiction.
\end{proof}

\begin{theorem}\label{pr220}
 The space $X_{G_{\xi}}$ does not contain a normalized sequence generating an $\ell_{1}$-spreading model.
\end{theorem}
In the proof of the theorem we shall use the following well known fact.
\begin{Fact}\label{fact1} Let $X$ be a Banach space and $(x_n)_{n}$ be a normalized sequence generating an $\ell_{1}$-spreading model with constant $c$.
Then for every $\e>0$ there  there exists a normalized sequence
$(y_n)_{n\in\N}$ generating an $\ell_{1}$-spreading model with constant $(1-\e)^{-1}$. In
particular $(y_{n})_{n}$ can be chosen such that each $y_{n}$ is a
normalized convex combination of $\xn$.
\end{Fact}
\begin{proof}[Proof of Theorem \ref{pr220}]
Let $\xn$ be a sequence which generates an $\ell_{1}$-spreading model with constant $c$.
Since the space $X_{G_{\xi}}$ does not contain $\ell_{1}$ by
standard arguments passing  to a subsequence and taking the
differences  we may assume that  $\xn$ is a block sequence and
generates a spreading model with constant $(1-\e)^{-1}$, $\e<10^{-10}$.

By proposition \ref{nol1} we get that there exist  $j_{1}\in\N$ and $L\in[\N]$ such that
for all $n\in L$ there exists $\phi_{n}\in G_{\ell_2}$ with index set contained in
$\{1,\dots,j_1\}$ such that $\phi_{n}(x_n)\geq 0.75$.

It follows that for every $n\in L$ and  every $\phi\in G_{\ell_2}$
with $\inde (\phi)\subset\{j_1+1,\dots,\}$ it holds that $\vert
\phi(x_{n})\vert<0.75$.

Indeed if there exists such  $\phi$ with $\vert\phi(x_n)\vert\geq
0.75$ then the functionals $\phi$ and $\phi_{n}$ have disjoint
index sets and hence $\psi=\dfrac{\phi+\phi_{n}}{\sqrt{2}}\in K$.
It follows that $\psi(x_n)\geq \frac{0.75+0.75}{\sqrt{2}}>1$, a
contradiction.

Let $n_0\in\N$ with  $n_0>n_{j_1}$. Let $x=\frac{1}{n_0}\sum_{i=1}^{n_0}x_{l_{i}}$ with $n_{0}\leq
 l_1<\dots<l_{n_0}\in L$.

Since $\xn$ is assumed to generate an $\ell_{1}$-spreading model with constant $(1-\e)^{-1}$ it
follows that  $\norm[x]_{G_\xi}\geq 1-\e$.
Let $\phi\in K$ be a functional which norms $x$. It readily follows that  $\phi=\sum_{i=1}^{d}\lambda_{i}\phi_{i}\in G_{\ell_2}$.

For every $i$ let $\phi_{i}^{1}$ be the part of $\phi_{i}$ with the
index set contained in $\{1,\dots,j_{1}\}$ and $\phi_{i}^{2}$  be
the part  with index set contained in $\{j_{1}+1,j_{1}+2,\dots\}$.
Let  $\phi_{1}=\sum_{i=1}^{d}\lambda_{i}\phi_{i}^{1}$ and
$\phi_{2}=\sum_{i=1}^{d}\lambda_{i}\phi_{i}^{2}$.

By the above note we get that $\phi_{2}(x_{l_i})\leq 0.75$ for all $i$ and hence  $\phi_{2}(x)\leq 0.75$.

Also  for every $f$ of type I with  $\ind(f)\leq j_1$ it follows that $\vert
f(x)\vert\leq\frac{\#\supp(f)}{n_{0}}\leq w(f)^{-1}$. Hence
$$\vert \phi^{1}(x)\vert\leq  \sum_{i=1}^{d}\sum_{k\in\inde(\phi_{i}^{1})}m_{k}^{-1}\leq\sum_{k=1}^{\infty}m_{k}^{-1}<0.1.
$$
Therefore  $\vert \phi(x)\vert\leq\vert \phi_{1}(x)\vert+\vert\phi_{2}(x)\vert<0.85$, a contradiction.
\end{proof}

\section{The space $X_{\xi}$ as an extension of the space $X_{G_{\xi}}$}\label{sec7}
In this section we define the space $X_{\xi}$ by its norming set $K_{\xi}$. We introduce the key
ingredient for the definition of $K_{\xi}$ - attractor sequences (Def. \ref{d7.1}), and the basic
tool in estimating the norm in $X_{\xi}$ - a tree-analysis of a functional from $K_{\xi}$ (Def.
\ref{def3.3}).

Let $(\Lambda_{i})_{i\in\N}$ be a partition of $\N$ into pairwise disjoint  infinite
sets. Let $\Q_{s}$ denote the set of all finite sequences $(f_{1},\dots, f_{d})$ such
that for all $i$, $f_{i}\in c_{00}(\N)$, $f_{i}\ne 0$, $f_{i}(n)\in\Q$ for all $n\in\N$
and $f_{1}<f_{2}<\dots<f_{d}$.

We fix a partition $N_{1},N_{2}$ of  $\N$.  Let $\sigma:\Q_{s}\to N_{2}$ be an injective  function such that
$$
m^{1/2}_{2\sigma(f_{1},\dots,f_{d})}>\max\{\vert f_{i}(e_{l})\vert^{-1}: l\in\supp f_{i},i\leq d\}\maxsupp f_{d}.
$$
Such an injective function exists since the set $\Q_{s}$ is countable.

Let $K_{\xi}$ be the minimal subset of $c_{00}(\N)$ satisfying the following conditions
\begin{enumerate}
\item  $K_{\xi}$ is symmetric i.e. if $f\in K_{\xi}$ then $-f\in K_{\xi}$, $K_\xi$ is closed under the restriction of its elements to intervals of $\N$ and $G_{\xi}\subset K_{\xi}$.
\item $K_{\xi}$ is closed under $(\mc{A}_{n_{2j}},m_{2j}^{-1})$-operations.
\item $K_{\xi}$ is closed under $(\mc{A}_{n_{2j-1}},m_{2j-1}^{-1})$-operations on attractor sequences.
\item $K_{\xi}$ is closed under the  operation $\sum_{i\in A}\lambda_{i}f_{i}$ whenever
\begin{enumerate}
\item  $f_{i}$ is the  result of an $(A_{n_{j_i}},
m_{j_i}^{-1})$-operation and $n_{j_i}\ne n_{j_{k}}$ for every
$i\ne k\in A$ \item $(\lambda_{i})_{i\in A}\in B_{\ell_2}\cap
[\Q]^{<\infty}$.
\end{enumerate}
\item $K_{\xi}$ is rationally convex.
\end{enumerate}
In order to complete the definition of  the set $K_{\xi}$ we have to
define the attractor sequences.
\begin{definition}\label{d7.1}
A finite sequence $(f_{1},\dots,f_{d})$ is said to be a $n_{2j-1}$-attractor sequence
provided that
\begin{enumerate}
\item $(f_{1},\dots,f_{d})\in \Q_{s}$ and $f_{i}\in K_{\xi}$ for all
$i=1,\dots,d\leq n_{2j-1}$. \item  The functional  $f_{1}$ is the
result of an $(\mc{A}_{n_{2j_1}},m_{2j_1}^{-1})$operation on a
family of functionals of $K_{\xi}$ for some $j_{1}\in N_{1}$ with
$m_{2j_1}>n^3_{2j-1}$. Also for  every $i=2,\dots,n_{2j-1}/2$,
$f_{2i-1}$ is the result of an
$(\mc{A}_{n_{2\sigma(f_{1},\dots,f_{2i-2})}},m_{2\sigma(f_{1},\dots,f_{2i-2})}^{-1})$
operation on a family of functionals of $K_{\xi}$. \item
$f_{2i}=e^{*}_{\lambda_{2i}}$ for some
$\lambda_{2i}\in\Lambda_{\sigma(f_{1},\dots,f_{2i-1})}$.
\end{enumerate}
\end{definition}
We say that $f\in K_{\xi}$ is of type I if it is a result of some
$(\mc{A}_{n_j},m_{j}^{-1})$ operation. In this case we set $w(f)=m_{j}$ and $\ind(f)=j$.

We say that $f\in K_{\xi}$ is of type II if
$f=\sum_{i=1}^{n}\lambda_{i}f_{i}$ for some
$(f_{i})_{i=1}^{n}\subset K_{\xi}$ with $w(f_{i})\ne w(f_{j})$  for all
$i\ne j$ and $\sum_{i=1}^{n}\lambda_{i}^{2}\leq 1$.

We say that  $f$ is of type III if it is a rational combination of elements of $K_{\xi}$.
\begin{remarks}\label{r3.2}
a) Using the partition of $\N=\cup_{i}\Lambda_{i}$ it  follows
that the set of the attractor sequences has a tree structure i.e. if
$(f_{i})_{i=1}^{n_{2j-1}}$, $(g_{i})_{i=1}^{n_{2k-1}}$ are two
attractor sequences either $g_{i}\ne f_{r}$ for every  $i,r$ or
there exists  $i_0\leq\min\{n_{2j-1},n_{2k-1}\}$ such that
$f_{i}=g_{i}$ for all $i\leq i_0-1$ and $f_{i}\ne g_{r}$ for all
$i,r\geq i_0$. In particular $w(f_{2i-1})\ne w(g_{2r-1})$ for all
$2i-1,2r-1>i_0$ and $f_{2i}\ne g_{2r}$ for all $2i,2r>i_0$.

b) Note that in the definition of the norming set when we take
functionals $\sum_{i=1}^{n}\lambda_{i}f_{i}$ with
$\sum_{i=1}^n\lambda_{i}^{2}\leq 1$ we  require $w(f_{i})\ne
w(f_{j})$ but do not require
$\supp(f_i)\cap\supp(f_{j})=\emptyset$.  This is will be essential
in order to get that no sequence  generates an  $\ell_{1}$-spreading
model.
\end{remarks}
\begin{definition}\label{def3.3}[The tree $T_{f}$ of a functional $f\in K_{\xi}$] Let $f\in K_{\xi}$. By a tree of $f$ we mean a finite family $T_{f}=(\fa)_{\alpha\in\mca}$ indexed by a finite tree $\mca$ with a unique root $0\in\mca$ such that the following holds
\begin{enumerate}
 \item $f_{0}=f$ and $\fa\in K_{\xi}$ for all $\alpha\in\mca.$
\item An $\alpha\in\mca$ is maximal if and only if $\fa\in G_{\xi}$.
\item For every $\alpha\in\mca$ not maximal, denoting by $S_{\alpha}$ the set of the immediate successors of $\alpha$ one of the following holds
\begin{enumerate}
\item There exists $j\in\N$ such that $\# S_\alpha\leq n_{j}$ and $\fa$ is the result  of
the  $(\mc{A}_{n_j},m_{j}^{-1})$-operation in the set $\{f_{\beta}: \beta\in S_{a}\}$. In
this case we say that the weight of $f$ is $w(\fa)=m_{j}$. \item  $\fa=\sum_{\beta\in
S_{\alpha}}\lambda_{\beta}f_{\beta}$ where $\sum_{\beta\in
S_{\alpha}}\lambda_{\beta}^{2}\leq 1$  the functionals $\fb$ are of type I and have
different weights. \item  $\fa$ is a rational convex combination of the elements
$f_{\beta}$, $\beta\in S_{\alpha}$. Moreover for every  $\beta\in S_{\alpha}$ $\ran
f_{\beta}\subset\ran \fa$
\end{enumerate}
\end{enumerate}
\end{definition}
The order $o(\fa)$ for each $\alpha\in\mca$ is also defined by backward induction as follows.

If $\fa\in G_{\xi}$ then $o(\fa) = 1$, otherwise $o(\fa)=1 + \max\{o(\fb):\beta \in S_{\alpha}\}$.

The order $o(T_{f})$ of the aforementioned tree is defined to be equal to $o(f_0)$.
\begin{definition} The order $o(f)$ of an $f \in K_{\xi}$, is defined as
                   $$o(f )=\min\{o(T_f ) :\text{$T_{f}$ is a tree of $f$}\}.$$
\end{definition}

Every functional $f\in K_{\xi}$  admits a tree analysis not necessarily unique.
\begin{definition} We define $X_\xi=\overline{(c_{00}(\N),\norm_{K_{\xi}})}$ 
and we   denote the norm $\norm_{K_{\xi}}$ by $\norm$.
\end{definition}
\begin{definition}
 Let $k\in\mathbb{N}$. A vector $x\in c_{00}(\N)$ is said to be
 a $C-\ell_{1}^{k}$ average if there
 exists $x_{1}<\ldots<x_{k}$, $\norm[x_{i}]\leq C\norm[x]$ and
 $x=\frac{1}{k}\sum_{i=1}^{k}x_{i}$. Moreover, if $\norm[x]=1$ then
 $x$ is called a normalized $C-\ell_{1}^{k}$ average.
\end{definition}
\begin{lemma}[ \cite{Sh}, (or \cite{GM}, Lemma 4)] \label{l4.5}
Let $j\geq 1$, $x$ be a $C-\ell_{1}^{n_{j}}$ average.  Then for
every $n\leq n_{j-1}$ and
 every
 $E_{1}<\ldots<E_{n}$,
we have that $$ \sum_{i=1}^{n}\norm[E_{i}x]\leq
C(1+\frac{2n}{n_{j}})
 <\frac{3}{2}C. $$
\end{lemma}
\begin{proposition}\label{p4.6}
For every block sequence $(y_{\ell})_{\ell}\subset X_\xi$, $\e>0$ and every $k\ge m_2$
there exists $x\in\langle (y_{\ell})_{\ell}\rangle$ which is a normalized $2-\ell_{1}^k$
average.

In particular for every $k\in\N$ there exists $n\in\N$ such that  for every finite block subsequence $(y_{\ell_{i}})_{i=1}^{n}$ of $(y_{\ell})_{\ell}$ there exists $(z_{i})_{i=1}^{k}$, a normalized block sequence of $(y_{\ell_i})_{i=1}^{n}$ such that $\norm[z_1+\dots+z_k]\geq k/2$.
\end{proposition}
Lemma II.22 \cite{ATod}
give us the existence of such average.
\begin{proposition}\label{ss}
$X_\xi$ is a strictly singular extension of $X_{G_{\xi}}$, i.e. the identity map
$X_{G_\xi}\to X_\xi$ is strictly singular.
\end{proposition}
\begin{proof}
Let $Y$ be a block subspace of  $X_\xi$ and $j\in\N$. By the Proposition \ref{p4.6} we
obtain a sequence $\xn[y]\subset Y$ of normalized $\ell_{1}^{k_n}$ averages and by
Proposition \ref{ris-sss} for any fixed $j>1$ we get $L\in [\N]$ such that for every
$g\in G_{\xi}$
$$
\{n\in L: \vert g(y_n)\vert\geq m_{2j}^{-2}\}\leq n_{2j-1}.
$$
 It follows that
$$
\|\frac{1}{n_{2j}}\sum_{i=1}^{n_{2j}}y_{i}\|_{G_{\xi}}
 \leq
\frac{n_{2j-1}}{n_{2j}}+\frac{1}{m_{2j}^2}\leq \frac{2}{m_{2j}^{2}}
$$
 while
$\norm[\frac{1}{n_{2j}}\sum_{i=1}^{n_{2j}}y_{i}] \geq \frac{1}{m_{2j}}$. Since $j$ was
arbitrarily chosen we have a strictly singular extension.
\end{proof}
\section{The auxiliary space}\label{sec8}
Now we define for any $j_0>1$ the auxiliary space $X_{W_{j_0}}$ and discuss construction
and relation between its norming sets: $W_{j_0}$ and $W_{j_0}^{\prime}$. Next in Subsection
\ref{subsec} we estimate the supremum norm of any norming functional from $W_{j_0}^{\prime}$
(Lemma \ref{linfty}) and its action on a suitable average of $(e_n)_n$ of length $n_{j_0}$  (Lemma \ref{42}).

Let $j_0>1$ be fixed. We set $C_{j_0}=\{\sum_{i\in F}\pm e_{i}^{*}: F\in\mc{A}_{n_{j_0-1}}\}$.
Let $W_{j_0}$ be the minimal subset of $c_{00}(\N)$ satisfying the following conditions
\begin{enumerate}
\item[1)] $W_{j_0}$ contains the set  $C_{j_0}$.
\item[2)] $W_{j_0}$ is closed under
$(\mca_{2n_j}, m_{j}^{-1})$-operations for every $j\in\N$.
\item[3)] $W_{j_0}$ is closed under the operation $\sum_{i\in A}\lambda_{i}f_{i}+\sum_{i\in
B}\lambda_{i}e_{t_i}$ whenever  $A\cap B=\emptyset$, $t_i\ne
t_{j}$ for all $i\ne j\in B$, $f_{i}$ is the result of an
$(\mca_{2n_{j_i}}, m_{j_i})$-operation,  $m_{j_i}\ne m_{j_k}$ for all
$i\ne k\in A$ and  $\sum_{i\in A\cup B}\lambda_{i}^{2}\leq 1$.
\item[4)] $W_{j_0}$ is rationally convex.
\end{enumerate}
The auxiliary space is  the space $X_{W_{j_0}}=\overline{(c_{00}(\N),\norm_{W_{j_0}})}$.

For every functional $f\in W_{j_0}$ which is the result of an $(\mca_{2n_j},
m_{j}^{-1})$-operation we set $w(f)=m_{j}$ and we say that $f$ is of type I. Every
functional of the form $\sum_{i\in A}\lambda_{i}f_{i}+\sum_{i\in B}\lambda_{i}e_{n_i}\in
W_{j_0}$ is called of type II. Every functional which is a convex combination of elements
of $W_{j_0}$ is called of type III.

Consider the minimal set $W_{j_0}^{\prime}$ which satisfies Properties 1)-3) in the
definition of $W_{j_0}$. Then following the arguments of Lemma 3.15 \cite{ATo}, every
$f\in W_{j_0}$ of type I with $w(f)=m_{j}$  can be written as a convex combination
$\sum_{i}\lambda_{i}f_{i}$ of functionals  from the set  $W_{j_0}^{\prime}$ with
$w(f_{i})=m_{j}$ for every $i$. Hence  in order to get an upper estimate for functionals
of type I from the set  $W_{j_0}$ it is enough to get an upper estimate for the
functionals from the set $W_{j_0}^{\prime}$.

An alternative definition of the norming set $W_{j_0}^{\prime}$ is the following:

Set $W_{0}=C_{j_0}$ and assume that $W_{n-1}$ has been defined.

We set
$$W_{n}^{1}=\bigcup_{j}\{m_{j}^{-1}\sum_{i=1}^{d}f_{i}: (f_{i})_{i=1}^{d}\,\text{is $\mc{A}_{2n_{j}}$-admissible subset of $W_{n-1}$}\}
$$
and
\begin{align*}
W_{n}^{2}=\{\sum_{i\in A}&\lambda_{i}f_{i}+\sum_{i\in B}\lambda_{i}e_{t_i}: \sum_{i\in A\cup B}\lambda_{i}^{2}\leq 1, \text{$t_{i}\ne t_{j}$ for all $i\ne j\in B$, $A\cap B=\emptyset$,}\\
& \text{$f_{i}\in W_{n-1}$ is of type I for all $i\in A$ with $w(f_{i})\ne w(f_{j})$ for
all $i\ne j\in A$}\}.
\end{align*}
We set $W_{n}=W_{n}^{1}\cup W_{n}^{2}$. Then $W_{j_0}^{\prime}=\cup_n W_n$.

Every $f\in W_{j_0}^{\prime}$ admits a tree analysis $(\fa)_{\alpha\in\mc{T}}$ satisfying
properties similar to Definition \ref{def3.3}.
\begin{remarks} Let $f\in W_{j_0}^{\prime}$ and $(\fa)_{a\in\mc{T}}$ be a tree analysis of $f$.

a) If $\alpha$ is a terminal node then $\fa\in  C_{j_0}$.  In particular if $\alpha\in
S_{\beta}$ and $\fb$ is of type II then  $\fa=\pm e_{k}^{*}$ for some $k$.

b) If $\alpha,\beta$ are different maximal nodes it is not necessary true that $\supp(\fa)\cap\supp(\fb)=\emptyset$.

c) A consequence of b) is that  it does not hold that $\norm[f]_{\infty}\leq 1$ for every $f\in W_{j_0}^{\prime}$. We show  in Lemma \ref{linfty} that there is a constant $c_1>1$ such that  $\norm[f]_{\infty}\leq c_{1}$ for every $f\in W_{j_0}^{\prime}$.
\end{remarks}
\subsection{Estimates on the auxiliary space}\label{subsec}
Set $c_0=\sum_{j=1}^{\infty}m_{j}^{-2}$ and $c_{1}=(\sum_{n=0}^{\infty}c_0^n)^{1/2}$.
\begin{lemma}\label{linfty} For every $f\in W_{j_0}^{\prime}$ it holds that
$\vert f(e_t)\vert\leq\begin{cases}\frac{c_1}{w(f)}\,\,&\text{if $f$ is of type I}\\
c_1 &\text{if $f$ is of type II}.
\end{cases}
$
\end{lemma}
\begin{proof}
 We prove by induction that for every $f\in W_{n}$ the following holds:
\begin{equation}\label{key1}
\vert f(e_t)\vert\leq u_{f}(1+c_0+\dots+c_{0}^n)^{1/2}
\end{equation}
where  $u_{f}=w(f)^{-1}$ if $f$ is of type  $I$ and $u_{f}=1$ otherwise.

Assume that for every  $f\in W_{n-1} $ it holds that $\vert f(e_t)\vert\leq u_{f}(1+c_0+\dots+c_{0}^{n-1})^{1/2}$ which is certainly true for $n=1$.

If $f=w(f)^{-1}\sum_{i=1}^{d}f_{i}\in W_{n}^{1}$ then there
exists unique  $i\leq d$ such that $t\in\supp(f_{i})$. It follows
from  the inductive hypothesis that
$$\vert f(e_t)\vert\leq \frac{1}{w(f)}\vert f_{i}(e_t)\vert\leq \frac{1}{w(f)}(1+c_0+\dots+c_{0}^{n-1})^{1/2}.
$$
If $f=\sum_{i\in A}\lambda_{i}f_{i}+\sum_{i\in
B}\lambda_{i}e^{*}_{i}\in W_{n}^{2}$ then either $t\not\in B$ or
there exists unique $i_{t}\in B$ such that $e_{t}=e_{i_t}$.
Assume that the latter holds. Setting  $A_{t}=\{i\in A: t\in\supp
f_{i}\}$ from the inductive hypothesis and H\"older inequality  we get
\begin{align*}
\vert f(e_t)\vert &\leq \sum_{i\in A_{t}}\vert \lambda_{i}f_{i}(e_t)\vert+\vert \lambda_{i_t}\vert
\leq
 (1+\sum_{i\in A_{t}}w(f_i)^{-2}(1+c_0+\dots+c_{0}^{n-1}))^{1/2}
 \\
 &\leq (1+(1+c_0+\dots+c_{0}^{n-1})\sum_{j}m_{j}^{-2})^{1/2}
 =(1+c_0+\dots+c_{0}^{n})^{1/2}.
\end{align*}
This proves \eqref{key1} and completes the proof  of the lemma.
\end{proof}
Recall that in the definition of the space we have chosen $n_{j}=(2n_{j-1})^{s_{j-1}}$ where $2^{s_{j-1}}=m_{j}^{3}$.
\begin{lemma}\label{linfty2} Let $f\in W_{j_0}^{\prime}$ such that $\vert f(e_{t})\vert>\frac{3}{m_{j_0}}$ for every $t\in\supp(f)$. Then
$$\#\supp(f)\leq (2n_{j_0-1})^{s_{j_0-1}-2}.
$$
 \end{lemma}
This lemma is the key ingredient of our evaluations in the auxiliary space. Its proof is
given in four steps. As usually we shall consider the tree analysis
$(\fa)_{\alpha\in\mc{T}}$ of a functional $f$ in $W_{j_0}^{\prime}$. In the first two
steps we shall show that we may assume that the height of the tree is less or equal to
$s_{j_0-1}-3$, and for all functionals $\fa$  of type I in a tree analysis of $f$ it
holds that $w(\fa)<m_{j_0}$. This will enable us to use an  argument similar to  Lemma A4
in \cite{AAT} and get the result.
\begin{proof}
Let  $(\fa)_{\alpha\in\mathcal{T}}$  be a tree analysis of the functional $f$.
For every $t\in \supp(f)$ we set
\begin{align*}
D_1=\{\alpha\in\mc{T}:\text{$\alpha$ is a minimal node} \text{ with $\fa$ of type I and
$w(\fa)\geq m_{j_0}$}\}.
\end{align*}
Let us observe that the set $D_1$ consists of incomparable nodes.
We set
$$
M_{1}=
\{\gamma\in\mc{T}: \text{$\gamma$ is a maximal  node and there exists $\alpha\in D_1$ with $\alpha\prec\gamma$}\}
$$
and also
$$
M_{2}=\{\gamma\in\mc{T}:\text{$\gamma$ is a maximal  node and $\gamma\not\in M_{1}$}\}.
$$
The definition of $M_{2}$ yields  that for every $\gamma\in M_{2}$ and every
$a\prec\gamma$ with $\fa$ of type I, $w(\fa)<m_{j_0}$.

We denote by $f_{1}$ the functional that we get following backwards the tree analysis of $f$ with
maximal nodes the nodes of $M_{1}$ and $f_{2}$
the corresponding one with maximal nodes those of $M_{2}$. Observe that $f=f_1+f_2$.

We also set $\mathcal{T}_{1}=\{\beta:\text{there exists $\alpha\in D_1$ with
$\beta\preceq\alpha$}\} $ which is a complete subtree of the tree $\mc{T}$  and
$(g_{\alpha})_{\alpha\in\mc{T}_{1}}$ the tree analysis of the functional $f_{1}$. Let us
observe that  every maximal node $\alpha$ of $\mc{T}_{1}$ belongs to the set
$D_1$ and $\ga=\fa$.
\begin{sublemma}
For the functional $f_{1}$, its tree analysis $(\ga)_{\alpha\in\mc{T}_{1}}$ and $t\in\N$
the following hold:
\begin{enumerate}
\item for every  $\alpha\in D_1$  we have $\vert \ga(e_t)\vert\leq\frac{c_{1}}{m_{j_0}}$.
\item for every non-maximal node $\alpha\in\mc{T}_{1}$  for which $\ga$ is of type I we
have
\begin{center}
$\vert\ga(e_t)\vert<\dfrac{9c_{1}}{8w(\ga)m_{j_0}}$.
\end{center}
\item For all $\alpha\in\mc{T}_{1}$ non-maximal with $\ga=\sum_{\beta\in
S_{\alpha}}\lambda_{\beta}\gb$ of type II  we have
\begin{center}$\vert \ga(e_t)\vert \leq\dfrac{9c_{1}}{8m_{j_0}}$.
\end{center}
\end{enumerate}
\end{sublemma}
\begin{proof}
(1) follows from Lemma \ref{linfty} and  the definition of the set $M_{1}$.

(2) and (3) are proved inductively.   Let $\alpha\in\mathcal{T}_1$ and assume the result
for all $\alpha\prec\beta$.

If $\ga=w(\ga)^{-1}\sum_{\beta\in S_{\alpha}}\gb$  it is  clear that (2) holds.

We pass now to prove (3). Assume that  $\ga=\sum_{\beta\in S_{\alpha}}\lambda_{\beta}\gb$. Then $\vert \ga(e_t)\vert
=\vert\sum_{\beta\in S_{a}}\lambda_{\beta}\gb(e_t)\vert$.

If   $\beta\in S_{\alpha}$  is not  maximal then $\gb$ is of type I and the inductive
assumptions yields
\begin{equation}\label{bv0}
\vert \gb(e_t)\vert\leq\frac{9c_{1}}{8w(\gb)m_{j_0}}.
\end{equation}
Let  $\beta\in S_{\alpha}$ be a maximal node, namely $\beta\in D_1$. From the definition of the set $M_{1}$ it follows
that  $\gb$ is of type  $I$ and $w(\gb)=m_{j_0+k}$ for some $k\geq
0$.  Lemma \ref{linfty} yields
\begin{equation}\label{bv}
\vert\gb(e_t)\vert\leq\frac{c_{1}}{w(\gb)}(\leq\frac{1}{m_{j_0}m_{j_0}^{k}} \mbox{ if $k>0$}).
\end{equation}
We also observe  that there exists at most one $\beta\in S_{\alpha}$  with
$w(\gb)=m_{j_0}$. Without loss of generality we assume that there does exist $\beta_0\in
S_{\alpha}$ with $w(g_{\beta_0})=m_{j_0}$.

Set $S_{\alpha}^{1}=\{\beta\in S_{\alpha}:\text{$\gb$ of type I with $w(\gb)<m_{j_0}$}\}$,

$S_{\alpha}^{2}=\{\beta\in S_{\alpha}:\text{$\gb$ of type I with $w(\gb)>m_{j_0}$}\}\cup\{\beta\in S_{\alpha}: \beta\,\text{is maximal}\}$.

It follows from \eqref{bv0}, \eqref{bv} and the fact that $w(\gb)\ne w(g_{\gamma})$ for every $\beta\ne\gamma\in S_{\alpha}$, that
\begin{align*}
\vert \ga(e_t)\vert & \leq\sum_{\beta\in S_{\alpha}^{1}}\vert\lambda_{\beta}g_{\beta}(e_t)\vert+
 \vert \lambda_{\beta_0}g_{\beta_0}(e_t)\vert+
 \sum_{\beta\in S_{\alpha}^{2}}\vert\lambda_{\beta} g_{\beta}(e_t)\vert
 \\
 &\leq
\vert \lambda_{\beta_0}g_{\beta_0}(e_t)\vert+
 \sum_{\beta\in S_{\alpha}^{1}}\frac{2c_{1}\vert\lambda_{\beta}\vert}{m_{j_0}w(\gb)}+
\sum_{\beta\in S_{\alpha}^{2}}\frac{c_{1}\vert\lambda_{\beta}\vert}{w(\gb)}
\\
&\leq
\vert g_{\beta_0}(e_t)\vert+
 \sum_{i<j_0}\frac{2c_{1}}{m_{j_0}m_{i}}
+
 \sum_{k\geq 1}\frac{1}{m_{j_0}^{k+1}}
\leq
\frac{c_{1}}{{m_{j_0}}}+
 \frac{4c_{1}}{m_{j_0}m_{1}}\leq \frac{9c_{1}}{8m_{j_0}}.
\end{align*}
\end{proof}
\begin{remark}\label{re4.5} Let $(\fa)_{\alpha\in\mc{T}}$ be a tree analysis of $f_1$ such that for all $\fa$ of type I it holds that $w(\fa)\ne m_{j_0}$. The proof of the above sublemma yields that in this case $\vert f_{1}(e_{t})\vert\leq \frac{9c_{1}}{8m_{j_0}^{2}}$.
\end{remark}
Let us observe that, since $f=f_1+f_2$, the previous sublemma yields that $\supp(f_2)=\supp(f)$.
So we may assume that $f=f_2$. Let $(\fa)_{\alpha\in\mc{T}}$ be the tree analysis of $f$ with all maximal nodes belonging to $M_{2}$.  For $\beta\in\mc{T}$ we denote by
\begin{center}
$\vert\beta\vert=\#\{\alpha\in\mc{T}:\alpha\prec\beta\}$
\end{center}
the order of $\beta$.

We set $d_{j_0}=s_{j_0-1}-3$ and we observe that  $c_{1}2^{-d_{j_0}}<m^{-2}_{j_0}$.
We also set \begin{align*}
D=\{\alpha:\alpha\,
\text{is a maximal node of $\mc{T}$ with} &\text{$\vert \alpha\vert\leq d_{j_0}$ and if  $\alpha\in S_{\beta}$,}\\
&\text{$\fb$ is of type II then $\vert\lambda_{\alpha}\vert\geq m_{j_0}^{-1}$}\}.
\end{align*}
With the next  two sublemmas  we  show that $\supp (f)=\cup_{\alpha\in D}\supp (\fa)$.
\begin{sublemma}\label{sl4.6}
Let $t\in\supp (f)$ be such that there is no  maximal node $\beta$
with $\vert \beta\vert \leq d_{j_0}$ and $t\in \supp(\fb)$. Then it
holds $\vert f(e_t)\vert\leq c_{1}2^{-d_{j_0}}$.
\end{sublemma}
\begin{proof}
Let $t\in\supp(f)$. By induction we shall show the following:
for every $\gamma\in\mathcal{T}$ such that $\vert \gamma\vert =d_{j_0}-j$, $j=0,1,\dots,d_{j_0}$
we have that
\begin{equation}
\vert \fg (e_t)\vert\leq c_{1}2^{-j}u_{\gamma},
\end{equation}
where $u_{\gamma}=m^{-1}_{r-1}$ if  $\fg$ is of type I and $w(\fg)=m_{r}$, (we make the
convention that $m_0=2^{-2}$),\, $u_{\gamma}=1$  if $\fg$ is of type II.

First we observe that Lemma \ref{linfty} yields $\vert \fb(e_t)\vert\leq c_{1}$ for every maximal node $\beta\in \mc{T}$.

Assume that the result holds for all $\delta\succ \gamma$ with  $\vert \delta\vert=d_{j_0}-(j-1)$
and $\vert \gamma\vert=d_{j_0}-j$.

If  $\fg=m_{r}^{-1}\sum_{\delta\in S_{\gamma}}\fd$ is of type I then since there exists
unique $\delta\in S_{\gamma}$ with $t\in\supp \fd$, from the inductive hypothesis we
obtain $$ \vert\fg(e_t)\vert\leq m_{r}^{-1}\frac{c_{1}}{2^{j-1}}\leq
\frac{c_{1}}{m_{r-1}2^{j}}.
$$
If $\fg=\sum_{\delta\in S_{\gamma}}\lambda_{\delta}\fd$ is of type II  then each $\fd$ is
of type I. This follows from our assumption  that every $\delta$ with
$\vert\delta\vert\leq d_{j_0}$ is not  a maximal node of the tree $\mc{T}$.

Hence setting $A_{t}=\{\delta: t\in\supp\fd\}$ we get
$$
\vert\fg(e_t)\vert\leq \sum_{\delta \in A_{t}}\vert \lambda_{\delta}\fd(e_t)\vert\leq \sum_{i}\frac{1}{m_{i-1}}\frac{c_{1}}{2^{j-1}}\leq \frac{c_{1}}{2^{j}}.
$$\end{proof}
By the above sublemma  we may assume
$$
\supp(f)=\cup\{\supp(\fa):\text{ $\alpha$ maximal node with $\vert \alpha\vert\leq d_{j_0}$}\}.
$$
\begin{sublemma} Let $t\in\supp(f)$ be such that for every maximal node $\alpha\in\mc{T}$ with $t\in\supp(\fa)$, we have $\vert \alpha\vert\leq d_{j_0}$  and $\alpha\in S_{\beta}$ with $\fb$ is of type II and $\vert \lambda_{\alpha}\vert\leq m_{j_0}^{-1}$.  Then
$$
\vert f(e_t)\vert< \frac{3}{2m_{j_0}}
$$
\end{sublemma}
\begin{proof}
We prove  by induction on $j=0,\dots,d_{j_0}$
that if $\vert \gamma\vert =d_{j_0}-j$  then
$$|\fg(e_t)|<
\begin{cases} 3/2m_{j_0},\,\,\text{if $\fg$ is of type II}
\\
1/2^2m_{j_0}m_{r-1}\,\,\,\text{if $f_{\gamma}$ is of type I  and $w(f_\gamma)=m_{r}$ }.
\end{cases}
$$
For the predecessors of maximal nodes we have it by assumption. If $f_\gamma$ is of type
I, then only one of its successors have $e_t$ in its support and by the inductive
hypothesis
\begin{center}
$\vert f_\gamma(e_t)\vert<  \dfrac{3}{2m_{j_0} m_{r}}\leq \dfrac{1}{2^2m_{j_0}m_{r-1}}$
\end{center}
If $\fg$ is of type II and $\vert \fg\vert=d_{j_0}-j$, let $\fg=\sum_{\beta\in
S_{\gamma}}\lambda_\beta f_\beta$. Set  $A_{t}=\{\beta\in S_{\gamma}: t\in\supp(\fb)\}$.
There exists at most one  $f_{\beta_0}=e_t^*$ and  all other $f_\beta$ are of type I with
different weights. The  inductive hypothesis and H\"older inequality yield
\begin{align*}
\vert f_{\gamma} (e_t)\vert
 \leq   \sum_{\beta\in  A_{t}, \beta\neq\beta_0}\vert \lambda_\beta \vert
\frac{1}{2^2m_{j_0}m_{\beta-1}}+\vert\lambda_{\beta_0}\vert e_t^*(e_t)
<\frac{1}{2m_{j_0}}+\frac{1}{m_{j_0}}\leq\frac{3}{2m_{j_0}}
\end{align*}
and the induction is finished.
\end{proof}
The above sublemmas  for the tree analysis $(\fa)_{\alpha\in\mc{T}}$ of $f$ with all maximal nodes belonging to $M_{2}$ yields that
\begin{align}
\supp(f)=\cup\{\supp(\fa):\alpha\,
&\text{is a maximal node of $\mc{T}$ with $\vert \alpha\vert\leq d_{j_0}$ and}
\notag\\
&\text{if $\alpha\in S_{\beta}$, $\fb$ is of type II then $\vert\lambda_{\alpha}\vert\geq
m_{j_0}^{-1}$}\}\label{la2}
\end{align}
\begin{sublemma}\label{sl4.7}
Let  $f\in W_{j_0}^{\prime}$ with an tree analysis
$(\fa)_{\alpha\in\mc{T}}$  such that
\begin{enumerate}
\item[1)] the height of $\mc{T}$ is less or  equal to $d_{j_0}$, \item[2)] if $\fa$ is of
type I then $w(\fa)<m_{j_0}$ holds, \item[3)] if $\alpha$ is a maximal node, $\alpha\in
S_{\beta}$, $\fb$ is of type II then $\vert\lambda_{\alpha}\vert\geq m_{j_0}^{-1}$.
\end{enumerate}
Then
$$\#\supp(f)\leq (2n_{j_0-1})^{d_{j_{0}}+1}.$$
\end{sublemma}
\begin{proof}
Inductively we show that for every $\alpha$ with $\vert \alpha\vert=d_{j_0}-j$, $j=0,1,\dots,d_{j_0}$ holds the following
\begin{equation}
\label{bv2}\#\supp(\fa)\leq (2n_{j_0-1})^{j+1}.
\end{equation}
If $\fa$ is a terminal node it is clear that \eqref{bv2} holds.

Assume that the result holds for all $\alpha$ of order $d_{j_0},\dots,d_{j_0}-(j-1)$ and let $\vert \alpha\vert=d_{j_0}-j$.

If $\fa=w(\fa)^{-1}\sum_{\beta\in S_{\alpha}}\fb$ is of type I, then  we get $w(\fa)\leq
m_{j_0-1}$ and hence $\# S_{\alpha}\leq 2n_{j_0-1}$. Using the  inductive hypothesis  we
obtain $$ \#\supp(\fa)\leq 2n_{j_0-1}\max\{\#\supp(\fb):\beta\in S_{\alpha}\}\leq
(2n_{j_0-1})^{j+1}.
$$
If $\fa=\sum_{\beta\in S_{\alpha}}\lambda_{\beta}\fb$, setting
$S_{\alpha}^{1}=\{\beta\in S_{\alpha}:\beta\,\text{maximal}\}$
from  assumption 3)  we get $|\lambda_{\beta}|\geq m_{j_0}^{-1}$
for all $\beta\in S_{\alpha}^{1}$ and therefore $\#
S_{\alpha}^{1}\leq m_{j_0}^{2}$.

Note also that for every $\beta\in S_{\alpha}\setminus S_{\alpha}^{1}$ it holds that
$w(\fb)<m_{j_0}$ and hence  $\# (S_{\alpha}\setminus S_{\alpha}^{1})\leq j_{0}-1$. Hence
bys the inductive hypothesis we get
\begin{align*}
\#\supp(\fa)&\leq (j_0-1)\max\{\#\supp(\fb):\beta\in S_{\alpha}\setminus S_{\alpha}^{1}\}+m_{j_0}^{2}\\
&\leq
(j_0-1)(2n_{j_0-1})^{j}+m_{j_0}^{2}\leq
 (2n_{j_0-1})^{j+1}.
\end{align*}
\end{proof}
\textit{Proof of Lemma \ref{linfty2} completed.} The desired inequality follows from
\eqref{la2}  and Sublemma  \ref{sl4.7}.
\end{proof}
\begin{remark}\label{re4.8}
Take $f\in W_{j_0}^{\prime}$ such that for every $t\in\supp(f)$,
$|f(e_t)|>3m_{j_0}^{-2}$ and every $\fa$ in the tree analysis   $(\fa)_{\alpha\in\mc{T}}$
of $f$ satisfies $w(\fa)\ne m_{j_0}$.  Then Remark \ref{re4.5}, Sublemma \ref{sl4.6} and
\ref{sl4.7} yields that  $\#supp(f)\leq (2n_{j_0-1})^{s_{j_0-1}-2}$.
\end{remark}
\begin{lemma} \label{42}
Let $f\in W_{j_0}^{\prime}$ be of type I. Then 
\begin{equation}\label{basise1}
\,\,\,\vert f(\frac{1}{n_{j_0}}\sum_{i=1}^{n_{j_0}}e_{t})\vert \leq
\begin{cases}
\frac{4}{w(f)\cdot m_{j_0}},\quad &\text{if}\,\,\,w(f)<m_{j_0}\\
\frac{c_{1}}{w(f)},\quad &\text{if}\,\,\,w(f)\geq m_{j_0}\,.
\end{cases}
\end{equation}
If moreover we  assume that there exists a tree analysis
$(f_{\alpha})_{\alpha\in\mc{T}}$ of $f$, such that
$w(f_{\alpha})\not= m_{j_0}$
for every $\alpha\in\mc{T}$,  we have that
\begin{equation}\label{basise2}
\vert f(\frac{1}{n_{j_0}}\sum_{i=1}^{n_{j_0}}e_{k_{i}})\vert \leq
\begin{cases}
\frac{4}{m_{i}m_{j_0}^{2}}\,&\textrm{if $w(f)=m_{i}<m_{j_0}$}\\
\frac{c_{1}}{m_{i}}\,&\textrm{if $w(f)=m_{i}>m_{j_0}$}
\end{cases}
\end{equation}
\end{lemma}
\begin{proof}
If $w(f)\geq m_{j_0}$ Lemma \ref{linfty} yields the result.

Let $f=w(f)^{-1}\sum_{i=1}^{d}f_{i}$ with $w(f)<m_{j_0}$.
We set
$$A_{1}=\{t: t\in\supp(f_i)\,\text{for some $i\leq  d$ and}\,\, f_{i}(e_t)\leq 3m_{j_0}^{-1}\}.
$$
It readily follows that
\begin{equation}\label{kkk0}
\vert f(e_t)\vert\leq \frac{3}{w(f)m_{j_0}}\,\,\,\text{for every $t\in A_1$}.
\end{equation}
We set $B_{i}=\supp(f_{i})\setminus A_{1}$. Lemma \ref{linfty2} yields that
$$\#\supp (B_{i})\leq (2n_{j_0-1})^{s_{j_0-1}-2}.
$$
It follows that $\#\{t: t\not\in A_{1}\}\leq 2n_{j_0-1}(2n_{j_0-1})^{s_{j_0-1}-2}=(2n_{j_0-1})^{s_{j_0-1}-1}$ and therefore
\begin{equation}\label{kkk1}
\vert f(\sum_{t\not\in A_{1}}\frac{e_t}{n_{j_0}})\vert\leq
\frac{(2n_{j_0-1})^{s_{j_0-1}-1}}{n_{j_0}}<m_{j_0}^{-2}.
\end{equation}
From \eqref{kkk0} and \eqref{kkk1} we get the result.

For the proof of \eqref{basise2}  we consider the set
$$
A_{1}=\{t: t\in\supp(f_i)\,\text{for some $i\leq  d$ and}\,\, f_{i}(t)\leq
3m_{j_0}^{-2}\}.
$$
and using Remark \ref{re4.8} we obtain  again the relation \eqref{kkk1} for the new set $A_{1}$.
\end{proof}
\section{The Basic Inequality and its consequences}\label{sec9}
In this section we introduce the notion of rapidly increasing sequences (RIS) and
prove the Basic Inequality (Prop. \ref{bin}), which by results from the previous section
provides estimation on the norm of suitable averages of RIS (Prop. \ref{ris}). In
particular we obtain reflexivity of $X_\xi$.
\begin{definition}\label{d9.1}
Let $\e>0$. A block sequence $(x_{k})$ in $X_{\xi}$ is said to be a
$(C,\varepsilon)$-rapidly increasing sequence  (RIS), if there
exists a strictly increasing sequence $(j_{k})$ of positive
integers such that
\begin{enumerate}
\item[a)] $\norm[x_{k}]\leq 1$ for all $k$.
\item[b)]
$m_{2j_1}^{-1/2}<\e$ and $\#(\ran(x_{k}))m_{2j_{k+1}}^{-1/2}<\e$
for all $k\geq 1$
\item[c)] For every $k=1,2,\ldots$ and every
$f\in K_{\xi}$ of type  I with $w(f)<m_{2j_{k}}$  we have that $\vert f(x_{k})\vert
\le \frac{C}{w(f)}$.
\end{enumerate}
\begin{remark}\label{r5.2} From  Proposition  \ref{p4.6} and Lemma  \ref{l4.5} we get that  for  every $\e>0$   any block subspace  contains a $(3,\e)$-RIS $(x_{k})_{k}$ where  $x_k$ is a normalized $2-\ell_{1}^{n_{j_k}}$ average with $(2j_k)_{k}$ satisfying  condition b) of the above definition.
\end{remark}
\end{definition}
\begin{proposition}[Basic Inequality]\label{bin}
Let $\e>0$, $1<j_0\in\N$, $(x_{k})_{k\in\N}$ be a $(C,\e)$-RIS in  $X_{\xi}$  with the associated sequence $(j_{k})_{k}$.
 Assume that for   every $g\in G_{\xi}$ the set
$I_{g}=\{k:\vert g(x_k)\vert\geq \e\}$
has cardinality at most $n_{j_0-1}$.

Let $(c_{k})_{k}$ be a sequence of scalars.
Then for every $f\in K_{\xi}$ and every interval $I$ there exists a functional $g\in W_{j_0}$ such that
\begin{equation}\label{ll1}
\vert f(\sum_{k\in I}c_{k}x_{k})\vert\leq C(g(\sum_{k\in
I}\vert c_k\vert e_{k})+\e_{f}\sum_{k\in I}\vert
c_k\vert),\,\,\,\e_{f}\leq \e.
\end{equation}
Moreover, if $f$ is the result of an $(\mc{A}_{n_{j}},m_{j}^{-1})$-operation then
$g=e^{*}_{r}$ or $0$ or $g$ is the result of an $(\mca_{2n_{j}}, m_{j}^{-1})$ operation and $\e_{f}\leq \e w(f)^{-1/2}$.

If we additionally assume that for every  $f\in K_{\xi}$ with $w(f)=m_{j_0}$, for every interval $J$   it holds that
\begin{equation}\label{ada}
\vert f(\sum_{k\in J}c_{k}x_{k})\vert
\leq C(\max_{k\in J}\vert c_kf(x_k)\vert+\e\,m_{j_0}^{-1/2}\sum_{k\in J}\vert c_k\vert),
\end{equation}
then we may select the functional $g$ to have a tree analysis $(g_{\alpha})_{\alpha}$ with $w(g_{\alpha})\ne m_{j_0}$ for all $\alpha\in\mc{A}$.
\end{proposition}
\begin{proof}
We shall treat the case that there exists  $j_0$ satisfying   \eqref{ada}.
We also assume  that
$\ran(x_k)\cap\ran(f)\ne\emptyset$ for every $k\in I$.
We proceed by induction of the order $o(f)$ of the functional $f$.

Let $o(f)=1$. Then we have that $f\in G_{\xi}$ and from  the assumptions
the set
$R_{f}=\{k\in I:\vert f(x_k)\vert\geq \e\}$ has cardinality at most
$n_{j_0-1}$.
We set $g_{f}=\sum_{k\in R_{f}}e_{k}^{*}$ and $\e_{f}=\e$.
It readily follows
$$
\vert f(\sum_{k\in I}c_{k}x_{k})\vert\leq g_{f}(\sum_{k\in I}\vert c_k\vert e_{k})+\e_{f}\sum_{k\in I}\vert c_{k}\vert.
$$
Suppose now that the result holds for every functional in $K_{\xi}$ with order less than
$q$ and consider $f\in K_{\xi}$ with $o(f)= q$.

We consider the following three cases.

\textit{Case 1}. $f$ is of type I and $w(f)=m_{j_0}$.
 We choose
$k_0\in I$ with $\vert c_{k_0}f(x_{k_0})\vert=\max_{k\in I}\vert
c_kf(x_k)\vert$ and we set $g_{f}=\vert f(x_{k_0})\vert
e^{*}_{k_0}$. Then from our assumption \eqref{ada} it follows
\begin{align*}
\vert f(\sum_{k\in I}c_{k}x_{k})\vert&\leq C\left(\max_{k\in I}\vert c_kf(x_{k})\vert+\e m_{j_0}^{-1/2}\sum_{k\in I}\vert c_k\vert\right)
\\
&\leq C\left(g(\sum_{k\in I}\vert c_k\vert e_{k})+\e m_{j_0}^{-1/2}\sum_{k\in I}\vert c_k\vert\right).
\end{align*}
\textit{Case 2}. $f$ is of type I and $w(f)\ne m_{j_0}$.

Then
$
f=m_{j}^{-1}\sum_{i=1}^{d}f_{i}
$
with $j\ne j_0$  and $d\leq n_{j}$.
We consider the following three subcases.

\textit{Subcase 2a}. $w(f)<m_{2j_k}$ for all $k\in I$.

For every $i\leq d$ we set
$$I_{i}=
\{k\in I: \ran(x_k)\cap \ran(f_i)\ne\emptyset\,\text{and}\,\ran(x_k)\cap\ran(f_{i^{\prime}})=\emptyset\,\text{for all}\,\, i^{\prime}\ne i\}.
$$
We also set $I_{0}=\{k\in I:\ran(x_k)\cap\ran(f_i)\ne\emptyset\,\text{ for at least
two}\,\, i\in \{1,\dots,d\}\}$.
We observe that $\# I_{0}\leq d$. Condition   c) in the definition
of the RIS yields
\begin{equation}\label{eq1}
\vert f(x_k)\vert\leq\frac{C}{w(f)}\,\,\,\,\text{for every  $k\in I_0$}.
\end{equation}
For every $i\leq d$   we have that $I_i$ is a subinterval of $I$, hence our inductive
assumption yields that there exists $g_{f_i}\in W_{j_0}$ with $\supp( g_{f_i})\subset I_i$ such that
\begin{equation}\label{eq2}
\vert f_{i}(\sum_{k\in I_{i}}c_{k}x_{k})\vert
\leq
C(g_{f_i}(\sum_{k\in I_{i}}\vert c_{k}\vert e_{k})+\e_{f_{i}}\sum_{k\in I_{i}}\vert c_k\vert).
\end{equation}
The family $\{I_1,\dots, I_d\}\cup\{\{k\} : k\in  I_0\}$ consists
of pairwise disjoint intervals and has cardinality less than or
equal to $2d\leq 2n_j$. We set
$$
g_{f}=\frac{1}{w(f)}(\sum_{i=1}^{d}g_{f_i}+\sum_{k\in
I_0}e^{*}_{k}).
$$
Then $g_{f}\in W_{j_0}$, $\supp g_{f}\subset I$, while from \eqref{eq1}, \eqref{eq2} we
obtain \begin{align}
\vert f(\sum_{k\in I}c_{k}x_{k})\vert &
\leq \sum_{k\in I_0}\vert c_k\vert\vert
f(x_k)\vert+\frac{1}{w(f)}\sum_{i=1}^{d}C\left(g_{f_{i}}(\sum_{k\in I_{i}}\vert c_k\vert
e_{k})+\e\sum_{k\in I_{i}}\vert c_k\vert\right)
\notag\\
&\leq C\left(g_{f}(\sum_{k\in I}\vert c_k\vert e_{k})+\e_{f}\sum_{k\in I}\vert c_k\vert\right),\,\,\text{where $\e_{f}=\e w(f)^{-1}$}.
\end{align}
\\

\textit{Subcase 2b.} $m_{2j_{k_0}}\leq w(f)<m_{2j_{k_0+1}}$ for
some $k_0\in I$.

From condition  $b)$ in the definition of RIS we get
\begin{equation}\label{eq3a}
\mbox{
$\vert f(x_k)\vert\leq \e\, w(f)^{-1/2}=\e_{f}$ for all $k\in I$ with $k<k_0$}.
\end{equation}
Using that $w(f)\geq m_{2j_1}$ and condition c) we get
\begin{equation}\label{eq4a}
\mbox{$\vert f(x_k)\vert\leq Cw(f)^{-1}\leq C\e w(f)^{-1/2}=C\e_{f}$
for every $k_0<k\in I$}.
\end{equation}
Thus setting $g_{f}=\vert f(x_{k_0})\vert e^{*}_{k_0}$  from \eqref{eq3a},\eqref{eq4a} 
we get
\begin{align}
\vert f(\sum_{k\in I}c_{k}x_{k})\vert &\leq \vert c_{k_0}f(x_{k_0})\vert+\sum_{k\in I\setminus\{k_0\}}\vert c_{k}f(x_k)\vert
\\
&\leq \vert c_{k_0}f(x_{k_0})\vert+C\e_{f}\sum_{k\in I\setminus\{k_0\}}\vert c_k\vert
\leq C\left(g_{f}(\sum_{k\in I}\vert c_{k}\vert e_{k})+\e_{f}\sum_{k\in I}\vert c_k\vert\right).\notag
\end{align}
\textit{Subcase 2c.} $m_{2j_{k+1}}\leq w(f)$ for all $k\in I$.

In this case  as in \eqref{eq3a}, $\vert f(x_k)\vert \leq
\e\,\,w(f)^{-1/2}=\e_{f}$ for all $k\in I$, and  we set $g_{f}=0$.
It follows easily that  \eqref{ll1} holds.

\smallskip

\textit{Case 3.} $f$ is of type II, i.e. $f=\sum_{i=1}^{d}\lambda_{i}f_{i}$.

For the given interval $I$  the inductive assumption associates to each $f_{i}$ a functional $g_{f_{i}}$ satisfying
\eqref{ll1}.

We note that  there exists at most one $i\in\{1,\dots,d\}$ with $w(f_i)=m_{j_0}$. Without
loss of generality we assume that there does exist such an $i$, denoted by $i_0$. From
\textit{Case} 1 we have $g_{f_{i_0}}=\vert f(x_{k_0})\vert e^{*}_{k_0}$ for some  $k_0\in
I$.

For every $1\leq i \leq d$ we set
 $C_{f_i}=\{k\in I: \ran(f_{i})\cap\ran(x_k)\ne\emptyset\}$.

We partition the set $M=\{1,\dots,d\}\setminus\{i_0\}$ as follows:
$$
L_{0}=\{i\in M: w(f_{i})<m_{2j_{k}}\,\,\text{for all $k\in C_{f_i}$}\}
$$
and for $k=1,\dots,$ we set
$$
\mbox{$L_{k}=\{i\in M: k\in C_{f_{i}}\,\,\text{and}\,\,
m_{2j_{k}}\leq w(f_i)<m_{2j_{k+1}}\}$}.
$$
We enlarge $L_{k_0}$ by adding  $i_0$ to its elements.
Note that if $i\not\in L_{0}\cup\cup_{k}L_{k}$ then  $g_{f_i}=0$.
We set
\begin{align*}
g_{f}=\sum_{i=1}^{d}\vert \lambda_{i}\vert g_{f_i}= \sum_{i\in
L_0}\vert \lambda_{i}\vert g_{f_i} +\sum_{k}(\sum_{i\in
L_{k}}\vert\lambda_{i}f_{i}(x_k)\vert)e_{k}^{*}.
 \end{align*}
We show that $g_{f}\in W_{j_0}$. Since for every $i\in L_{0}$ it holds  that
$w(g_{f_i})=w(f_{i})$ we have that the functionals $g_{f_i}, i\in L_{0}$, have different
weights. Also since the sets  $L_{k}$ are pairwise disjoint it follows
$$
\sum_{i\in L_0}\lambda_{i}^{2}+
\sum_{k}(\sum_{i\in L_{k}}\vert\lambda_{i}f_{i}(x_k)\vert)^2\leq
\sum_{i\in L_0}\lambda_{i}^{2}+
\sum_{k}(\sum_{i\in L_{k}}\lambda_{i}^2)\norm[x_k]\leq
\sum_{i=1}^{d}\lambda_{i}^{2}\leq 1,
$$
hence $g_{f}\in W_{j_0}$.

We show that  \eqref{ll1} holds.  First we observe that since for all $i$, $\e_{f_{i}}\leq\e\,w(f_{i})^{-1/2}$ it follows that
$
\sum_{i\leq d}\vert\lambda_{i}\vert \e_{f_{i}}\leq \e.
$
Also,
\begin{align*}
\vert f(\sum_{k\in I}c_{k}x_{k})\vert
& \leq
\sum_{i=1}^{d}\vert\lambda_{i}f_{i}(\sum_{k\in C_{f_i}}c_{k}x_{k})\vert
\leq \sum_{i=1}^{d}C\vert\lambda_{i}\vert
(g_{f_i}(\sum_{k\in C_{f_i}}\vert c_{k}\vert e_{k}) +
\e_{f_{i}}\sum_{k\in C_{f_i}}\vert c_{k}\vert)
\\
&\leq C\left(\sum_{i\in L_0}\vert\lambda_{i}\vert g_{f_i}+
\sum_{k}(\sum_{i\in L_{k}}\vert\lambda_{i}f_{i}(x_k)\vert)
e_{k}^{*}\right) \left(\sum_{k\in I}\vert c_k\vert e_{k}\right)
+C\e\sum_{k\in I}\vert c_k\vert
\\
&=Cg_{f}(\sum_{k\in I}\vert c_k\vert e_{k}) +C\e\sum_{k\in I}\vert c_k\vert.
\end{align*}

\textit{Case 4.} $f=\sum_{i=1}^dr_{i}f_{i}$, where $(r_{i})_{i=1}^d\subset\Q$, is a
rational convex combination.

As in the previous case for every $i=1,\dots,d$ we set
$$I_{i}=\{k\in I: \ran(f_{i})\cap\ran(x_k)\ne\emptyset\}.
$$
Take suitable $(g_{f_i})$ by the inductive hypothesis. Setting $g_{f}=\sum_{i=1}^{d}r_{i}g_{f_i}$ we get
$$
\vert f(\sum_{k\in I}c_{k}x_{k})\vert \leq  C(g_{f}(\sum_{k\in I}\vert c_{k}\vert e_{k})+\e\sum_{k\in I}\vert c_k\vert)
$$
\end{proof}
\begin{proposition}\label{ris}
Let $\e>0$, $1<j\in\N$, $(x_{k})_{k\in\N}$ be a $(C,\e)$-RIS in  $X_{\xi}$  with the associated sequence $(j_{k})_{k}$ and  $j<j_{1}$.
 Assume that for   every $g\in G_{\xi}$ the set
$I_{g}=\{k:\vert g(x_k)\vert\geq 2m_{j}^{-2}\}$
has cardinality at most $n_{j-1}$.
Then

a) If $\varepsilon\leq\frac{2}{m_{j}^{2}}$ then for every $f\in K_{\xi}$ of type I
$$
\vert f(\frac{1}{n_{j}}\sum_{k=1}^{n_{j}}x_{k})\vert \leq
\begin{cases}
\frac{5C}{m_{j}w(f)}\,,\quad &\text{if}\,\,\, w(f)<m_{j}\\
\frac{c_{1}C}{w(f)}+\frac{2C}{m^2_{j}}\,,&\text{if}\,\,\, w(f)\geq
m_{j}\,.
\end{cases}
$$
In particular
$\|\frac{1}{n_{j}}\sum\limits_{k=1}^{n_{j}}x_{k}\|
\leq\frac{3C}{m_{j}}$.

b) If  moreover for $j_{0}=j$ the additional assumption  of the
Basic Inequality is fulfilled (Proposition~\ref{bin} \eqref{ada}), then for a linear
combination $\frac{1}{n_{j}}\sum_{i=1}^{n_{j}}b_{i}x_{i}$, where $\vert b_{i}\vert\leq
1$, we have
$$
\|\frac{1}{n_{j}}\sum_{i=1}^{n_{j}}b_{i}x_{i}\|
\leq\frac{3C}{m_{j}^{2}}\,\,.
$$
\end{proposition}
\begin{proof}
The proof is an application of the Basic Inequality and Lemma
\ref{42}.
\end{proof}
\begin{proposition}
The space  $X_{\xi}$ is reflexive.
\end{proposition}
The reflexivity of $X_{\xi}$ is consequence of  Proposition \ref{ris} following standard arguments, see \cite{AAT},\cite{AT}.
\begin{remark}\label{denseK}
The reflexivity of the space $X_{\xi}$ yields that  $K_{\xi}$ is norm-dense subset of $B_{X_{\xi}^{*}}$.  Indeed since  $K_{\xi}$ is norming set   it follows that  $\conv(K_{\xi})$ it is $w^{*}-$dense.  The reflexivity of $X_{\xi}$ implies that $\conv(K_{\xi})$  is $w-$dense and hence $\norm$-dense. Since $K_{\xi}$ is rationally convex we get that $K_{\xi}$ is in fact norm-dense subset of $B_{X_{\xi}^{*}}$.
\end{remark}
\section{Exact pairs and  attracting sequences}\label{sec10}
Now we define exact pairs and show saturation of $X_\xi$ by them. Next we introduce the
crucial notion of attracting sequences and estimate the norms of averages of elements
forming these sequences: vectors (Corollary \ref{depest}) and functionals (Corollary
\ref{c412}).
\begin{definition}\label{d4.8}
A pair $(x,f)$ with $x\in X_{\xi}$ and $f\in K_{\xi}$ is said to be a $(C,2j)$-exact pair, $C\geq 1$, $j\in\N$ if the following conditions are satisfied
\begin{enumerate}
\item[1)] $f(x)=1$ and $\ran f=\ran x$.
\item[2)] $f$ is  of type I and $w(f)=m_{2j}$.
\item[3)] $1\leq\norm[x]\leq 3C$, $\norm[x]_\infty\leq m_{2j}^{-2}$ and for every $g$ of type I
with $w(g)<m_{2j}$ it holds that $\vert g(x)\vert\leq \frac{5C}{w(g)}$ while for $g$ of
type I with $w(g)>m_{2j}$, $\vert g(x)\vert\leq 5Cm_{2j}^{-1}$.
\end{enumerate}
\end{definition}
\begin{proposition}\label{p4.9}
Let $j\in\N$ and  $Y$ be a block subspace of $X_{\xi}$. There exists a
$(3,2j)$-exact pair with $x\in Y$.
\end{proposition}
\begin{proof}
Let  $Y$ be a block subspace  of $X_{\xi}$ and  $j\in\N$. By Remark \ref{r5.2},
Proposition \ref{ris-sss} we can choose for $\e\leq 2m_{j}^{-2}$ a $(3,\e)$-RIS
$(x_{k})_{k=1}^{n_{2j}}$ satisfying the assumptions of Proposition \ref{ris}. It follows
that
$$
\|\frac{m_{2j}}{n_{2j}}\sum_{i=1}^{n_{2j}}x_{i}\|\leq 9.
$$
Choosing $f_{i}\in K_{\xi}$ such that $f_{i}(x_{i})=1$ and $\ran (f_i)\subset\ran (x_i)$ we have that $f=m_{2j}^{-1}\sum_{i=1}^{n_{2j}}f_{i}\in K_{\xi}$ and
$f\left(\frac{m_{2j}}{n_{2j}}\sum_{i=1}^{n_{2j}}x_{i}\right)= 1.$
Setting  $x=E\frac{m_{2j}}{n_{2j}}\sum_{i=1}^{n_{2j}}x_{i}$, where
$E=\ran(f)$,   Proposition \ref{ris} a) yields that  $(x,f)$ is a
$(3,2j)$-exact pair.
\end{proof}
\begin{definition}\label{d10.3}
A double sequence $(x_k,f_k)_{k=1}^{n_{2j-1}}$ is called a $(C,2j-1)$ attracting
sequence, if there is a sequence $(j_k)_{k=1}^{n_{2j-1}}$ such that
\begin{enumerate}
\item $(f_k)_{k=1}^{n_{2j-1}}$ is a $(2j-1)$-attractor sequence
with $w(f_{2k-1})=m_{j_{2k-1}}$ and $f_{2k}=e^*_{l_{2k}}$ where
$l_{2k}\in\Lambda_{j_{2k}}$ for all $k\leq n_{2j-1}/2$,
\item
$x_{2k}=e_{l_{2k}}$ for all $k\leq n_{2j-1}/2$,
\item
$(x_{2k-1},f_{2k-1})$ is a $(C,j_{2k-1})$ exact pair.
\item $j_k=2\sigma(f_1,\dots,f_{k-1})$ for any $k\leq n_{2j-1}$
\end{enumerate}
\end{definition}
\begin{remark} If $(x_k,f_k)_{k=1}^{n_{2j-1}}$ is a $(C,2j-1)$-attracting
sequence, then $(x_k/(3C))_{k=1}^{n_{2j-1}}$ is a $(5/3, n_{2j-1})$ RIS. Indeed,
$$\# (\ran x_k)\frac{1}{m^{1/2}_{j_{i_{k+1}}}}=\# (\ran
f_k)\frac{1}{m^{1/2}_{2\sigma(f_1,\dots,f_k)}}<\min \{
\norm[f_i]_\infty,\ i\leq k\}\leq m_{j_1}^{-1/2}\leq n^{-1}_{2j-1}
$$
by the condition on $\sigma$. Condition (c) in definition of RIS is satisfied thanks
to $x_{2k}=e_{l_{2k}}$ and the fact that $(x_{2k-1},f_{2k-1})$ is a $(C,j_{2k-1})$ exact pair.
\end{remark}
\begin{corollary}\label{depest}
Let  $(x_k,f_k)_{k=1}^{n_{2j-1}}$ be a $(C,2j-1)$ attracting sequence with associated
sequence $(j_k)_{k}$, with $\norm[x_{2k-1}]_{G_{\xi}}\leq m_{2j-1}^{-2}$ for every $k\leq
n_{2j-1}/2$. Then
\begin{equation}\label{qq2} \frac{1}{2m_{2j-1}^{2}}\leq \|\frac{1}{n_{2j-1}}\sum\limits_{k=1}^{n_{2j-1}}
  (-1)^{k+1}x_k\| \leq \frac{15C}{m_{2j-1}^2}
\end{equation}
\end{corollary}
\begin{proof}
To see the lower estimate in \eqref{qq2} note that
$\tilde{f}=m_{2j-1}^{-2}\sum_{k=1}^{n_{2j-1}/2}f_{2k}\in G_{\xi}\subset K_{\xi}$.
Hence
$$
\|\frac{1}{n_{2j-1}}\sum\limits_{k=1}^{n_{2j-1}}
  (-1)^{k+1}x_k\|\geq \vert \tilde{f}(\frac{1}{n_{2j-1}}\sum\limits_{k=1}^{n_{2j-1}}
  (-1)^{k+1}x_i)\vert =m_{2j-1}^{-2}/2.
$$
The upper estimation in \eqref{qq2} follows from Proposition
\ref{ris} b) for $b_{k}=(-1)^{k}$ and $j_0=2j-1$ after we show
that
\begin{enumerate}
 \item[K1)]for every $g\in G_{\xi}$ it holds that $\#\{k:\vert g(x_k)\vert\geq 2m_{2j-1}^{-2}\}\leq n_{2j-2}$
\item[K2)]  The additional property of the Basic Inequality holds for the sequence $(x_k/(3C))_{k=1}^{n_{2j-1}}$ with constant $5/3$.
\end{enumerate}
To see K1) note that from the assumption we have that
$\norm[x_{2k-1}]_{G_{\xi}}\leq m_{2j-1}^{-2}$. Take now any $g\in
G_{1}^{r}=\{m_{2r-1}^{-2}\sum_{i\in F}\pm e_{i}^{*}: F\in \mca_{n_{2r-1}}\}$. If $r\geq j$ then $\vert g(e_{l_{2k}})\vert\leq
m_{2j-1}^{-2}$ while if $r<j$ it holds that $\#\{l_{2k}: \vert
g(e_{l_{2k}})\vert\geq m_{2j-1}^{-2}\}\leq n_{2r-1}$.

From the definition of the $G_\xi-$special functionals $g=E\sum_{r=1}^{d}g_{r}$ we obtain
$$ \vert g(e_{l_{2k}})\vert\geq 2m_{2j-1}^{-2}\Rightarrow l_{2k}\in\supp
g_{r}\,\textrm{with}\,g_{r}\in \cup_{i\leq j-1}G_{1}^{i}
$$
and therefore
$$\#\{l_{2k}: \vert g(e_{l_{2k}})\vert\geq 2m_{2j-1}^{-2}\}\leq \sum_{r<j}n_{2r-1}\leq n_{2j-2}.
$$
Finally let $y^{*} = \sum_{k=1}^d a_ky_k^{*} \in G_{\ell_2}$. For every
$k=1,\dots,d$ let $y_k^{*} = y_{k,1}^{*} + y_{k,2}^{*}$ with
$\ind(y_{k,1}^{*}) \subset\{1,\dots,j\}$ and
$\ind(y_{k,2}^{*}) \subset \{j+1, j+2,\dots\}.$ So we may
write $y^{*}= \sum_{k=1}^d a_k y_{k,1}^{*} + \sum_{k=1}^d a_k
y_{k,2}^{*}.$

Since $\sum_{i\geq j}m_{2i-1}^{-2}<\frac{3}{2m_{2j-1}^{2}}$
and the
sets $\inde(y_{k}^{*})_{k}$ are pairwise disjoint
in
order $\vert y^{*}(e_{l_{2i}})\vert\geq 2/m_{2j-1}^{2}$ it must hold
that  $l_{2i}\in\supp y_{k,1}^{*}$ for some $k\leq d$.
As in
the previous case we get
$$\#\{l_{2i}: \vert y^{*}(e_{l_{2i}})\vert\geq 2m_{2j-1}^{-2}\}\leq \sum_{r<j}n_{2r-1}\leq n_{2j-2}.
$$
To see K2) we have to show that for any $(2j-1)$-attractor
functional $g$ and for any interval $J\subset\{1,\dots,n_{2j-1}\}$
we have
$$
\vert g(\sum_{k\in J}(-1)^kx_k)\vert\leq 5C(\max\vert g(x_k)\vert+m_{2j-1}^{-4}\# J).
$$
Let $g=m_{2j-1}^{-1}\sum_{i=1}^dg_i$, $d\leq n_{2j-1}$. If $g=f=m_{2j-1}^{-1}\sum_{i=1}^{n_{2j-1}}f_{i}$
then $\vert g(\sum_k(-1)^kx_k)\vert=0$. Otherwise let
$i_0=\min\{i\leq d:\ f_i\ne g_i\}$. Then, by definition of the
attracting sequence and since $J$ is an interval, we have
\begin{equation}\label{bas1}
\vert \sum_{i=1}^{i_0-1}g_i(\sum_{k\in J}(-1)^kx_k)\vert\leq 3\max_{i\leq i_0-1}\vert g_{i}(x_{i})\vert.
\end{equation}
By Remark \ref{r3.2}a) for any $2i-1>i_0,2k-1>i_0$ we
have $w(g_{2i-1})\ne w(f_{2k-1})$ and for any $2i>i_0$, $2k>i_0$
we have $g_{2i}\ne f_{2k}$. Notice $g_{2i}(x_{2k})=0$ for any
$2i>i_0$ and any $2k\in J$.
By the definition of the attractor sequence we get for any $2k>i_0$
\begin{equation}\label{bas2}
\vert\sum_{i \geq i_0}g_{i}(x_{2k})\vert\leq \max_{i}\norm[g_{2i-1}]_{\infty}\leq m_{j_1}^{-1}\leq m_{2j-1}^{-4}.
\end{equation}
Now by the definition of the exact sequence for any $2i-1\geq i_0$ and
any $2k-1>i_0$ we have
$$
\vert g_{2i-1}(x_{2k-1})\vert\leq 5C\max\{w(g_{2i-1})^{-1},
m_{j_{2k-1}}^{-1}\}\leq 5Cn_{2j-1}^{-2}.
$$
Hence using that $\norm[x_{2k-1}]_{\infty}\leq m_{j_{2k-1}}^{-2}$ we obtain
\begin{equation}\label{bas3} \vert\sum_{i\geq i_0}g_{i}(x_{2k-1})\vert\leq 5C
n_{2j-1}(n_{2j-1}^{-2}+\norm[x_{2k-1}]_{\infty})\leq 5C m_{2j-1}^{-4}.
\end{equation}
From \eqref{bas1}, \eqref{bas2}, \eqref{bas3} we get
\begin{align*}
\vert g(\sum_{k\in J}(-1)^{k}x_{k})\vert
&=
\vert \frac{1}{m_{2j-1}}(\sum_{i=1}^{i_0-1}g_{i}+\sum_{i\geq i_0}g_{i})(\sum_{k\in J}(-1)^{k}x_{k})\vert
\\
&\leq 3\max_{i<i_0}\vert g(x_i)\vert+\vert g(x_{i_0})\vert+5C\sum_{k\in J} m_{2j-1}^{-4}
\leq
5C(\max_{k\in J}\vert g(x_k)\vert+\frac{\#J}{m_{2j-1}^{4}})
\end{align*}
which ends the proof of K2) and thus the whole proof.
\end{proof}
\begin{corollary}\label{c412}
Let $(x_i,f_i)_{i=1}^{n_{2j-1}}$ be a $(C,2j-1)$-attracting sequence of
length $n_{2j-1}$ satisfying the assumption of Corollary
\ref{depest}. Set
$$
\phi=m_{2j-1}^{-2}\sum\limits_{i=1}^{n_{2j-1}/2}f_{2i-1},\ \ \ \ \psi=m_{2j-1}^{-2}\sum\limits_{i=1}^{n_{2j-1}/2}f_{2i}
$$
Then
$$
\frac{1}{30C}\leq\norm[\psi]\leq 1,\,\,\,\norm[\phi+\psi]\leq
m_{2j-1}^{-1}.
$$
\end{corollary}
\begin{proof}
Notice that $m_{2j-1}(\phi+\psi)\in K_{\xi}$ hence the second
inequality holds. To prove the first, from Corollary \ref{depest} we have
$\|\frac{1}{n_{2j-1}}\sum\limits_{k=1}^{n_{2j-1}}
  (-1)^{k+1}x_k\| \leq \frac{15C}{m_{2j-1}^2}$ and therefore
$$
\norm[\psi]\geq\psi(\frac{m_{2j-1}^2}{15Cn_{2j-1}}\sum_{i=1}^{n_{2j-1}}(-1)^{i+1}x_i)
=\frac{m_{2j-1}^2}{15Cn_{2j-1}}
\frac{1}{m_{2j-1}^{2}}\sum_{i=1}^{n_{2j-1}/2}f_{2i}(x_{2i})
=\frac{1}{30C}.
$$
\end{proof}
\section{Spaces with no $\ell_{p}$ as a spreading model}
\label{sec11}
In this section we show that the space $X_{\xi}$ does not admit
$c_0$ or $\ell_{p}, 1\leq p<\infty$, as a
spreading model.
Actually we  show that  this holds   for a wider class of Banach spaces
which describe now.

Let $G$ be a ground set.
Let $\mathcal{W}_{G}$ denote the smallest subset of $c_{00}(\N)$ which
\begin{enumerate}
 \item is symmetric, closed under the projections of its elements on intervals of $\N$ and  $G\subset\mathcal{W}_{G}$.
 \item  for every $j\in\N$ is closed under the $(A_{n_j}, m_{j}^{-1})$ operation.
 \item whenever $(f_{i})_{i=1}^{d}$ is the result  of an
$(A_{n_{j_i}},m_{j_i}^{-1})$-operation  with $n_{j_k}\ne n_{j_{m}}$ for $k\neq m$,  then
$\sum_{i=1}^{d}\lambda_{i}f_{i}\in \mathcal{W}_{G}$ for all $(\lambda_{i})_{i=1}^{d}\in
B_{\ell_2}\cap [\Q]^{<\infty}$.
 \item is rationally convex.
\end{enumerate}
\begin{definition}\label{trcomp}
A subset $D_{G}$ of $\mathcal{W}_{G}$ is said to be an extension of $G$ if:
 \begin{enumerate}
 \item[(i)]  The set  $D_{G}$ is symmetric, closed under the projections of its elements
on intervals of $\N$ and  $G\subset D_G$.
 \item[(ii)] For any $j\in\N$ we have that $D_{G}$ is  closed under the $(A_{n_{2j}},m_{2j}^{-1})$-operation.
\item[(iii)] Whenever $(f_{i})_{i=1}^{d}$ is the result  of an
$(A_{n_{j_i}},m_{j_i}^{-1})$-operation  with $n_{j_k}\ne n_{j_{m}}$ for $k\neq m$,  then
$\sum_{i=1}^{d}\lambda_{i}f_{i}\in D_{G}$ for all $(\lambda_{i})_{i=1}^{d}\in
B_{\ell_2}\cap [\Q]^{<\infty}$. \item[iv)] It  is rationally convex.
\end{enumerate}
 \end{definition}
\begin{definition}\label{norm1}
 Let $D_{G}$ be an extension subset of $\mc{W}_{G}$.
We define $\mathcal{Y}_{D_{G}}=\overline{(c_{00}(\N),\norm_{D_{G}})}$.
\end{definition}
We prove now the following theorem
\begin{theorem}\label{t11.3}
Let $G$ be a ground set  such that the corresponding space $X_{G}$ does not admit
$\ell_{1}$ as a spreading model. Then for every extension $D_{G}$ the space $\mc{Y}_{D_G}$ does not admit any $\ell_p$ or $c_0$ as a spreading model.
\end{theorem}
\begin{remark}
The first example of a Banach space $X$ with no $\ell_{p}$ as a spreading model was given
by E. Odell and Th. Schlumprecht \cite{OSsm}.  The spaces we consider in Theorem
\ref{t11.3} are extensions of their example. In particular  the space $\mc{Y}_{D_G}$
is similar to their example when  $G=\{\pm e_{n}^{*}: n\in\N\}$ and  $D_{G}=\mc{W}_{G}$.  Our proof provides also an alternative proof of their result.
\end{remark}
\begin{proof}[Proof of the theorem]
First we note that  Proposition \ref{p4.6}  holds for the space $\mc{Y}_{D_G}$. It follows
that  only  $\ell_{1}$ is finitely block representable  in  $\mc{Y}_{D_G}$ and hence
$\mc{Y}_{D_G}$ does not admit $c_0$ or $\ell_{p}$, $p>1$, as a spreading model. We prove
now that the space $\mc{Y}_{D_G}$ does not contain a normalized sequence generating an
$\ell_{1}$-spreading model.

Since the space  $X_{G}$ does not admit  $\ell_{1}$ as aspreading  model Erdos-Magidor
theorem \cite{EM},  yields  that for every bounded sequence $\xn$ and
every $\e>0$ we can choose a block sequence $\xn[y]$  of $\xn$
where each  $y_{n}=\sum_{k\in F_{n}}x_{k}/n_0$,  $\#F_{n}=n_0$,
such that $\norm[y_n]_{G}<\e$.

Also by the Fact \ref{fact1}  if a sequence $\xn$ generates an $\ell_{1}$-spreading model with constant $c$ then for every $\e>0$ there
exists a block sequence  $\xn[y]$ of $\xn$ generating an $\ell_{1}$-spreading model with constant $(1-\e)^{-1}$.
So assuming   that a normalized block sequence $(y_n)$ generates an
$\ell_{1}$-spreading model with constant $C$,  passing to suitable  block sequence $\xn[z]$ of $\xn[y]$ we may assume that
\begin{enumerate}
\item[A)]  $\xn[z]$ generates an $\ell_{1}$-spreading model with constant $(1-\e)^{-1}$.
\item[B)] $\norm[z_n]_{G}<\e$ for every $n\in\N$.
\end{enumerate}
We shall need  also the following lemmas
\begin{lemma}\label{le7.1}
 Let $x\in \mathcal{Y}_{D_G}$.
Then for every $\e>0$ there exists $j_0\in\N$ such that for every $\phi\in D_{G}$ of
type II with $\inde(\phi)> j_0$ it holds that $\vert \phi(x)\vert<\e$.
\end{lemma}
The proof is similar to the proof of Lemma \ref{ll1a} and we omit it.
\begin{lemma}\label{le7.2} Assume that $\xn[z]$  satisfies A) and B) for $\e=10^{-3}$.
Then there exists $j_0\in\N$ such that for all  $n\in\N$ there exists $\phi$ of type II
with  $\inde(\phi)\leq j_0$ such that $\vert \phi(z_{n})\vert \geq 0.9$.
\end{lemma}
\begin{proof}
Let $\phi\in D_{G}$ be such that $\phi(z_{2}+z_{n})\geq 1.998.$
It follows that  $\vert \phi (z_{2})\vert \geq 0.998$ and $\vert\phi(z_{n})\vert\geq 0.998$.

By B) we get that $\phi=\sum_{i=1}^{d}\lambda_{i}f_{i}$, where each $f_{i}$ is a
weighted functional i.e. is a result of an $(\mc{A}_{n_{j_i}}, m_{j_i}^{-1})$-operation
and weights of $(f_i)_{i=1}^d$ are different.
Let
$j_0\in\N$ be the  number we obtain from Lemma \ref{le7.1} for $z_{2}$. Setting
$A=\{i\leq d: w(f_{i})\leq m_{j_0}\}$ and $B$ its complement we get $\vert\sum_{i\in
B}\lambda_{i}f_{i}(z_2)\vert<0.001$. Therefore
$$
(\sum_{i\in A}\lambda_{i}^2)^{1/2}\geq
\vert\sum_{i\in A}\lambda_{i}f_{i}(z_{2})\vert\geq 0.997.
$$
It follows that
$$(\sum_{i\in B}\lambda_{i}^{2})^{1/2}\leq (1-0.997^2)^{1/2}\Rightarrow\|\sum_{i\in B}\lambda_{i}f_{i}\|\leq 0.09
$$
Hence $\vert \sum_{i\in A}\lambda_{i}f_{i}(z_{n})\vert\geq 0.9$.
\end{proof}
By the above lemma we get that
\begin{equation}\label{key}\mbox{for every $\phi$ of type II with $\inde(\phi)> j_0$ it holds that $\vert \phi(z_n)\vert<0.6$ for all $n\in\N$.}
\end{equation}
Indeed, assume that there exists $\phi_{2}$ of type II with $\inde(\phi_{2})>j_0$ and
$\vert\phi_{2}(z_{n})\vert\geq 0.6$. Then by the previous lemma we get $\phi_{1}$ of type
II with $\inde(\phi_{1})\leq j_0$ such that $\phi_1(z_{n})\geq 0.9$. It follows that
$\phi=\frac{1}{\sqrt{2}}(\phi_{1}+\phi_{2})\in D_{G}$ and hence
$\norm[z_{n}]\geq\phi(z_n)\geq \frac{0.9+0.6}{\sqrt{2}}>1$, a contradiction.

Consider now the vector
$u=\frac{1}{n_{j_0+1}}\sum_{i=1}^{n_{j_0+1}}z_{n_{j_0+1}+i}$.
Since we have  that $\xn[z]$ generates an $\ell_{1}$-spreading model with constant $0.999^{-1}$,
there exists  $\phi\in  D_{G}$
such  that $\phi(u)\geq 0.999$.
Since  $\norm[z_n]_{G}\leq 10^{-3}$ for all $n\in\N$  we
get
that $\phi=\sum_{i=1}^{d}\lambda_{i}f_{i}$ where each $f_{i}$ is a weighted functional and their weights are different.
Set
\begin{center}
$R_{1}=\{i: w(f_{i})\leq m_{j_0}\}$ and  $R_{2}=\{i: w(f_{i})>m_{j_0}\}.$
 \end{center}
 By
\eqref{key}  for every $n$ we obtain $\vert\sum_{i\in
R_{2}}\lambda_{i}f_{i}(z_n)\vert\leq 0.6$ and hence
\begin{equation}\label{key2}
\vert\sum_{i\in R_{2}}\lambda_{i}f_{i}(u)\vert\leq 0.6.
\end{equation}
On the other hand if $i\in R_{1}$    by Lemma \ref{l4.5} we get
$
\vert f_{i}(u)\vert\leq\frac{2}{w(f_{i})}.
$

Combining the above inequality with  \eqref{key2} we get
$$
0.999 \leq \vert\phi(u)\vert\leq \sum_{i\in
R_{1}}\frac{2\vert\lambda_{i}\vert}{w(f_{i})}+0.6<0.8,
$$
a contradiction.
\end{proof}
\begin{corollary}\label{c4.15} The space  $X_{\xi}$ does not admit  any $\ell_p$ (or $c_0$) as a spreading model.
\end{corollary}
The abstraction of the properties of the set $D_{G}$ enable us to derive also the following
\begin{theorem} There exists a reflexive Hereditarily Indecomposable Banach space $\mathfrak{X}_{HI}$ with no $\ell_{p}$, $1\leq p<\infty$, or $c_{0}$ as a spreading model.
\end{theorem}
\begin{proof}[Sketch of the proof]
First we shall define the norming set $K_{HI}$ of the space $\mathfrak{X}_{HI}$.  The set $K_{HI}$ is the smallest subset of $c_{00}(\N)$ satisfying the following conditions
\begin{enumerate}
\item  $K_{HI}$ is symmetric i.e. if $f\in K_{HI}$ then $-f\in K_{HI}$, it is closed under the restriction of its elements to intervals of $\N$ and $\{e_n^{*}: n\in\N\}\subset K_{HI}$.
\item $K_{HI}$ is closed under $(\mc{A}_{n_{2j}},m_{2j}^{-1})$-operations.
\item $K_{HI}$ is closed under $(\mc{A}_{n_{2j-1}},m_{2j-1}^{-1})$-operations on special sequences.
\item $K_{HI}$ is closed under the  operation $\sum_{i\in A}\lambda_{i}f_{i}$ whenever
\begin{enumerate}
\item  $f_{i}$ is the  result of an $(A_{n_{j_i}},
m_{j_i}^{-1})$-operation and $n_{j_i}\ne n_{j_{k}}$ for every
$i\ne k\in A$ \item $(\lambda_{i})_{i\in A}\in B_{\ell_2}\cap
[\Q]^{<\infty}$.
\end{enumerate}
\item $K_{HI}$ is rationally convex.
\end{enumerate}
To complete the definition of $K_{HI}$ we must define the special sequences. An
$n_{2j-1}-$special sequence $(f_{i})_{i=1}^{n_{2j-1}}$ is defined  as the
$n_{2j-1}$-attractor sequence  (Def. \ref{d7.1})  with the exception  that the functionals
$f_{2i}$  are  the results of an $(\mc{A}_{n_{2\sigma(f_{1},\dots,f_{2i-1})}},
m_{2\sigma(f_{1},\dots,f_{2i-1})}^{-1})$-operation of functionals of $K_{HI}$  instead of
$(e^{*}_{\lambda_{2i}})$ we use in the attractor sequence. Observe that as a ground set  we take the set $\{e_n^{*}:n\in\N\}$. The space $\mathfrak{X}_{HI}$
is the completion  of $(c_{00}(\N), \norm_{K_{HI}})$. In order to show that
$\mathfrak{X}_{HI}$ is HI space we follow the method initiated in \cite{GM} and extended
in  \cite{AT},\cite{AAT},\cite{ATod}. Namely first we observe that  Proposition
\ref{p4.6} holds for the space $\mathfrak{X}_{HI}$ hence there exist seminormalized
$\ell_{1}$ averages in every block subspace. Next we consider  rapidly increasing sequence (RIS) of $\ell_{1}$
averages, Def. \ref{d9.1} and we observe that Basic Inequality, Prop.\ref{bin},  holds
also for the space $\mathfrak{X}_{HI}$. It follows that the  estimations of Prop.
\ref{ris} holds for  RIS in the space $\mathfrak{X}_{HI}$. This enable us to  consider
the exact pairs, Def. \ref{d4.8}.   and also the  dependent  sequences. A $(C,
n_{2j-1})$-dependent  sequence $(x_k,f_{k})_{k=1}^{n_{2j-1}}$, $j\in\N$,  is  defined as
the attractor sequence, Def. \ref{d10.3}, with the exception that for all $k\leq
n_{2j-1}$, $(x_{k},f_{k})$ is an exact pair.  From the Basic Inequality and the
estimations on RIS  we obtain that estimations  of  Corollary \ref{depest} also holds for
the dependent sequences.  The estimations  in Corollary \ref{depest} easily yields that
$\mathfrak{X}_{HI}$ is indeed  HI space.

Since the ground set of the space $X_{HI}$  is  the set $G=\{\pm e_{n}^{*}:n\in\N\}$,
Theorem \ref{t11.3} yields that the space
$\mathfrak{X}_{HI}$ does  not admit any   $\ell_{p}$ or $c_{0}$ as a spreading model.
\end{proof}
\section{The $c_0$-index of the dual of $X_\xi$}\label{sec12}
We show now that every subspace
of  the dual space $X_\xi^{*}$ has $c_0$-index greater than $\omega^{\xi}$ and
does not contain a sequence generating a $c_0$-spreading model. As we have observe in Remark \ref{denseK}  the set $K_{\xi}$ is a norm-dense subset of $B_{X_{\xi}^{*}}$. So proving that every block subspace generated by  a block sequence of elements of $K_{\xi}$ has  $c_0$-index greater than $\omega^{\xi}$ we get that the same holds for all block subspaces of $X_{\xi}^*$.  Hence in the sequel we will  assume that the block subspaces are generated by block sequences of $K_{\xi}$.

We shall need  the dual result to Proposition \ref{p4.6}.
\begin{proposition}\label{p4.7}
For every $j\in\N$ every block subspace $Y$ of $X_{\xi}^{*}$ contains a
$2-c_{0}^{n_{2j}}$ average i.e. there exists a block sequence
$(x_{i}^{*})_{i=1}^{n_{2j}}$ in $Y$ such that $\norm[x_{i}^{*}]\geq 2^{-1}$ for every
$i\leq n_{2j}$ and
$
2^{-1}\leq \|\sum_{i=1}^{n_{2j}}x_{i}^{*}\|\leq 1.
$
\end{proposition}
For the proof see   Lemma 5.4 in \cite{AT}.
\begin{proposition}\label{p413}
 a) The  $c_0$-index of every subspace $Y$ of $X_\xi^{*}$ is greater than $\omega^{\xi}$.

b) The dual space $X_\xi^{*}$ of $X_\xi$ does not contain a normalized basic sequence
generating a $c_0$-spreading model.
\end{proposition}
\begin{proof}
a) It is enough to prove the result for the block subspaces. Let
$Y$ be a block subspace of $X_\xi^{*}$. By Proposition \ref{p4.7}
for every $j\in\N$ we choose $x_{i}^{*}$, $i\in\N$  such that
$x^{*}_{i}$ is a $2-c_{0}^{n_{2j_i}}$ average. Let
$x^*_{i}=\sum_{t=1}^{n_{2j_i}}x_{i,t}^{*}$ and
$x_{i}=\sum_{t=1}^{n_{2j_i}}\frac{x_{i,t}}{n_{2j_i}}$ where
$x^{*}_{i,t}(x_{i,t})\geq 2^{-1}$ and $\norm[x_{i,t}]=1$.

Let $1<j\in\N$. Passing to a subsequence we may assume that
$(x_{i})_{i=1}^{n_{2j}}$ is a $(3,1/n_{2j})-$RIS and by Proposition \ref{ris-sss} for every
$g\in G_{\xi}$ holds
$$
\#\{n\in\N: \vert g(x_n)\vert\geq \frac{2}{m_{2j}^2}\}\leq
20m_{2j}^{4}\leq n_{2j-1}
$$
Set $\breve{y}=\frac{m_{2j}}{n_{2j}}\sum_{i=1}^{n_{2j}}x_{i}$ and
$y^{*}=m_{2j}^{-1}\sum_{i=1}^{n_{2j}}x_{i}^{*}$.

Then setting $y:=\lambda \breve{y}$ for some $\lambda\in [1,2]$,
Proposition \ref{ris} yields that $(y,y^{*})$ is a
$(6,2j)-$exact pair.

Note that for every  $g\in G_{\xi}$, setting $A=\{i:\vert
g(x_i)\vert\geq \frac{2}{m_{2j}^{2}}\}$ we get
\begin{equation}\label{6}
\vert g(y)\vert
\leq
\frac{2m_{2j}}{n_{2j}} ( \sum_{i\in
A}\vert g(x_i)\vert+\sum_{i\not\in A}\vert g(x_i)\vert)
\leq
\frac{2m_{2j}}{n_{2j}} (n_{2j-1}+\frac{2n_{2j}}{m_{2j}^{2}})\leq
\frac{6}{m_{2j}}
\end{equation}
It follows that for every $j\in\N$ and  every block subspace we have a $(6,2j)$-exact
pair $(y,y^*)$ with $y^*\in Y$ and $\norm[y]_{G_{\xi}}\leq 6/m_{2j}$.

Therefore we are able to construct for every $j\in\N$ a $n_{2j-1}$-attractor sequence
$(y_{i},y_{i}^{*})_{i=1}^{n_{2j-1}}$ with $(y_{2i-1},y^{*}_{2i-1})$ - a
$(6,2j_{2i-1})$-exact pair, $y_{2i-1}^*\in Y^*$ and
$\norm[y_{2i-1}]_{G_{\xi}}<m_{2j-1}^{-2}$ for every $i$. From Corollary  \ref{depest} we
get
$$
\frac{1}{2m_{2j-1}^{2}}\leq
\|\frac{1}{n_{2j-1}}\sum_{i=1}^{n_{2j-1}}(-1)^iy_{i}\|\leq
\frac{90}{m_{2j-1}^{2}}.
$$
Setting
$z_{j}^{*}=m_{2j-1}^{-2}\sum_{i=1}^{n_{2j-1}/2}y_{2i-1}^{*}$ and
$w_{j}^{*}=-m_{2j-1}^{-2}\sum_{i=1}^{n_{2j-1}/2}y_{2i}^{*}$ we get
that
\begin{equation}\label{dd1}
\norm[z_{j}^{*}],\norm[w_{j}^{*}]\geq
\frac{1}{180},\,\,\,\textrm{and}\,\,\,\norm[z^{*}_{j}-w^{*}_{j}]\leq
\frac{1}{m_{2j-1}}.
\end{equation}
Let $F\in\mc{S}_{\xi}$ and $(w_{j}^{*})_{j\in F}$ be a $G_{\xi}-$special
sequence such that $w_j^*$ satisfies \eqref{dd1} for an
appropriate $z_{j}^{*}$ and $\minsupp w_j^*>j$ for every $j\in F$.

From the definition of the ground set $G_{\xi}$ we have
$
\sum_{j\in F}\e_{j}w_{j}^{*}\in G_{\xi}\subset B_{X_{\xi}^{*}}
$
and hence
\begin{equation*}\|\sum_{j\in F}\e_{j}z_{j}^{*}\|\leq
\|\sum_{j\in F}\e_{j}w_{j}^{*}\|+\|\sum_{j\in
F}\e_{j}z_{j}^{*}-\e_{j}w_{j}^{*}\| \leq 2.
\end{equation*}
Notice that the tree consisting of sequences $(z_j^*)_{j\in F}$, with $F\in S_\xi$,
obtained in the way described above has order greater or equal $o(\mc{S}_\xi)$.

b) Notice that if there is a normalized basic sequence in $X_\xi^*$ generating a
$c_0$-spreading model, then we get also a block sequence $(f_n)$ generating a
$c_0$-spreading model, and then a normalized block sequence $(x_n)\subset X_\xi$  with
$f_i(x_j)=\delta_{i,j}$ generates an $\ell_{1}$-spreading model in $X_{\xi}$, a
contradiction with Corollary \ref{c4.15}.
\end{proof}
\section{The  space $\mathfrak{X}_{\xi}$}\label{sec13}
In this section we define our final space $\mathfrak{X}_{\xi}$ as a suitable quotient
$X_\xi/X_L$ and show the desired properties of $\mathfrak{X}_{\xi}$ (Corollaries
\ref{last} and \ref{clast}). In order to this we show for any subspace $Y\subset X_\xi$
with $Y/X_L$ infinite dimensional the existence of an $\ell_{1}$ average and a norming
functional with controlled behavior with respect to $Y/X_L$ and $X_L^\perp$ (Lemma
\ref{lf}) and consequently the existence of a suitable attracting sequence (Corollary
\ref{c5.3}).

For an  infinite subset $L$ of $\N$ we set $X_{L}=\overline{\langle(e_{n})_{n\in
L}\rangle}$, where $(e_n)_{n\in\N}$ is the basis of $X_{\xi}$. We shall  prove that that
for every $L\in [\N]$ with $\N\setminus L$ also infinite and  satisfying $\#(L\cap
\Lambda_{i})=\infty$ for every $i\in\N$, the  quotient space
$\mathfrak{X}_{\xi}=X_{\xi}/X_{L}$, does not contain a (normalized) sequence generating
an $\ell_{1}$-spreading model but every of its subspaces has  $\ell_{1}$-index greater
than $\omega^{\xi}$. Let $Q:X_{\xi}\to \mathfrak{X}_{\xi}$ denotes the quotient map.

Let $\N=L\cup M$, where $M\cap L=\emptyset$. First we prove the following
\begin{proposition}\label{p13.1}
The sequence  $(Q(e_{n}))_{n\in M}$ is a basis for the quotient $\mathfrak{X}_{\xi}$.
\end{proposition}
\begin{proof}For every
$x=\sum_{n \in\N}a_ne_n\in X_{\xi}$ we get
$Q(x)=\sum_{n\in M}a_nQ(e_n)$.
Let  $M=(m_i)$ and $j<n$. By the bimonotonicity of $(e_n)_{n\in\N}$ we have  that
$\norm[\sum_{i=1}^{n}a_{m_i}Q(e_{m_i})]=\norm[\sum_{i=m_1}^{m_n}a_{i}e_{i}]$ for some
$a_{i}, i\in [m_1,m_{n}]$. Using again that $(e_n)_{n\in\N}$  is bimonotone, for every
$j<n$ we obtain $$ \|\sum_{i=m_1}^{m_j}a_{i}e_{i}\| \leq
\|\sum_{i=m_1}^{m_n}a_{i}e_{i}\|= \|\sum_{i=1}^{n}a_{m_i}Q(e_{m_i})\|$$ and
therefore
$
\|\sum_{i=1}^{j}a_{m_i}Q(e_{m_i})\|\leq
\|\sum_{i=m_1}^{m_j}a_i{e}_{i}\| \leq
\|\sum_{i=1}^{n}a_{m_i}Q(e_{m_i})\|.
$
\end{proof}
We shall need the following  result which  is analogous to Lemma 11 \cite{F3}.
\begin{lemma}\label{lf}
Let $N,m\in\N$, $\e\in(0,1/4)$ and $Y$ subspace of $X_{\xi}$ such that  the quotient $Y/X_{L}$ is infinite dimensional.
Then there exist a $2-\ell_{1}^{N}$ average $x\in\{e_{i}:i\geq m\}$ and $f\in K_{\xi}$ such that
\begin{center}
\mbox{$f(x)\geq 1/2,\,\,\,\minsupp f\geq\minsupp x$,
$
\dist(Q(x), Y/X_{L})<\e$ and $\dist(f,
X_{L}^{\perp})<\e.$}
\end{center}
\end{lemma}
\begin{proof}
Let $(\e_{n})_{n\in\N}$ be positive numbers with $\sum_{n}\e_{n}<\e$.
Let $\hyn$ be a normalized $w-$null sequence in $Y/X_{L}$. Passing
to  a subsequence we may assume that  $\hyn$ is $(\e_n)$-close to
a normalized block sequence $\hxn\in \mathfrak{X}_{\xi}$  and there
exists $x_{n}\in X_{\xi}$ such that
$$
\mbox{ $Q(x_n)=\widehat{x}_n$, $\norm[x_n]=\norm[\widehat{x}_{n}]$
and
$\ran(x_n)=[\minsupp(\widehat{x}_{n}),\maxsupp(\widehat{x}_{n})]$}.
$$
Let also $y_{n}^{*}\in \mathfrak{X}_{\xi}^{*}=X_{L}^{\perp}$ such
that $y_{n}^{*}(x_{n})=1=\norm[\widehat{x}_{n}]$,
$\norm[y_n^{*}]=1$ for every  $n\in\N$. Since
$(e_n^{*})_{n\in\N}$ is bimonotone basis of $X_{\xi}^{*}$, setting
$x_{n}^{*}=E_{n}y_{n}^{*}$, where $E_{n}=\ran(x_n)$, we get that
$(x_{n}^{*})_{n\in\N}$ is a block sequence
$$
\mbox{$\norm[x_n^{*}]\leq \norm[y_n^{*}],\,\,\,\,x_{n}^{*}(x_{n})=y_{n}^{*}(x_n)$
and  $x_{n}^{*}\in X_{L}^{\perp}$}
$$
Let $k,j\in\N$ be such that $2^k>m^2_{2j}$ and $N^k\leq n_{2j}$. Such $j,k$ exist, since
by definition $n_{2j}=(2n_{2j-1})^{s_{2j-1}}$ and $2^{s_{2j-1}}=m_{2j}^{3}$. Hence if
$N\leq m_{2j-1}$ setting $k=s_{2j-1}$ we get  $N^k\leq n_{2j}$ and  $2^{-k}<
m^{-2}_{2j}$. We set
$$
A_{1}=\left\{L\in [\N]: L=(l_{i})_{i\in\N}:
\|\frac{1}{N}\sum_{i=1}^{N}\widehat{x}_{l_{i}}\|_{\mathfrak{X}_\xi}>1/2\right\}.
$$
By Ramsey theorem we may find an  $L\in[\N]$ such that either $[L]\subset A_{1}$ or $[L]\cap A_1=\emptyset$.

Assume first that $[L]\subset A_{1}$. We may assume that $x_{\min L}\geq m$.
Since  $\widehat{x}_{n}=Q(x_n)$ for every $n\in\N$ it follows that
$$
\|\frac{1}{N}\sum_{i=1}^{N}x_{l_{i}}\|\geq
\|\frac{1}{N}\sum_{i=1}^{N}\widehat{x}_{l_{i}}\|_{\mathfrak{X}_\xi}> 1/2.
$$
Let $x=\frac{1}{N}\sum_{i=1}^{N}x_{l_{i}}$. Take $g\in
\mathfrak{X}_{\xi}^{*}=X_{L}^{\perp}$ such that
$g(\frac{1}{N}\sum_{i=1}^{N}x_{l_{i}})>1/2$. As before we may
assume that  $\minsupp x=\minsupp g$.

Setting  $\widehat{x}=Q(\frac{1}{N}\sum_{i=1}^{N}x_{l_{i}})$ we get
$$
\dist(\widehat{x}, Y/X_{L})\leq
\|\frac{1}{N}\sum_{i=1}^{N}(\widehat{x}_{l_{i}}-\widehat{y}_{l_i})\|_{\mathfrak{X}_{\xi}}\leq
\frac{1}{N}\sum_{i=1}^{N}\norm[\widehat{x}_{l_{i}}-\widehat{y}_{l_i}]_{\mathfrak{X}_{\xi}}<\sum_{i}\e_{l_i}\leq\e.
$$
Since  $B_{X_{\xi}^{*}}=\overline{K_{\xi}}^{w^{*}}$ we choose $f\in K_{\xi}$ with
$$
f(x)\geq 1/2,\,\,\minsupp f\geq\minsupp x\,\,\mbox{and}\,\,\norm[f-g]<\e.
$$
This is possible by taking an interval $B$ such that $\norm[g-Bg]<\e/2$ and using that
$\minsupp x=\minsupp g$ and $B_{X_{\xi}^{*}}=\overline{K_{\xi}}^{w^{*}}$.

Assume now that  $[L]\cap A_{1}=\emptyset$. Set $\widehat{y}_{n}^{(1)}=\frac{1}{N}\sum_{i\in F_n}\widehat{x}_{l_{i}}$
where  $\# F_{i}=N$, $F_{n}<F_{n+1}$ for every $n$  and $\cup_{i}F_{i}=L$.

Passing to a  subsequence we may assume that  there exists $a_1\geq 2$ such that
$\widehat{x}_{n}^{(2)}=a_{1}\frac{1}{N}\sum_{i\in F_n}\widehat{x}_{i}^{(1)}$
is a normalized  sequence in $Y/Z$.
We may again apply Ramsey theorem defining $A_{2}$ as before. If we get
some  $L$ with $[L]\subset A_{2}$ the proof finishes as before.

Assume that  in none of the first $k$ steps we  get  $L\in[\N]$ with $[L]\subset A_{k}$.
Then there exist $a_1,a_2,\ldots,a_k\ge
 2$ and $l_1<l_2<\cdots<l_{N^k}$ in $\N$ such that the vector
$$
\widehat{y}=a_1a_2\cdot\ldots\cdot
 a_k\frac{1}{N^{k}}\sum\limits_{i=1}^{N^{k}}\widehat{x}_{l_i}
$$
satisfies $\norm[\widehat{y}]_{\mathfrak{X}_\xi}=1$.

The functional $y^{*}=\frac{1}{m_{2j}}\sum\limits_{i=1}^{N^{k}}x_{l_i}^{*}\in
X_{L}^{\perp}$ satisfies $\norm[y^{*}]\le 1$. Therefore
$$
1=\norm[\widehat{y}]_{\mathfrak{X}_\xi}\geq y^{*}(\widehat{y}) =\frac{2^kN^k}{m_{2j}N^k}
=\frac{2^k}{m_{2j}}
$$
which contradicts  our choice of $k$ and $j$.
\end{proof}
\begin{corollary}\label{cf}
Let $Y/X_{L}$ be an infinitely dimensional subspace of the quotient $X_{\xi}/X_L$. For
every $j\in\N$, $\e>0$ there exists a $(6,2j)$-exact pair $(y,f)$ such that
\begin{enumerate}
 \item $\dist(Q(y),Y/X_L)<\e$ and $\dist(f, X_{L}^{\perp})<\e$,
\item $\norm[y]_{G_{\xi}}<6m_{2j}^{-1}$.
\end{enumerate}
\end{corollary}
\begin{proof}
Let $j\in\N$ and $(\e_{i})_{i}\subset (0,1)$ such that $\sum_{i}\e_{i}<\e\leq n_{2j}^{-1}$.
Using Lemma \ref{lf} inductively we choose $(z_{i},f_{i})_{i=1}^{n_{2j}}\in (c_{00}(\N)\times K_{\xi})$
such that
\begin{enumerate}
\item[a)] $z_{i}$ is a  $2-\ell_{1}^{n_{2j_i}}$ average and $(z_{i})_{i=1}^{n_{2j}}$ is a
$(3,\e)$-RIS.
 \item[b)]  $\dist(Q(z_{i}), Y/X_L)<\e_i$ and  $\dist(f_i,X_{L}^{\perp})<\e_i$ for every $i$.
\item[c)] $f_i(z_i)\geq 1/2$.
\item[d)] $\ran(z_{i})\cup\ran(f_{i})<\ran (z_{i+1})\cup\ran(f_{i+1})$ for every $i$.
\end{enumerate}
By Proposition \ref{ris-sss} we can assume that
for every  $g\in G_{\xi}$ the  set $\{i\in\N:\ \vert g(z_i)\vert\geq2m_{j}^{-2}\}$  has
cardinality at most $n_{j-1}$. Setting
$z=\frac{m_{2j}}{n_{2j}}\sum_{i=1}^{n_{2j}}z_{i}$ and
$f=m_{2j}^{-1}\sum_{i=1}^{n_{2j}}f_{i}$ we have $f(z)\geq 1/2$.
Taking $y:=\lambda z$ for some $\lambda\in [1,2]$ we get from
Proposition \ref{ris}  that  $(y,f)$ is a $(6,2j)$-exact pair.
It is easy to see that the exact pair $(y,f)$ satisfies the
requirements $(1)$.
The argument of \eqref{6} yields that (2) also holds.
\end{proof}
\begin{corollary}\label{c5.3}
For every  $j\in\N$, any $Y/X_L$ infinitely dimensional subspace of the quotient
$X_\xi/X_L$ and any $\e>0$ there exists a $(6,2j-1)$-attracting sequence
$(y_{i},f_{i})_{i=1}^{n_{2j-1}}$ with the associated sequence $(j_{i})_{i}$ such that
\begin{enumerate}
\item[A)]  $\supp f_{2i}=\supp y_{2i}\subset L$,
\item[B)] For every $i\leq n_{2j-1}$, $(y_{2i-1},f_{2i-1})$ is a
$(6,j_{2i-1})$-exact pair and
$\norm[y_{2i-1}]_{G_{\xi}}<m_{2j-1}^{-2}$,
\item[C)] $\dist (Q(\sum_iy_{2i-1}),Y/X_L)<\e$,
\item[D)] There exists  $g\in X_{L}^{\perp} $ with
$$
\|g-\frac{1}{m_{2j-1}^{2}}\sum_{i=1}^{n_{2j-1}/2}f_{2i-1}\|\leq
m_{2j-1}^{-2}\,\,\,\textrm{and}\,\,\,
\|g-\frac{1}{m_{2j-1}^{2}}\sum_{i=1}^{n_{2j-1}/2}f_{2i}\|\leq
2m_{2j-1}^{-1}.
$$
\end{enumerate}
\end{corollary}
\begin{proof}
Using Corollary \ref{cf} and assumption on the set $L$ we can
choose an attracting sequence $(y_{i},f_{i})_{i=1}^{n_{2j-1}}$
such that $A)$ holds and for every $i\leq n_{2j-1}/2$ the couple
$(y_{2i-1},f_{2i-1})$ is a $(6,j_{2i-1})-$exact pair satisfying the conclusion of Corollary
\ref{cf} for $\e_i=\e2^{-i}$, i.e.
\begin{enumerate}
 \item $\dist(Q(y_{2i-1}),Y/X_L)<\e_i$ and $\dist(f_{2i-1}, X_{L}^{\perp})<\e_i$
\item $\norm[y_{2i-1}]_{G_{\xi}}<6m_{j_{2i-1}}^{-1}<m_{2j-1}^{-2}$.
\end{enumerate}
For every $i\leq n_{2j-1}/2$ choose $g_{2i-1}\in
X_{L}^{\perp}$ with $\norm[f_{2i-1}-g_{2i-1}]<\e2^{-i}$ and let
$g=m_{2j-1}^{-2}\sum_{i=1}^{n_{2j-1}/2}g_{2i-1}$.

By Corollaries  \ref{depest}, \ref{c412} the chosen vectors and functionals satisfy the desired conditions.
\end{proof}
\begin{corollary}\label{last} Every subspace of the quotient $\mathfrak{X}_\xi$ has   $\ell_{1}$-index greater than $\omega^{\xi}$.
 \end{corollary}
\begin{proof}
By Corollary \ref{c5.3}  we  can choose for every subspace $Y/X_L$ of the quotient
$\mathfrak{X}_{\xi}$ and any $j$ an attracting sequence
$(y^{j}_{l},f^{j}_{l})_{l=1}^{n_{2j-1}}$ and $g_{j}\in X_L^{\perp}$ such that for all
$l$, $\norm[y_l^j]_G<m_{2j-1}^{-2}$ and for
$$
z^{*}_{j}=m_{2j-1}^{-2}\sum_{l=1}^{n_{2j-1}/2}f^{j}_{2l}\in G_1,\
\ \ \
f_j=m_{2j-1}^{-2}\sum_{l=1}^{n_{2j-1}/2}f_{2l-1}^{j}, \ \
\ \
u_j=\frac{m_{2j-1}^{2}}{n_{2j-1}}\sum_{l=1}^{n_{2j-1}/2}y^{j}_{2l-1}
$$
we have
$$
\norm[g_{j}-z_{j}^{*}]<2m_{2j-1}^{-1}\,\,\textrm{and}\,\,\,
\norm[g_{j}-f_j]\leq m_{2j-1}^{-2}\,\,\textrm{and}\,\,\, \dist
(Q(u_j),Y/X_L)< 16^{-j}.
$$
In particular we get $\norm[g_j]\leq 2$. Setting
$\widehat{u}_{j}=Q(u_j) =
Q(\frac{m_{2j-1}^{2}}{n_{2j-1}}\sum_{l=1}^{n_{2j-1}}(-1)^{j+1}y^{j}_{l})
$ we get by Corollary \ref{depest}
\begin{align*}
\frac{1}{8}\leq\frac{1}{4}-m_{2j-1}^{-2}/2\leq f_{j}(u_{j})/2-\norm[f_j-g_j]/2 &\leq
g_{j}(u_{j})/2\leq\norm[\widehat{u}_{j}]_{\mathfrak{X}_\xi}
\\
&\leq
\|\frac{m_{2j-1}^{2}}{n_{2j-1}}\sum\limits_{l=1}^{n_{2j-1}}(-1)^{j+1}y^{j}_{l}\|\leq
90.
\end{align*}
For any $F\in \mc{S}_\xi$ we pick $(z_j^*)_{j\in F}$, $(g_j)_{j\in
F}$, $(u_j)_{j\in F}$ so that
\begin{enumerate}
\item $z_j^*$, $g_j$, $u_j$ are picked as above for any $j\in F$,
\item $\minsupp (z_j)>j$ for any $j\in F$,
\item $(z_j^*)_{j\in F}$ is a $G_{\xi}$-special sequence.
\end{enumerate}
Then we have that
$
\|\sum_{j\in F}\e_{j}z_{j}^{*}\|\leq 1\,\,\,\textrm{for every
$\e_{j}\in\{-1,1\}$}.
$
Since  $\norm[g_j-z^{*}_j]\leq 2^{-j}$ for every $j\in F$ it follows that
$\norm[\sum_{j\in F}\e_{j}g_{j}]\leq 2$ for every $\e_{j}=\pm 1$. Therefore
$$
\|\sum_{j\in F}a_j\widehat{u}_{j}\|_{\mathfrak{X}_\xi}\geq\frac{1}{2} \sum_{j\in
F}\textrm{sgn}(a_j)g_{j}(\sum_{j\in F}a_ju_{j})\geq \frac{1}{2}\sum_{j}\frac{1}{8}\vert
a_j\vert\geq \frac{1}{16}\sum\vert a_j\vert.
$$
Since $\dist (\widehat{u}_j,Y/X_L)< 16^{-j}$ for any $j$, by the above procedure we can
obtain an $\ell_{1}$-tree in $Y/X_L$ with order greater or equal $o(\mc{S}_\xi)$, hence the
 $\ell_{1}$-index of $Y/X_L$ is greater than $\omega^{\xi}$.
\end{proof}
\begin{corollary}\label{clast}
The space $\mathfrak{X}_{\xi}=X_{\xi}/X_{L}$ does not contain a normalized basic
sequence generating an $\ell_{1}$-spreading model.
\end{corollary}
\begin{proof}
If a normalized basic sequence $\hxn\subset \mathfrak{X}_{\xi}$ generates an $\ell_{1}$-spreading model  in
$\mathfrak{X}_{\xi}$, we may assume that it is a block sequence. Then we get a block
sequence in $X_{\xi}$ which generates an $\ell_{1}$-spreading model, a contradiction by
Corollary \ref{c4.15}.
\end{proof}
Gathering the results of this and previous sections  we get the following
\begin{theorem} For every countable  ordinal $\xi$ there exists a separable   reflexive Banach space $\mathfrak{X}_{\xi}$ with the hereditary Bourgain $\ell_{1}$-index greater than $\omega^{\xi}$ such that $\mathfrak{X}_{\xi}$ does not admit an $\ell_{1}$-spreading model. Moreover the dual  $\mathfrak{X}_{\xi}^{*}$ has  hereditary  $c_0$-index greater than $\omega^{\xi}$ and does not admit a $c_0$-spreading model.
\end{theorem}
\begin{remark}
The results presented above concerning indices of the dual space $X^*_\xi$ and the
quotient space $\mathfrak{X}_{\xi}$ suggest the following problem:

Assume that $X$ is a reflexive  Banach space such that every subspace of the dual has
 $c_0$-index greater than  $\omega^{\xi}$. Does there exist a subspace $Y$ such
that  every subspace of the quotient has  $X/Y$ has  Bourgain $\ell_{1}$-index greater than
$\omega^{\xi}$?
\end{remark}


\begin{thebibliography}{99999}
\bibitem{Kal}  F. Albiac and N.J. Kalton, {\sl Topics in Banach space theory}, Graduate Texts in Mathematics 233, Springer, New York, 2006.
\bibitem{AA} D. Alspach and S.A. Argyros, {\sl Complexity of weakly null
sequences}, Diss. Math. 321 (1992), 1--44.
\bibitem{AAT} S.A. Argyros, A. D. Arvanitakis and A. Tolias, {\sl Saturated
extensions, the attractors method and Hereditarily James Tree
Spaces}, London Math. Soc. Lecture Note Ser., 337, Cambridge Univ. Press,
Cambridge, 2006.
\bibitem{AG} S.A. Argyros and I.Gasparis, {\sl Unconditional structures of
weakly null sequences}, Trans. Amer. Math. Soc. 353(5), 2019-2058, 2001.
\bibitem{AK}
S.A. Argyros and V. Kanellopoulos, {\sl Determining $c_0$ in  $C(K)$ spaces},
Fund. Math. 187 (2005), 61--93.
\bibitem{AMT} S. A. Argyros, S. Mercourakis and A. Tsarpalias, {\sl Convex
unconditionality and summability of weakly null sequences}, Israel J. Math. 107
(1998), 157--193.
\bibitem{ATod} S.A. Argyros and S.Todorcevic, {\sl Ramsey methods in analysis. Advanced Courses in Mathematics}. CRM Barcelona. Birkh\"auser Verlag, Basel, 2005.
\bibitem{AT} S.A. Argyros and A. Tolias, {\sl Indecomposability and
unconditionality in duality}, Geom. Funct. Anal. 14 (2004), no. 2, 247--282.
\bibitem{ATo} S.A. Argyros and A. Tolias, {\sl Methods in the theory of
hereditarily indecomposable Banach spaces},
Mem. Amer. Math. Soc. 170 (2004), no. 806.
\bibitem{B} J. Bourgain, {\sl On convergent sequences of continuous functions},
Bull. Soc. Math. Belg. S\'er. B 32 (1980), 235--249.
\bibitem{BS} A. Brunel and L.Sucheston, {\sl On B-convex Banach spaces}, Math.
Systems Theory, 7 (1974), no. 4, 294-299.
\bibitem{EM} P. Erd\"os and M. Magidor, {\sl A note on regular methods of
summability and the Banach-Saks property}, Proc. Amer. Math. Soc.,  59 (1976), no. 2,
232--234.
\bibitem{F3} V. Ferenczi, {\sl Quotient hereditarily indecomposable Banach
spaces}, Canadian J. Math. 51 (1999), no.3, 566--584.
\bibitem{GA} I. Gasparis, {\sl A dichotomy theorem for subsets of the power set
of the natural numbers}, Proc. Amer. Math. Soc. 129 (2001), no. 3, 759--764.
\bibitem{G} W.T. Gowers {\sl An infinite Ramsey theorem and some Banach-space
dichotomies},  Ann. of Math. (2), 156 (2002), no. 3, 797-833.
\bibitem{GM}W.T. Gowers and B. Maurey, {\sl The unconditional basic sequence
problem}, J. Amer. Math. Soc., 6 (1993), no.4, 851--874.
\bibitem{Ja}
R.C. James,  {\sl A separable somewhat reflexive space with nonseparable dual},
Bull. Amer. Math. Soc.,  80 (1974), 738--743.
\bibitem{JO} R. Judd and E. Odell, {\sl Concerning Bourgain's $\ell_{1}$-index of
a Banach space}, Israel J. Math. 108 (1998), 145--171.
\bibitem{K} J.L. Krivine, {\sl Sous espaces de dimesion finie des espaces de
Banach r\'eticul\'es}, Ann. of Math. (2) 104 (1976), 1--29.
\bibitem{LT} J. Lindenstrauss and L. Tzafriri,
{\sl Classical Banach spaces I}, Springer-Verlag 92, 1977.
\bibitem{MMT} B. Maurey, V.D. Milman and N. Tomczak-Jaegermann, {\sl Asymptotic
Infinite-Dimensional theory of Banach spaces}, Geometric aspects of functional analysis (Israel, 1992--1994),  149--175, Oper. Theory Adv. Appl., 77, Birkh\"auser, Basel, 1995.
\bibitem{MR} B. Maurey  and H. Rosenthal, {\sl Normalized weakly null sequence with no unconditionall subsequence}, Studia Math. 61 (1977), 77--98.
\bibitem{MT} V. Milman and N. Tomczak-Jaegermann, {\sl Asymptotic $\ell_p$
spaces and bounded distortion}, Banach spaces (Merida
1992), 173--195, Comtemp. Math. 144, Amer. Math. Soc., Providence, RI, 1993.
\bibitem{OSsm} E. Odell and Th. Schlumprecht, {\sl On the richness of the set of
p's in Krivine's theorem},  Geometric aspects of functional analysis (Israel, 1992--1994),  177--198, Oper. Theory Adv. Appl., 77, Birkh\"auser, Basel, 1995.
\bibitem{P} A.M. Pelczar, {\sl Note on distortion and Bourgain $\ell_{1}$-index},
Studia Math. 190 (2009), no. 2, 147--161.
\bibitem{PS} A. Pe{\l}czy\'nski and Z. Semadeni, {\sl Spaces of continuous
functions III, Spaces $C(\Omega)$ for $\Omega$ without perfect subsets}, Studia
Math. 18 (1959), 211--222.
\bibitem{Sh} {\rm Th. Schlumprecht}, {\sl An arbitrarily distortable Banach
space}, Israel J. Math. 76 (1991), 81--95.
\bibitem{T} B. S. Tsirelson, {\sl Not every Banach space contains  $\ell _{p}$
or $c_{0}$}, Funct. Anal. Appl.  8 (1974), 138--141.
\end{thebibliography}
\end{document}